%% file: SBD-gamma-convergence_final.tex
\documentclass[reqno,10pt]{amsart}
\usepackage{ esint }

\usepackage{overpic}
\usepackage{amssymb}
\usepackage{amsmath}
\usepackage{amsthm}
\usepackage{mathrsfs}
 \usepackage{tikz}
\usetikzlibrary{shapes,snakes,spy,calc,shapes.geometric}

\usepackage{pgf}
\usepackage{color}
\usepackage{graphicx}   
\usepackage{hyperref}

\usepackage{amsmath}
  \usepackage{paralist}
\usepackage{graphicx}  
\usepackage{psfrag}
\usepackage{comment}
\usepackage{esvect}

\newcommand{\eop}{\nopagebreak\hspace*{\fill}$\Box$\smallskip}

\newtheorem{theorem}{Theorem}[section]
\newtheorem{corollary}[theorem]{Corollary}

\newtheorem{lemma}[theorem]{Lemma}
\newtheorem{proposition}[theorem]{Proposition}

\theoremstyle{definition}
\newtheorem{definition}[theorem]{Definition}
\newtheorem{remark}[theorem]{Remark}

\numberwithin{equation}{section}

\newcommand{\R}{\mathbb{R}}
\newcommand{\N}{\mathbb{N}}

\newcommand{\eps}{\varepsilon}

\newcommand{\mres}{\mathbin{\vrule height 1.6ex depth 0pt width
0.13ex\vrule height 0.13ex depth 0pt width 1.3ex}}

\newcommand{\EEE}{\color{black}} 
\newcommand{\RRR}{\color{black}}

\usepackage{latexsym,amsfonts}
\usepackage{mathrsfs}
\usepackage{centernot}
\usepackage{paralist}

\usepackage{eso-pic}  
\usepackage{pifont}
\usepackage{fixmath}

\numberwithin{equation}{section}

\begin{document}

\title[$\Gamma$-convergence  for free-discontinuity problems in linear elasticity]{$\Gamma$-convergence  for free-discontinuity problems in linear elasticity: Homogenization and relaxation}



\author{Manuel Friedrich}
\address[Manuel Friedrich]{Applied Mathematics,  
Universit\"{a}t M\"{u}nster, Einsteinstr. 62, D-48149 M\"{u}nster, Germany}
\email{manuel.friedrich@uni-muenster.de}
\urladdr{https://www.uni-muenster.de/AMM/Friedrich/index.shtml}

\author{Matteo Perugini}
\address[Matteo Perugini]{Applied Mathematics,  
Universit\"{a}t M\"{u}nster, Einsteinstr. 62, D-48149 M\"{u}nster, Germany}
\email{matteo.perugini@uni-muenster.de}

\author{Francesco Solombrino}
\address[Francesco Solombrino]{Dip. Mat. Appl. ``Renato Caccioppoli'', Univ. Napoli ``Federico II'', Via Cintia, Monte S. Angelo
80126 Napoli, Italy}
\email{francesco.solombrino@unina.it}
\urladdr{http://www.docenti.unina.it/francesco.solombrino}

\begin{abstract}

We analyze the $\Gamma$-convergence of sequences of free-discontinuity functionals  arising  in the modeling of linear elastic solids with surface discontinuities, including phenomena as fracture, damage,   or material voids. We prove compactness with respect to $\Gamma$-convergence and represent the $\Gamma$-limit in an integral form defined on the space of generalized special functions of bounded deformation ($GSBD^p$). We identify the integrands in terms of asymptotic cell formulas and prove a non-interaction property between bulk and surface contributions. Eventually, we investigate sequences of corresponding boundary value problems and show convergence of minimum values and minimizers. In particular, our techniques allow to characterize relaxations of functionals on $GSBD^p$, and cover the classical case of periodic homogenization.   
\end{abstract}
\maketitle

\section{Introduction}\label{sec: introduction}
This paper deals with the $\Gamma$-convergence  of  sequences of free-discontinuity functionals  $(\mathcal{E}_n)_n$  of
the form 
\begin{align}\label{eq: intro basic energy}
\mathcal{E}_n(u)  = \int_{\Omega} f_n\big(x, e(u)(x) \big) \, {\rm d}x +  \int_{J_u \cap \Omega} g_n\big(x,  u^+(x), u^-(x),  \nu_u(x)\big) \, {\rm d} \mathcal{H}^{d-1} (x), 
\end{align}
 where $\Omega \subset \R^d$ denotes the reference configuration,  $e(u)$ is  the symmetric part of the gradient of a {\it vector-valued} displacement $u\colon \Omega \to \mathbb{R}^d$,  and  $J_u$ denotes the set of discontinuities of $u$, oriented by a normal vector $\nu_u$ with  one-sided traces  $u^+$ and $u^-$.   (By   $\mathcal{H}^{d-1}$ we indicate the $(d-1)$-dimensional Hausdorff measure.)  Such functionals are  prototypes for  many variational models of fracture mechanics in a small strain setting. In this framework,  the bulk  density $f_n$ depending only  on the linearized stress tensor  accounts for elastic bulk terms for the unfractured region of the body,  while the surface  integrand  $g_n$  represents the energy spent to produce a crack.  Usually, $g_n$ is assumed to be  bounded and accounts for  both  brittle \cite{Griffith:1921} and cohesive \cite{barenblatt} fracture, where in the latter case $g_n$ depends explicitly on the crack opening $[u]:= u^+-u^-$.

Minimization  problems for \eqref{eq: intro basic energy} are usually complemented with Dirichlet data. Their  well-posedness in the space of \emph{generalized functions of bounded deformation} ($GSBD$) \cite{DM} has been a challenging task in very recent years,  see \cite{Flav, Iu3, Crismale, CrismaleCalcVar, newvito, Conti-Focardi-Iurlano:19, FPM}.  We also refer to  \cite{Almi, Caroccia, Chambolle-Conti-Francfort:2014, Crismale2, Scilla, Friedrich:15-2, FriedrichSolombrino, Bernd} for some recent applications.   In the present paper, we are interested in the effective behavior of a sequence of functionals $(\mathcal{E}_n)_n$ and corresponding minimization problems. The parameter $n$ may have different meanings: it may account for a regularization of the energy, represent  the size of a microstructure, or model the different mechanical responses of a composite material in  each of its components.  Identifying   the limit of $\mathcal{E}_n$  in the sense of $\Gamma$-convergence \cite{Braides-Defranceschi:98, DalMaso:93} arises as a natural problem  with various possible applications:    let us mention, for instance, the case of homogenization, i.e.,   $f_n(x, \xi)=f(x/\eps_n, \xi)$ and $g_n(x, \zeta_1, \zeta_2, \nu)=g\left(x/\eps_n, \zeta_1-\zeta_2, \nu\right)$ for $f$ and $g$ being $1$-periodic in the first variable. Here,   the limiting functional corresponds to the effective energy of the homogenized material.

\textbf{State-of-the-art:} 
Due to both its theoretical interest and its relevance for applications, the $\Gamma$-convergence analysis of free-discontinuity problems has been the subject of many contributions over the last three decades. Most of the attention, however, has been focused  on  the related but different context of functionals of the form
\begin{align}\label{eq: intro bv energy}
\mathcal{E}^{BV}_n(u)    = \int_{\Omega} f_n\big(x, \nabla u(x) \big) \, {\rm d}x +  \int_{J_u \cap \Omega} g_n\big(x,   [u](x),  \nu_u(x)\big) \, {\rm d} \mathcal{H}^{d-1} (x) 
\end{align} 
 involving the \emph{full deformation gradient} $\nabla u$.  (With a slight abuse of notation, we still use the notation  $g_n$ although the density  depends on $u^+$ and $u^-$ only in terms of $[u]= u^+-u^-$.)  The first result in this direction is the seminal work \cite{Braides-Defranceschi-Vitali} addressing the case of periodic homogenization. By assuming a linear growth of $g_n$ in $[u]$ along with  $p$-growth assumptions on $f_n$, the authors derive a limiting homogenized functional with $x$-independent densities    $f_{\hom}$ and $g_{\hom}$ on   the natural energy space of \emph{special functions of bounded variation} ($SBV$) \cite[Section 4]{Ambrosio-Fusco-Pallara:2000}. Unfortunately, the growth assumptions on $g_n$ do not comply with standard models for brittle and cohesive fracture where surface densities are assumed to be bounded.
 
  The result has been extended in  \cite{GP} to the non-periodic case and more natural growth conditions. There, however, the authors study a slightly simplified setting motivated by applications to quasistatic crack growth:   first, competitors for \eqref{eq: intro bv energy} are {\it scalar-valued}. Secondly, the surface energy densities $g_n$ are independent of the  jump height $[u]$. The scalar nature of the problem allows to use truncation techniques  (at least for bounded  Dirichlet data)  such that $SBV$ endowed with the $L^1$-topology  is still a natural setting for the $\Gamma$-convergence result.  In both contributions \cite{Braides-Defranceschi-Vitali, GP},  a remarkable property is that the limiting densities $f_{\hom}$ and $g_{\hom}$ are completely determined by $f$ and $g$, respectively. This  {\it non-interaction} of the bulk and the surface part of the energy is due to the  exponent $p>1$ in the growth assumption for $f_n$. In contrast, for $p=1$, it is indeed known by other examples that interaction effects in the limit are possible,  see, e.g., \cite{BBBF, BDM, Sto lavoro GSBV3}.

A comprehensive treatment of the $\Gamma$-convergence analysis for functionals of the form \eqref{eq: intro bv energy}  has been achieved  recently in the paper  \cite{Sto lavoro GSBV}, which includes the  vector-valued setting, assumes no periodicity, and complies with weaker coercivity conditions.  More precisely, the densities are assumed to satisfy 
\begin{equation}\label{eq: intro assumption}
\alpha  |\xi|^p\le f_n(x, \xi)\le  \beta  (1+|\xi|^p), \quad  \quad \quad  \alpha \le g_n(x, \zeta, \nu)\le \beta (1+|\zeta|)
\end{equation}
for $x \in \Omega$, $\xi \in \R^{d \times d}$, $\zeta \in \R^d$, and $\nu \in \R^d$ with $|\nu|=1$,  where
  $0<\alpha\le \beta$  are  positive constants. Due to the weaker growth assumptions on $g_n$ compared to \cite{Braides-Defranceschi-Vitali} and the vector-valued nature, the model is more relevant for applications in fracture mechanics. This, however, comes at the expense of considering a weaker  functional setting. Indeed, compactness of competitors with bounded energy can now only be expected  with respect to the convergence in measure in  the larger space of {\it generalized} special functions of bounded variation $GSBV$ \cite[Section 4]{Ambrosio-Fusco-Pallara:2000}.  Eventually, based on a general compactness result in  $GSBV^p$ (the subspace of $GSBV$ of functions with $p$-integrable gradient),  \cite{Manuel} complements the $\Gamma$-convergence analysis in \cite{Sto lavoro GSBV} by investigating corresponding  boundary value problems and showing   convergence of minimizers.

The authors in  \cite{Sto lavoro GSBV} make use of an abstract viewpoint: they first show that, under assumptions \eqref{eq: intro assumption}, the $\Gamma$-limit of \eqref{eq: intro bv energy} can still be represented as an integral functional on $GSBV^p$. To this aim, they use the localization method for $\Gamma$-convergence and the global method for relaxation developed in \cite{BFLM, BDM}.  This method compares asymptotic Dirichlet problems on small balls with different boundary data depending on the local properties of the functions, and provides a characterization of the energy densities in terms of cell formulas.   A technical difficulty lies in the fact that the procedure needs linear growth of $g_n$ in $[u]$ which is not available due to \eqref{eq: intro assumption}. This can be overcome, however, by means of a \emph{perturbation trick}:  a small perturbation of the functionals, depending on the jump
 height,  is considered, which  can be represented as an integral functional on $SBV^p$.   Then, by letting the perturbation parameter vanish and by truncating functions suitably, the representation  is extended to $GSBV^p$. Similar truncation techniques are also employed for the localization method in connection with the fundamental estimate \cite[Proposition~3.1]{Braides-Defranceschi-Vitali} to pass from the $L^p$-topology to the topology of measure convergence.  We already remark that such a tool is not available  for   functionals of the form \eqref{eq: intro basic energy}. In fact,  given  a control only on the symmetrized gradient, it is in general not possible to use smooth truncations to decrease the energy up to a small error.
 
After establishing the abstract representation result,  the authors in \cite{Sto lavoro GSBV} show that the bulk and surface parts of the energy do not interact in the limit.   This allows to derive the results of \cite{Braides-Defranceschi-Vitali, GP} as a corollary of their approach.  For further discussion below, we point out that a key step for the identification of the surface density relies on the $BV$-coarea formula to approximate   $GSBV^p$ functions  by piecewise constant functions.  A similar method has been used for the characterization of  lower semicontinuity in $SBV$ \cite{Ambrosio:90, Ambrosio:90-2}, for the so-called \emph{jump transfer} \cite[Theorem 2.1]{Francfort-Larsen:2003}, and it  is also at the core of the compactness result \cite{Manuel} needed to  treat boundary value problems.  This is a delicate issue for functionals of the form \eqref{eq: intro basic energy} since this technique is not available if only symmetrized gradients are controlled.

 \textbf{The present paper:} Summarizing, a rather complete picture of $\Gamma$-convergence for functionals given in \eqref{eq: intro bv energy} has been developed over the last years, extending also to the case of stochastic homogenization \cite{Sto lavoro GSBV2}. By way of contrast, the understanding of the counterpart  in the linearized setting is scarce.  This leads us to the purpose of our paper, which  exactly  aims at extending the results in \cite{Sto lavoro GSBV, Manuel}   to the more general framework of free-discontinuity functionals of the form \eqref{eq: intro basic energy}. 

Our first main result (Theorem \ref{th: gamma}) provides a  general  compactness and representation result for  $\Gamma$-limits  of sequences $(\mathcal{E}_n)_n$  on the space $GSBD^p$ (the subspace of $GSBD$ with $e(u) \in L^p$), endowed  with the topology of measure convergence. More precisely,  we  assume that the energy densities $f_n$ and $g_n$ satisfy
\[
\alpha | (\xi^T + \xi)/2 |^p \le f_n(x,\xi) \le  \beta (1+ |(\xi^T + \xi)/2 |^p), \quad  \quad
\alpha \le g_n(x,a,b,\nu) \le \beta 
\]
for $x \in \Omega$, $\xi \in \R^{d \times d}$, $a,b \in \R^d$, and $\nu \in \R^d$ with $|\nu|=1$, where $p>1$ and $0 < \alpha \le \beta$. We prove that, up to a subsequence, the functionals $\mathcal{E}_n$ converge to  a $\Gamma$-limit $\mathcal{E}$ which is still an integral functional of the form  \eqref{eq: intro basic energy}. 
 
Our second main result (Theorem \ref{thm: main thm d=2})  deals with the identification of the $\Gamma$-limit $\mathcal{E}$ by analyzing the relation of the densities $(f_n)_n$ and $(g_n)_n$ with the densities $f_\infty$ and $g_\infty$ of $\mathcal{E}$. In particular,    we investigate  under which conditions  bulk and surface effects decouple in the limit.  For the bulk density, we obtain 
\begin{align}\label{eq: asy-cell}
f_\infty(x,\xi) = \limsup_{\rho \to 0^+} \limsup_{n \to \infty} \frac{1}{\rho^d}  \int_{Q_{\rho}(x)} f_n\big(x, e(u)(x) \big) \, {\rm d}x,   
\end{align}
where we denote by $Q_\rho(x)$ the cube centered at $x$ with sidelength $\rho$, and the infimum is taken  among all functions $v\in W^{1,p}(Q_{\rho}(x);\R^d)$ satisfying $v(y) = \xi y$ near $\partial Q_\rho(x)$. 
When it comes to the surface energy instead, we consider some additional restrictions: we focus on the case where $g_n(x, \nu)$ is independent of the traces at  the  jump, and for $d>2$ we further assume that  $g_n=h$ is independent of $n$. We prove that  
\begin{align}\label{eq: asy-cell2}
g_\infty(x,\xi) = \limsup_{\rho \to 0^+} \limsup_{n \to \infty} \frac{1}{\rho^{d-1}}  \int_{Q^\nu_{\rho}(x)} g_n\big(x, \nu_u(x) \big) \, {\rm d}\mathcal{H}^{d-1},   
\end{align}
where  $Q^\nu_\rho(x)$ is a cube of  sidelength $\rho$ oriented by $\nu$, and the infimum is taken over all \emph{piecewise rigid functions} $v \in PR(Q^\nu_\rho\RRR(x))\EEE$ which near  $\partial Q^\nu_\rho(x)$ agree with the \emph{jump function}
\[
 \bar u_{x,\nu}(y) = \begin{cases}  0 & \text{if } (y-x) \cdot \nu \ge   0,\\  (1,0,\ldots,0)  & \text{if }  (y-x) \cdot \nu < 0. \end{cases}  
\] 
Here, $PR$ is the subset of $GSBD^p$ consisting of functions $u$ with  $e(u)\equiv 0$. (Due to independence of $g_n$ on the jump  height,  we point out that equivalently piecewise constant functions could be considered,  i.e., $u$ satisfies $\nabla u \equiv 0$.)  The reason for considering $d=2$ for general sequences $(g_n)_n$  lies in a technique for  approximating $GSBD^p$ functions by piecewise rigid functions which is only available in the planar setting.  As a corollary, again restricted to $d=2$, we are also able to deduce a (periodic) homogenization result (Corollary \ref{cor: hom}).

 In general dimensions, we can treat the case of constant sequences $g_n = h$. Here, it turns out that the limiting density $g_\infty$  is the so-called $BV$-elliptic envelope introduced in \cite{AmbrosioBraides2} as a condition for lower semincontinuity of functionals defined on partitions. As $h$ is independent of the jump height, $g_\infty$ also coincides with the $BD$-elliptic envelope introduced in \cite{FPM} in the more general context of variational problems in spaces of Caccioppoli-affine functions. As a special case (Corollary \ref{cor: relax}),  we deduce in any space dimensions that the relaxation with respect to measure convergence of integral functionals on $GSBD^p$ of the form
\[
\int_\Omega  f(x, e(u)(x))\,\mathrm{d}x+ \int_{J_u\cap \Omega} g(x, \nu_u(x))\,\mathrm{d} \mathcal{H}^{d-1}
\]
has the same structure with densities $\bar{f}$ and $\bar{g}$,  where $\bar{f}$ denotes the quasiconvex envelope of $f$ and $\bar{g}$ is the $BV$-elliptic envelope of $g$.  In particular, it is simply given by the superposition of the relaxation of the bulk energy in $W^{1,p}$ and of the surface energy in the space of piecewise constant functions.

In our  third main result, we eventually incorporate Dirichlet boundary data  (Proposition \ref{lemma: gamma bdy}), and show the convergence  of  (almost) minimizers for a sequence  $(\mathcal{E}_n)_n$ with the given conditions  to minimizers of $\mathcal{E}$ (Theorem \ref{th: Gamma existence}).

\textbf{Proof techniques and challenges:} In the sequel, we highlight some of the proof techniques focusing on the additional challenges with respect to models \eqref{eq: intro bv energy} in $GSBV^p$. For the compactness of $\Gamma$-convergence (Theorem \ref{th: gamma}),   we specify the localization technique already used in \cite{Sto lavoro GSBV} to the setting at hand. The  key ingredient  is a construction for  \emph{joining} two functions $u,v \in GSBD^p(\Omega)$, which is usually called the \emph{fundamental estimate} (Proposition \ref{lemma: fundamental estimate}). In doing this, one must ensure that the energy spent in  a  transition layer is small, when the two  functions  are close in the considered topology. Typically, this is achieved by means of a cut-off construction of the form $w := u \varphi + (1 - \varphi) v$ for some smooth $\varphi$ with $0 \le \varphi \le 1$. This, however, requires $L^p$-integrability of the functions $u$ and $v$, which is not a priori given in our context. In contrast to the $GSBV^p$ setting, this cannot be recovered with truncation arguments, and  we need a  considerably more involved  strategy to overcome this issue.

The main novel tool is a \emph{Korn-type inequality for functions with small jump}, established  recently  by {\sc Cagnetti, Chambolle, and Scardia}  \cite{FinalKorn},   which  generalizes   a two-dimensional result in \cite{Conti-Focardi-Iurlano:15} (see also \cite{Friedrich:15-3}) to arbitrary dimension.  It provides a control of the full gradient in terms of the symmetrized gradient, up to an exceptional set whose perimeter has a surface measure comparable to that of the discontinuity  set. We combine this tool with a covering technique of the transition layer by means of small cubes, which enables us to cut out an exceptional set with small \RRR volume \EEE and perimeter such that, in the residual set, the $L^p$ norm of $u-v$ is controlled in terms of $\|e(u-v)\|_{L^p}$ (up to a small rest). This finally allows  to   perform the usual cut-off construction. 
Concerning the representation of the limit, we profit of a very recent integral representation result proved in \cite[Theorem 2.1]{Crismale-Friedrich-Solombrino},  tailored for energies of the form \eqref{eq: intro basic energy}.

For  the identification of the limiting densities $f_\infty$ and $g_\infty$ (Theorem  \ref{thm: main thm d=2}), the essential point is to show that the minimization  in the asymptotic cell formulas \eqref{eq: asy-cell}--\eqref{eq: asy-cell2}  can  indeed  be restricted from $GSBD^p$  to Sobolev and piecewise constant functions, respectively. For the bulk density, this is achieved by using the Korn inequality for functions with small jump set \cite{FinalKorn}  to approximate $GSBD^p$ functions with asymptotically vanishing jump set by Sobolev functions. Afterwards, we can follow the lines of the $SBV$ proof \cite{Sto lavoro GSBV, GP}, in particular involving truncation methods to obtain a sequence of equiintegrable Sobolev functions, see Lemma \ref{lem: our Larsen}. For the surface density instead,  the challenge lies in approximating a sequence  in  $GSBD^p$  with vanishing symmetrized gradient  by  a sequence of  characteristic functions of sets with finite perimeter, see Lemma \ref{lemma: partition}.  The nonavailability of the \emph{coarea formula} in our setting  makes this a very delicate task which can be overcome by the application of a \emph{piecewise Korn-Poincar\'e inequality}, see Proposition \ref{th: kornpoin-sharp} and \cite{FriedrichSolombrino}, which has been derived only for $d=2$.  For the relaxation result in general space dimensions, this technique is not at our disposal, and we use directly a recent lower semicontinuity result for surface integrals in $GSBD^p$ \cite{FPM} for so-called  \emph{symmetric jointly convex functions}, see Definition~\ref{def: symm-conv}.

Finally, the  extension to this case of Dirichlet boundary conditions (Proposition \ref{lemma: gamma bdy} and Theorem~\ref{th: Gamma existence}) is not straightforward because of two main reasons.  First, the construction of  recovery sequences complying with the given data requires the usage of our novel fundamental estimate (Proposition \ref{lemma: fundamental estimate}) and  a recent extension result  \cite{Matteo}. Secondly,   according to the  compactness  results of \cite{Crismale, newvito, FriedrichSolombrino},  sequences with equibounded energy   in the Dirichlet setting converge in a slightly weaker sense compared to convergence in measure.

Summarizing, we believe that the present paper gives a  thorough analysis   of the $\Gamma$-convergence and relaxation problem for free-discontinuity problems in  linearized elasticity.  We  fix a convenient setting, develop a number of technical tools (all arising from very recent advances in the topic), and illustrate the main issues to be overcome in order to  remove the restrictions on the surface density  that we need to impose in Theorem \ref{thm: main thm d=2}.  This will  hopefully be the object of forthcoming achievements and contributions to the problem.

 The paper is organized as follows. In Section \ref{sec: main} we introduce the model and state our three main results. In Section \ref{section: preliminaries}, we collect basic properties of the function space $GSBD^p$, and we recall  integral representation formulas for functionals defined on Sobolev functions and piecewise constant functions. Section~\ref{sec: compactness} is devoted to the proof of Theorem \ref{th: gamma}, in particular we formulate and prove a fundamental estimate in $GSBD^p$ (Proposition \ref{lemma: fundamental estimate}). In Section \ref{sec: apprix}, we present two approximation results of $GSBD^p$ functions by Sobolev and characteristic functions, respectively. These are fundamental ingredients for the proof of Theorem~\ref{thm: main thm d=2} in Section \ref{sec: identi}, but also of independent interest. In Section \ref{sec: mini-proof}, we finally address the minimization problems  with Dirichlet boundary data.

\section{Setting of the problem and main results}\label{sec: main}

 In this section we present our main  results.  We start with some basic notation. In the sequel, $\Omega \subset \R^d$ always denotes an open set. Let $\mathcal{A}(\Omega)$ be the  family of open  and bounded  subsets of $\Omega$. We write $\chi_A$ for the characteristic function of any  $A\subset \R^d$, which is $1$ on $A$ and $0$ otherwise.  If $A$ is a set of finite perimeter, we denote its essential  boundary by $\partial^* E$, see \cite[Definition 3.60]{Ambrosio-Fusco-Pallara:2000}. For two sets $A,B \subset \R^d$, we denote by $A \triangle B$ their symmetric difference and by  ${\rm dist}(A,B)$   their Hausdorff distance.   Moreover, we write $B \subset \subset A$ if $\overline{B} \subset A$. 
 
  For $x$, $y\in  \R^d  $, we use the notation $x\cdot y$ for the scalar product and $|x|$ for the  Euclidean  norm.   Moreover, we  set     $\mathbb{S}^{d-1}:=\{x \in \R^d  \colon |x|=1\}$. We denote  by $\R^{d \times d}$   the set of $d\times d$ matrices. By $\R^{d\times d}_{\rm sym}$  and  $\R^{d\times d}_{\rm skew}$ we indicate  the subsets of symmetric and skew-symmetric matrices, respectively.  In particular, for $\xi \in \R^{d\times d}$, we define ${\rm sym}(\xi) = (\xi + \xi^{\rm T})/2$.  The Frobenius norm of a matrix $\xi$ is denoted by $|\xi|$.   
  
    For every $x\in \R^d  $ and $\rho>0$ we indicate by $B_\rho(x) \subset   \R^d  $ the open ball with center $x$ and radius $\rho$. Additionally, for $\nu \in \mathbb{S}^{d-1}$, we denote by  $Q^\nu_\rho(x)$ the cube centered at $x$ with sidelength $\rho$ and two faces normal to $\nu$. For $x = 0$, we simply write $Q^\nu_\rho$. We let $e_1 = (1,0,\ldots,0) \in \R^d$ and set $Q_\rho(x) = Q^{e_1}_\rho(x)$ for all $x \in \R^d$ and $\rho>0$.  We denote by $\mathcal{L}^d$ and $\mathcal{H}^k$ the $d$-dimensional Lebesgue measure and the $k$-dimensional Hausdorff measure, respectively.  We set $\R_+ = [0,+\infty)$.

For definition and properties of the space $GSBD^p(\Omega)$,  $1 < p < \infty$, we refer the reader to \cite{DM}. Some relevant properties are collected in Subsection \ref{sec: GSBD} below. In particular,  the approximate gradient is denoted by $\nabla u$ (it is well-defined, see  Lemma \ref{lemma: approx-grad})  and the (approximate) jump set is denoted by $J_u$  with corresponding normal $\nu_u$ and one-sided limits $u^+$ and $u^-$.  We also define $e(u) = \frac{1}{2} (  \nabla u + (\nabla u)^{\mathrm{T}})$.

\subsection{$\Gamma$-convergence for free-discontinuity problems in $GSBD^p$}

Given $0 < \alpha \le \beta <+ \infty$, let $f\colon \Omega \times \R^{d \times d} \to [0,+\infty)$ be a  Charath\'eodory  function such that for a.e.\ $x\in \Omega$ and all  $\xi \in \R^{d \times  d}$  we have
\begin{align}\label{eq: general bound}
\alpha | {\rm sym}(\xi)  |^p \le f(x,\xi) \le  \beta (1+ | {\rm sym}(\xi)  |^p). 
\end{align}
Moreover, let $g \colon \Omega  \times \R^d \times \R^d \times  \mathbb{S}^{d-1} \to [0,+\infty)$ be a Borel function satisfying 
\begin{align}\label{eq: general bound2}
\alpha \le g(x,a,b,\nu) \le \beta 
\end{align}
for $\mathcal{H}^{d-1}$-a.e.\ $x \in \Omega$ and  for all  $a,b \in \R^d$,   $\nu \in \mathbb{S}^{d-1}$. In what follows, we consider energy functionals $\mathcal{E}\colon GSBD^p(\Omega) \times \mathcal{A}(\Omega) \to [0,+\infty)$ of the form 
\begin{align}\label{eq: basic energy}
\mathcal{E}(u,A) = \int_{A} f\big(x, e(u)(x) \big) \, {\rm d}x +  \int_{J_u \cap A} g\big(x,  u^+(x), u^-(x),  \nu_u(x)\big) \, {\rm d} \mathcal{H}^{d-1} (x) 
\end{align} 
for each $u \in GSBD^p(\Omega)$ and each $A \in \mathcal{A}(\Omega)$. In this subsection, we present a general $\Gamma$-convergence result for functionals of this form.  (For an exhaustive treatment of $\Gamma$-convergence we refer the reader to \cite{Braides-Defranceschi:98, DalMaso:93}.) To formulate the result, we introduce some further notation:   for every $u \in GSBD^p(\Omega)$ and $A \in \mathcal{A}(\Omega)$ we define 
\begin{align}\label{eq: general minimization} 
\mathbf{m}_{ \mathcal{E}}(u,A) = \inf_{v \in GSBD^p(\Omega)} \  \big\{ \mathcal{E}(v,A)\colon \ v = u \ \text{ in a neighborhood of } \partial A \big\}.
\end{align}
For $x_0 \in \Omega$, $u_0 \in \R^d$, and $\xi \in  \mathbb{R}^{d \times d}  $ we introduce the functions  $\ell_{x_0,u_0,\xi}\colon \R^d \to \R^d$ by 
\begin{align}\label{eq: basic affine function}
\ell_{x_0,u_0,\xi}(x) =  u_0 + \xi (x-x_0). 
\end{align}
Moreover, for $x_0 \in \Omega$, $a,b \in \R^d$, and  $\nu \in \mathbb{S}^{d-1}$ we introduce  $u_{x_0,a,b,\nu} \colon \R^d \to \R^d$ by 
\begin{align}\label{eq: jump competitor}
u_{x_0,a,b,\nu}(x) = \begin{cases}  a & \text{if } (x-x_0) \cdot \nu  \ge   0,\\ b & \text{if }  (x-x_0) \cdot \nu < 0. \end{cases} 
\end{align}

We now proceed with our first main result.

\begin{theorem}[$\Gamma$-convergence]\label{th: gamma}
 Let $\Omega \subset \R^d$ be open.  Let $(f_n)_n$ and $(g_n)_n$ be sequences of functions satisfying \eqref{eq: general bound} and \eqref{eq: general bound2}, respectively.  Let $\mathcal{E}_n \colon GSBD^p(\Omega) \times \mathcal{A}(\Omega) \to [0,+\infty)$ be the corresponding sequence of functionals given in \eqref{eq: basic energy}.  Then, there exists $\mathcal{E}\colon  GSBD^p(\Omega)\times \mathcal{A}(\Omega) \to [0,+\infty)$ and a subsequence (not relabeled) such that
$$\mathcal{E}(\cdot,A) =\Gamma\text{-}\lim_{n \to \infty} \mathcal{E}_n(\cdot,A) \ \ \ \ \text{with respect to convergence in measure on $A$} $$
for all $A \in  \mathcal{A}(\Omega) $. Moreover, for every $u\in GSBD^p(\Omega)$ and  $A\in \mathcal{A}(\Omega)$ we have that
\begin{align}\label{eq: representation}
\mathcal{E}(u,A)= \int_A f_{\infty}\big(x,u(x),   \nabla u(x)   \big)\, {\rm d}x +\int_{J_u\cap A}g_{\infty}(x,u^+(x), u^-(x),\nu_u(x))\, {\rm d}\mathcal{H}^{d-1}(x),
\end{align}
where $f_\infty$ is given by  
\begin{align}\label{eq: f^epsilon_infty}
f_{\infty}(x_0,u_0,\xi):= \limsup_{\rho\to 0^+}\frac{\mathbf{m}_{\mathcal{E}}(\ell_{x_0,u_0,\xi},Q_{\rho}(x_0))}{\rho^d},
\end{align}
for all $x_0 \in \Omega$, $u_0 \in \R^d$, $\xi \in \mathbb{R}^{d \times d}$, and $g_\infty$ is given by  
\begin{align}\label{eq: g^epsilon_infty}
 g_{\infty}(x_0,a,b,\nu):= \limsup_{\rho\to 0^+}\frac{\mathbf{m}_{\mathcal{E}}(u_{x,a,b,\nu},Q^\nu_{\rho}(x_0))}{\rho^{d-1}}
\end{align}
for all $  x_0  \in \Omega$,  $a,b \in \R^d$, and $\nu \in \mathbb{S}^{d-1}$. 
\end{theorem}

 The compactness of $\Gamma$-convergence is proved  via the localization technique for $\Gamma$-convergence,  see Section \ref{sec: compactness}.   Here, the main ingredient is a novel fundamental estimate in the space $GSBD^p$, see  Proposition  \ref{lemma: fundamental estimate}. Afterwards, the representation \eqref{eq: representation} in terms of the densities $f_\infty$ and $g_\infty$  follows by the recent integral representation result \cite{Crismale-Friedrich-Solombrino}.

\begin{remark}[Invariance under rigid motions, cell formulas]\label{rem: invariance}
(i) Suppose that each $\mathcal{E}_n$ satisfies   $\mathcal{E}_n(u+a,A) = \mathcal{E}_n(u,A)$ for all affine functions  $a\colon \R^d \to \R^d$ with $e(a) = 0$ and all $A \in \mathcal{A}(\Omega)$. Then, as $\Gamma$-limit,  $\mathcal{E}$ satisfies the same property. Thus,    \cite[Remark 2.2(iii)]{Crismale-Friedrich-Solombrino} implies that $\mathcal{E}$ has the form
\begin{align}\label{eq: funci-old}
\mathcal{E}(u,A)= \int_A f_{\infty}\big(x,e(u)(x)\big)\, {\rm d}x +\int_{J_u\cap A} g_{\infty}\big(x,[u](x),\nu_u(x))\, {\rm d}\mathcal{H}^{d-1}(x),
\end{align}
where  $[u](x) := u^+(x)-u^-(x)$, and the densities $f_\infty$, $g_\infty$ are  given by  
\begin{align}\label{eq: f^epsilon_infty-new}
f_{\infty}(x_0,{\rm sym}(\xi))= \limsup_{\rho\to 0^+}\frac{\mathbf{m}_{\mathcal{E}}(\ell_{0,0,\xi},Q_{\rho}(x_0))}{\rho^d}, \quad \quad  g_{\infty}(x_0,\zeta,\nu)= \limsup_{\rho\to 0^+}\frac{\mathbf{m}_{\mathcal{E}}(u_{x,\zeta,0,\nu},Q^\nu_{\rho}(x))}{\rho^{d-1}}
\end{align}
for all $x_0 \in \Omega$,  $\xi \in \mathbb{R}^{d \times d}$,  $\zeta \in \R^d$, and $\nu \in \mathbb{S}^{d-1}$. \\
(ii) A variant of the proof shows that, in the minimization problems \eqref{eq: f^epsilon_infty}--\eqref{eq: g^epsilon_infty}, one may replace the cubes by balls $B_\rho(x_0)$ with radius $\rho$, centered at $x_0$. 
\end{remark}

\subsection{Identification of the $\Gamma$-limit: homogenization and relaxation}\label{sec: ident}

We now address the structure of the $\Gamma$-limit by showing that there is no interaction between the bulk and surface densities, i.e, $f_\infty$ is only determined by $(f_n)_n$ and $g_\infty$ is only determined by $(g_n)_n$. As applications, we discuss homogenization and relaxation results. The statements announced in this subsection are proved in Section \ref{sec: identi}. In this part, we restrict our assumptions to a more specific setting, namely to surface densities $g$ of the form $g \colon \Omega \times  \mathbb{S}^{d-1} \to [0,+\infty)$, still being Borel functions and satisfying 
\begin{align}\label{eq: general bound2-new}
\alpha \le g(x,\nu) \le \beta 
\end{align}
for $\mathcal{H}^{d-1}$-a.e.\ $x \in \Omega$ and  for all  $\nu \in \mathbb{S}^{d-1}$, where $0 < \alpha \le \beta <+ \infty$ as before. Moreover, for some parts we will further restrict our analysis to the planar setting $d=2$ and to exponents $p \ge 2$. We refer to Remark \ref{rem: assumptions} at the end of the subsection  for comments on these restrictions.

To formulate the  non-interaction between   bulk and surface density, we need to restrict  functionals $\mathcal{E}$ of the form \eqref{eq: basic energy}  to Sobolev functions  $W^{1,p}(\Omega;\R^d)$ and to \emph{piecewise rigid functions} $PR(\Omega)$, respectively. Here, we set $PR(\Omega) := \lbrace u \in GSBD^p(\Omega)\colon e(u) \equiv 0 \rbrace$, see \cite{FM}. Then, similarly to \eqref{eq: general minimization},   for every $u \in W^{1,p}(\Omega;\R^d)$ and $A \in \mathcal{A}(\Omega)$ we define 
\begin{align}\label{eq: general minimization2} 
\mathbf{m}^{1,p}_{\mathcal{E}}(u,A) = \inf_{v \in W^{1,p}(\Omega;\R^d)} \  \big\{ \mathcal{E}(v,A)\colon \ v = u \ \text{ in a neighborhood of } \partial A \big\},
\end{align}
and, for every $u \in PR(\Omega)$, we let 
\begin{align}\label{eq: general minimization3} 
\mathbf{m}^{PR}_{\mathcal{E}}(u,A) = \inf_{v \in PR(\Omega)} \  \big\{ \mathcal{E}(v,A)\colon \ v = u \ \text{ in a neighborhood of } \partial A \big\}.
\end{align}
 Due to the fact that the surface integral vanishes on $W^{1,p}$, and by the  definition of $PR(\Omega)$ as well as the upper control in \eqref{eq: general bound}, we find for all $A \in \mathcal{A}(\Omega)$ that 
\begin{align}
& \mathcal{E}(u,A) = \int_A f(x,e(u)(x))\, {\rm d}x \quad \text{for all }  u\in W^{1,p}(\Omega;\R^d),\label{eq: almost the same1} \\
& \Big|\mathcal{E}(u,A) - \int_{J_u \cap A} g(x,\nu_u(x))\,{\rm d}\mathcal{H}^{d-1}\Big| \le \beta \mathcal{L}^d(A) \quad \text{for all }  u\in PR(\Omega). \label{eq: almost the same2}
\end{align}
For convenience, we specify the notation in  \eqref{eq: basic affine function}--\eqref{eq: jump competitor}   and write
\begin{align}\label{eq: specified not}
\bar{\ell}_{\xi}= \ell_{0,0,\xi} \ \text{ for $\xi \in \R^{d \times d}$ } \quad \quad \text{ and } \quad \quad  \bar{u}_{x_0,\nu} = u_{x_0,e_1,0,\nu} \ \text{ for  $e_1= (1,0,\ldots,0)$}.
\end{align}
In particular, we have $\bar{\ell}_{\xi}(y) = \xi y$ for all $y\in \R^d$.

\begin{proposition}\label{prop: prep}
Let $\Omega \subset \R^d$ be open. Let $(f_n)_n$ and $(g_n)_n$ be sequences of functions satisfying \eqref{eq: general bound} and \eqref{eq: general bound2-new}, respectively.  Correspondingly, define $(\mathcal{E}_n)_n$ as in \eqref{eq: basic energy}.   Then, by passing to a subsequence  (not relabeled) the following holds:  for all $ x  \in \Omega$ and all $\xi \in \R^{d \times d}$  we have
\begin{align}\label{eq: bulk assumption}
\limsup_{\rho \to 0^+} \liminf_{n \to \infty} \frac{\mathbf{m}^{1,p}_{\mathcal{E}_n}(\bar{\ell}_\xi,Q_{\rho}(x))}{\rho^d} = \limsup_{\rho \to 0^+} \limsup_{n \to \infty} \frac{\mathbf{m}^{1,p}_{\mathcal{E}_n}(\bar{\ell}_\xi,Q_{\rho}(x))}{\rho^d} =: f(x,\xi) ,  
\end{align}
and it holds that $f(x,\xi) = f(x,{\rm sym}(\xi))$. For all $x \in \Omega$ and all $\nu \in \mathbb{S}^{d-1}$ we have
\begin{align}\label{eq: surface assumption}
\limsup_{\rho \to 0^+} \liminf_{n \to \infty} \frac{\mathbf{m}^{PR}_{\mathcal{E}_n}(\bar{u}_{x,\nu},Q^\nu_{\rho}(x))}{\rho^{d-1}} = \limsup_{\rho \to 0^+} \limsup_{n \to \infty} \frac{\mathbf{m}^{PR}_{\mathcal{E}_n}(\bar{u}_{x,\nu},Q^\nu_{\rho}(x))}{\rho^{d-1}}  =:g(x,\nu). 
\end{align}
\end{proposition}

In the above formulas, we intend that $\rho$ is always chosen sufficiently small such that the cubes $Q_\rho(x)$ and $Q^\nu_\rho(x)$ are contained in $\Omega$.  In view of \eqref{eq: almost the same1}--\eqref{eq: almost the same2},  the density $f$ is completely determined by $(f_n)_n$, whereas $g$ is completely determined by $(g_n)_n$.    This motivates the definition of \eqref{eq: general minimization2}--\eqref{eq: general minimization3}. The proof of Proposition \ref{prop: prep} essentially relies on $\Gamma$-convergence results for Sobolev functions and piecewise constant functions, see Subsection \ref{sec: prep intrep}.

  By Theorem \ref{th: gamma} we get that up to subsequence (not relabeled), the functionals $\mathcal{E}_n(\cdot,A)$  given in \eqref{eq: basic energy}  with densities $f_n$ and $g_n$, $\Gamma$-converge with respect to the convergence in measure to a functional $\mathcal{E}(\cdot,A)$ for every $A\in\mathcal{A}(\Omega)$.  As each  $\mathcal{E}_n$ satisfies   $\mathcal{E}_n(u+a,A) = \mathcal{E}_n(u,A)$ for all affine functions $a$ with  $e(a) = 0$  and all $A \in \mathcal{A}(\Omega)$, Remark \ref{rem: invariance}(i) shows that the densities $f_\infty$ and $g_\infty$ of the $\Gamma$-limit can be represented by \eqref{eq: f^epsilon_infty-new}.  We now proceed with our second main result.  We show that the density $f_\infty$  coincides with  the function $f$ provided by Proposition \ref{prop: prep}. Hence, the surface energies are not contributing to the bulk part of the limiting functional.  We also prove the analogous property $g_\infty=g$ for the surface densities in two specific situations: (a) in the planar case $d=2$ and (b) for $d>2$, whenever  the surface densities $g_n$ are independent of $n$.

\begin{theorem}[Identification of the $\Gamma$-limit]\label{thm: main thm d=2}
Let $\Omega \subset \R^d$ be open. Let $(f_n)_n$ and $(g_n)_n$ be sequences of functions satisfying \eqref{eq: general bound} and \eqref{eq: general bound2-new}, respectively.  Suppose that \eqref{eq: bulk assumption}--\eqref{eq: surface assumption} hold, and define $f$ and $g$ accordingly. Let $f_\infty$ and $g_\infty$ be defined by  \eqref{eq: f^epsilon_infty-new}.  Then, the following holds:
\begin{enumerate}
\item[(i)]
For all $u\in GSBD^p(\Omega)$ we have that 
\begin{align}\label{eq: f_infty=f}
 f_\infty(x,e(u)(x))  = f(x,e(u)(x))\quad \textit{\emph{for $\mathcal{L}^d$-a.e.\ $x\in\Omega$}}.
\end{align}
\item[(ii)]
 If $d=2$, $p \ge 2$, we additionally  have 
\begin{align}\label{eq: g_infty=g}
g_\infty\big(x,[u](x),\nu_u(x)\big)=  g (x,\nu_u(x))\quad \textit{\emph{for  $\mathcal{H}^{1}$-a.e.\ $x\in J_u$}}
\end{align}
for all  $u\in GSBD^p(\Omega)$. In particular,  for all $A\in\mathcal{A}(\Omega)$ the functionals $\mathcal{E}_n(\cdot,A)$, given by \eqref{eq: basic energy} corresponding to $f_n$ and $g_n$,  $\Gamma$-converge with respect to the convergence in measure to $\mathcal{E}(\cdot,A)$, where $\mathcal{E}$ is given by  
 \begin{align*}
\mathcal{E}(u,A)=\int_A f\big(x,e(u)(x)\big)\,{\rm d}x + \int_{J_u\cap A} g\big(x,\nu_u(x)\big)\, {\rm d}\mathcal{H}^{d-1}(x).
\end{align*}
\item[(iii)]
Let $d \ge 2$ and assume that $g_n = h$ for all $n \in \N$, for a continuous density $h$ which satisfying  \eqref{eq: general bound2-new} such that $\nu \mapsto h(x,\nu)$ is even for all $x \in \Omega$. Then   for all $u \in GSBD^p(\Omega)$ we have 
\begin{align}\label{eq: g_infty=g-relax}
g_\infty\big(x,[u](x),\nu_u(x)\big)= \bar{h} (x,\nu_u(x))\quad \textit{\emph{for $\mathcal{H}^{d-1}$-a.e.\ $x\in J_u$}},
\end{align}
where $\bar{h}$ is given by 
\begin{align}\label{eq: envelope}
 \hspace{1.1cm} \bar{h}(x,\nu):= \inf_{v \in PR(Q^\nu_1) }  \Big\{ \int_{J_v \cap Q^{\nu}_1}  h( x, \nu_v(y)) \, {\rm d}\mathcal{H}^{d-1}(y) \colon \, v =  \bar{u}_{0,\nu} \text{ in a neighborhood of $\partial Q^\nu_1$} \Big\}.
\end{align}
In particular,  for all $A\in\mathcal{A}(\Omega)$ the functionals $\mathcal{E}_n(\cdot,A)$, given by \eqref{eq: basic energy} corresponding to $f_n$ and $g_n=h$,  $\Gamma$-converge with respect to the convergence in measure to $\mathcal{E}(\cdot,A)$, where $\mathcal{E}$ is defined by  
\begin{align*}
\mathcal{E}(u,A)=\int_A f\big(x,e(u)(x)\big)\,{\rm d}x + \int_{J_u\cap A}\bar{h}\big(x,\nu_u(x)\big)\, {\rm d}\mathcal{H}^{d-1}(x).
\end{align*}
\end{enumerate}
\end{theorem}

The function $\bar{h}$  in \eqref{eq: envelope} is the so-called $BD$-elliptic envelope of $h$, introduced in \cite{FPM}. Let  us  point out that  in the present setting it  coincides with the $BV$-elliptic envelope introduced in \cite{AmbrosioBraides2}, see Remark \ref{rem: envelope} below.  It also turns out the $\bar{h}$ can be characterized via \eqref{eq: surface assumption}, see Corollary \ref{cor: surf}. 
By Proposition \ref{prop: prep} the above result implies a non-interaction between the bulk and surface energy:  $f_\infty$  and (under certain restrictions) $g_\infty$  are completely determined by $(f_n)_n$ and $(g_n)_n$, respectively.  A first application is the following homogenization result, which we state in dimension $d=2$ as it needs part (ii) of the above statement.

\begin{corollary}[Homogenization]\label{cor: hom}
Let  $p \ge 2$,  $\Omega = \R^2$, and consider $f$ and $g$ satisfying \eqref{eq: general bound} and \eqref{eq: general bound2-new}, respectively.  Suppose that for every $x \in \R^2$, $\xi \in \R^{2 \times 2}$, and $\nu \in \mathbb{S}^1$ we have that the limits
\begin{align}\label{eq: hom-lim}
f_{\rm hom}({\rm sym}(\xi)) := \lim_{r \to \infty} \frac{\mathbf{m}^{1,p}_{\mathcal{E}}(  \bar{\ell}_\xi,  Q_r(rx))}{r^2}, \quad \quad  g_{\rm hom}(\nu) := \lim_{r \to \infty} \frac{\mathbf{m}^{PR}_{\mathcal{E}}(  \bar{u}_{rx,\nu},  Q^\nu_r(rx))}{r}
\end{align} 
exist and are independent of $x$. Let $(\eps_n)_n \subset (0,\infty)$ be a sequence with $\eps_n  \to 0$, and for $n \in \N$ let 
$$f_n(x,\xi) := f(x/\eps_n,\xi), \quad \quad \quad g_n(x,\nu) = g(x/\eps_n,\nu) $$
for $x \in  \R^2  $, $\xi \in \R^{2 \times 2}$, and $\nu \in  \mathbb{S}^1$. Then, for all $A\in\mathcal{A}(\Omega)$ the functionals $\mathcal{E}_n(\cdot,A)$, given by \eqref{eq: basic energy} with densities $f_n$ and $g_n$,  $\Gamma$-converge with respect to the convergence in measure to $\mathcal{E}_{\rm hom}(\cdot,A)$, where  $\mathcal{E}_{\rm hom}$  is defined in \eqref{eq: basic energy} with densities $f_{\rm hom}$ and $g_{\rm hom}$. 
\end{corollary}

In view of  \cite[Propositions 2.1 and 2.2]{Braides-Defranceschi-Vitali}, \eqref{eq: hom-lim} can always be verified,  whenever $f$ and $g$ are periodic of period $1$ with respect to the coordinates $e_1$ and $e_2$.  In general dimension, using part (iii) of Theorem \ref{thm: main thm d=2}, we obtain the following relaxation result.

\begin{corollary}[Relaxation]\label{cor: relax}
Let $\Omega \subset \R^d$ be open for $d \ge 2$. Suppose that $f$  and $g$ satisfy  \eqref{eq: general bound} and \eqref{eq: general bound2-new}, respectively. Denote by $\bar{\mathcal{E}}$ the relaxation of $\mathcal{E}$ given by   \eqref{eq: basic energy} corresponding to $f$ and $g$, i.e.,
$$\bar{\mathcal{E}}(u,A) = \inf\Big\{ \liminf_{n \to \infty} \mathcal{E}(u_n,A) \colon \, u_n \to u \text{ in measure on $A$}   \Big\} $$
for all $u\in GSBD^p(\Omega)$ and $A\in\mathcal{A}(\Omega)$. Then, $\bar{\mathcal{E}}$ is characterized by
\begin{align*}
\bar{\mathcal{E}}(u,A)=\int_A \bar{f}\big(x,e(u)(x)\big)\,{\rm d}x + \int_{J_u\cap A}\bar{g}\big(x,\nu_u(x)\big)\, {\rm d}\mathcal{H}^{d-1}(x)
\end{align*}
for all $u\in GSBD^p(\Omega)$ and $A\in\mathcal{A}(\Omega)$, where $\bar{f}$ denotes the quasiconvex envelope  (with respect to the second variable) of $f$,   and  $\bar{g}$ is   the $BD$-elliptic envelope of $g$ defined in \eqref{eq: envelope}.   
\end{corollary}

\begin{remark}[Discussion on assumptions]\label{rem: assumptions} 
Theorem \ref{thm: main thm d=2}(ii) only holds in dimension $d=2$ since for the identification of the surface density we apply a piecewise Korn-Poincar\'e inequality \cite{Friedrich:15-4} which is only available in the planar setting. For similar reasons, we need to restrict ourselves to exponents $p \ge 2$. We refer to Remark \ref{rem: d,p} for more details in that direction. In the statement of  Theorem~\ref{thm: main thm d=2}(iii)  we need to assume continuity of $h$ in order to apply relaxation results for piecewise constant functions \cite{AmbrosioBraides2},  see Proposition \ref{prop: surface int repr-relax}. Eventually, the assumption that $h$ is even turns out to be instrumental to apply lower semicontinuity results in $GSBD^p$, see  Theorem \ref{thm: GSBD LSC} and Proposition \ref{prop: only nu2}  below. Let  us also mention that our strategy exploits explicitly the fact that the surface densities are of the form \eqref{eq: general bound2-new}, i.e., do not depend on $u^+(x)$ and $u^-(x)$.  
\end{remark}

\begin{remark}[Continuity of $h$ in Theorem \ref{thm: main thm d=2}(iii)]\label{rem: not-cont}
As a final remark, we record that the continuity assumption on $h$ in Theorem \ref{thm: main thm d=2}(iii) can be  slightly  altered  in the following sense: suppose that  there exists  $D \in \mathcal{A}(\Omega)$ with Lipschitz boundary such that $h$ is uniformly continuous on $D$ and  
\begin{align}\label{eq: h-inequaly}
\sup_{(x,\nu) \in D \times \mathbb{S}^{d-1}} h (x,\nu) \le  \inf_{(x,\nu) \in (\Omega \setminus  D  ) \times \mathbb{S}^{d-1}} h (x,\nu). 
\end{align}
Then \eqref{eq: g_infty=g-relax} holds for all $u \in GSBD^p(\Omega)$ with $J_u \subset \overline{D} \cap \Omega$. We refer to Subsection \ref{sec: sub4} for details.
\end{remark}

\subsection{Minimization problems for given boundary data}\label{sec: mini}

We complement the $\Gamma$-convergence results of the previous subsection by convergence results for minimizers of certain boundary value problems, as it is customary in many applications. 
We impose Dirichlet data   on  $\partial_D \Omega := \Omega' \cap \partial \Omega$, where $\Omega$ denotes a bounded Lipschitz domain and  $\Omega' \supset \Omega$ denotes another bounded Lipschitz domain    such that   also $\Omega' \setminus \overline{\Omega}$ has Lipschitz boundary.  This will be achieved by requiring $u = u^0$ on $\Omega' \setminus \overline{\Omega}$ for some datum $u^0 \in  W^{1,p}(\Omega';\R^d)$, i.e.,  we will treat the non-attainment of the boundary data (in the sense of traces) as internal jumps. To this end,  we introduce energy functionals defined on $\Omega'$. Consider sequences of densities $(f_n)_n$ and $(g_n)_n$ as in \eqref{eq: general bound} and \eqref{eq: general bound2-new}, respectively.  We define 
\begin{align}\label{eq: f ext}
f'_n(x,\xi) := \begin{cases} f_n(x,\xi) & \text{if } x \in \Omega, \\ \alpha |{\rm sym}(\xi)|^p&  \text{otherwise.} \end{cases}
\end{align}
and  
\begin{align}\label{eq: g ext}
g'_n(x,\nu) := \begin{cases} g_n(x,\nu) & \text{if } x \in   \Omega  , \\ \beta +1 &  \text{otherwise.} \end{cases}
\end{align}
We assume that, both for $\mathcal{E}_n$ and $\mathcal{E}'_n$, \eqref{eq: bulk assumption}--\eqref{eq: surface assumption} hold, which is always true up to taking a subsequence. Accordingly, we  define $f$, $g$ (for $\mathcal{E}_n$), and $f'$, $g'$ (for $\mathcal{E}'_n$).   We remark that, in this setting, one can show that  $ f'(x,\xi) =  f(x,\xi)$ for $x \in \Omega$ and $f'(x,\xi) = \alpha |{\rm sym}(\xi)|^p$ else, as well as $g'(x,\nu) = g(x,\nu)$ for $x \in \Omega$. Finally, for $x \in \partial_D \Omega$ the value of $g'(x,\nu)$ is completely determined by $(g_n)_n$, and is independent of the choice of   $\Omega'$, see \cite[Remark 4.4]{Manuel} for details.

By Theorem \ref{th: gamma} and Remark \ref{rem: invariance}(i) the functionals $\mathcal{E}_n$, with densities $f_n$ and $g_n$, $\Gamma$-converge with respect to the convergence in measure  (up to a subsequence) to a limiting functional  $\mathcal{E}$ with densities  $f_\infty$ and $g_\infty$. By the results in the previous subsection,  we know that  $f_\infty$ agrees with the function $f$, see Theorem \ref{thm: main thm d=2}(i). In the sequel, we suppose that  $g_\infty=g$.   (For instance, such a property holds in the setting of Theorem \ref{thm: main thm d=2}(ii),(iii).) In a similar fashion, the functionals $\mathcal{E}_n'$ with densities $f_n'$ and $g_n'$, $\Gamma$-converge to some  $\mathcal{E}'$.  Again, we know that the bulk density of $\mathcal{E}'$ is the function $f'$ in  \eqref{eq: bulk assumption} and we assume that the surface density is given by $g'$  in \eqref{eq: surface assumption}. As before, this  characterization can be ensured in the setting of  Theorem \ref{thm: main thm d=2}(ii) or in the setting of Theorem~\ref{thm: main thm d=2}(iii) under the assumption that $h$ is uniformly continuous on $\Omega$. For the latter case, we need to resort to  Remark \ref{rem: not-cont}  (with  $\Omega'$ in place of $\Omega$ and $D = \Omega$)   since the continuity of the density in \eqref{eq: g ext} gets lost through the extension. (Note that indeed  $J_u \subset \overline{D}  \cap \Omega' $ holds since we require $u = u^0$ on $\Omega' \setminus \overline{\Omega}$.)

We now  present  the following version of the $\Gamma$-convergence result which takes  boundary data into account.  We remark that  the statement takes a more general point of view than assuming the setting of Theorem \ref{thm: main thm d=2}(ii),(iii): the result below is true \emph{whenever} the limiting surface density $g$ is determined solely by the functions $g_n$ through the asymptotic minimization problems discussed in the previous subsection.

\begin{proposition}[$\Gamma$-convergence with boundary data]\label{lemma: gamma bdy}
Let $(f_n)_n$ and $(g_n)_n$ be sequences of functions satisfying \eqref{eq: general bound} and \eqref{eq: general bound2-new}, respectively. Consider the sequence of functionals $(\mathcal{E}'_n)_n$ with  densities $(f'_n)_n$, $(g'_n)_n$ defined as in \eqref{eq: f ext}--\eqref{eq: g ext}. Assume that, both for $\mathcal{E}_n$ and $\mathcal{E}'_n$, \eqref{eq: bulk assumption}--\eqref{eq: surface assumption} hold, and accordingly define $f$, $g$ (for $\mathcal{E}_n$), and $f'$, $g'$ (for $\mathcal{E}'_n$).  Consider the $\Gamma$-limit  $\mathcal{E}'$  of $(\mathcal{E}'_n)_n$ with densities $f'_\infty$ and $g'_\infty$, and assume that  $g'_\infty=g'$. 
Suppose that $(u^0_n)_n \subset W^{1,p}(\Omega';\R^d)$ converges strongly to $u^0$ in $W^{1,p}(\Omega';\R^d)$. Then the sequence of functionals 
$$\tilde{\mathcal{E}}'_n(u) = \begin{cases} \mathcal{E}'_n(u) & \text{if } u = u^0_n \text{ on } \Omega' \setminus \overline{\Omega}, \\ +\infty &  \text{otherwise},  \end{cases} $$ 
$\Gamma$-converges with respect to the convergence in measure  to 
$$\tilde{\mathcal{E}}'(u) = \begin{cases} \mathcal{E}'(u) & \text{if } u = u^0\text{ on } \Omega' \setminus \overline{\Omega}, \\ +\infty &  \text{otherwise.}  \end{cases}  $$
\end{proposition}

We emphasize once more that the assumption  $g'_\infty=g'$  on the surface density covers the setting of Theorem \ref{thm: main thm d=2}(ii),(iii). However, it is  not limited to that since above we have no restriction on the dimension.  Instead, for our main result about convergence of minimizers,  we focus again  on the setting of Theorem \ref{thm: main thm d=2}(ii),(iii). 

\begin{theorem}[Convergence of minima and minimizers]\label{th: Gamma existence}
Let $(f_n)_n$ and $(g_n)_n$ be sequences of functions satisfying \eqref{eq: general bound} and \eqref{eq: general bound2-new}, respectively. Suppose either that
\begin{itemize}
\item[(i)]$d = p = 2$,
\item[(ii)]  $d\ge 2$ and $g_n = \hat{g}$ for $n \in \N$, where $\hat{g}$ denotes  a  uniformly continuous density with $\nu \mapsto \hat g(x,\nu)$ being even for all  $x \in \Omega$. 
\end{itemize} 
Consider the  sequence of functionals $(\tilde{\mathcal{E}}'_n)_n$ and the limiting energy $\tilde{\mathcal{E}}'$ given by Proposition   \ref{lemma: gamma bdy}, for boundary data $(u^0_n)_n \subset W^{1,p}(\Omega';\R^d)$ which converge strongly in $W^{1,p}(\Omega';\R^d)$ to $u^0$. Then 
\begin{align}\label{eq: eps control2}
\inf_{v \in GSBD^p(\Omega')} \tilde{\mathcal{E}}'_n(v) \  \to \  \min_{v \in GSBD^p(\Omega')} \tilde{\mathcal{E}}'(v)
\end{align}
 for $n \to \infty$. Moreover, for each sequence $(u_n)_n$ with
\begin{align}\label{eq: eps control}
\tilde{\mathcal{E}}'_n(u_n) \le \inf_{v \in GSBD^p(\Omega')} \tilde{\mathcal{E}}'_n(v) + \eps_n 
\end{align}
for a sequence $\eps_n \to 0$, there exist a subsequence (not relabeled), modifications $(y_n)_n$ satisfying  $\mathcal{L}^d(\lbrace e(y_n) \neq  e(u_n)  \rbrace) \to 0$ as $n \to \infty$, and  $u \in GSBD^p(\Omega')$ with  $y_n \to u$ in measure on $\Omega'$ such that 
$$ \lim_{n \to \infty} \tilde{\mathcal{E}}'_n(u_n) =  \tilde{\mathcal{E}}'(u) = \min_{v \in GSBD^p(\Omega')} \tilde{\mathcal{E}}'(v).$$
In case (i), we additionally have $\lim_{n \to \infty} \tilde{\mathcal{E}}'_n(y_n) = \tilde{\mathcal{E}}'(u)$, i.e., $(y_n)_n$ is a minimizing sequence converging to the minimizer $u$.
\end{theorem}
 For the proofs of the results we refer to Section \ref{sec: mini-proof}. We point out that in case (i) we obtain a slightly stronger statement. This is due to compactness properties of $GSBD^p$ functions and the construction of certain modifications, see Theorem \ref{th: comp} below.

\section{Preliminaries}\label{section: preliminaries}

In this section, we collect basic properties of the function space $GSBD^p$ and we recall integral representation formulas for functionals defined on Sobolev functions and piecewise constant functions.

\subsection{Generalized special functions of bounded deformation}\label{sec: GSBD}

 In this subsection, we collect  fundamental properties of the function space $GSBD^p$.

\textbf{$GSBD$-functions, basic properties:}  The  space  $GSBD(\Omega)$ of \emph{generalized special functions of bounded deformation}  has been introduced  in  \cite[Definitions~4.1 and 4.2]{DM}). We recall that every $u\in GSBD(\Omega)$ has an \emph{approximate symmetric gradient} $e(u)\in L^1(\Omega;\R^{d \times d}_{\rm sym})$ and an \emph{approximate jump set} $J_u$.  For 
$x\in J_u$  there exist $u^+(x)$, $u^-(x)\in \R^d$  and $\nu_u(x)\in\mathbb{S}^{d-1}$  
 such that
\begin{equation}\label{0106172148}
\lim_{\rho \to 0}\rho^{-d}\mathcal{L}^d\big(\{y \in B_\rho(x)\colon \pm(y-x)\cdot \nu_u(x)>0\} \cap \{|u-u^\pm(x)|>\eps\}\big)=0
\end{equation}
 for every $\eps>0$, and the function  $ [u]:=u^+-u^-  \colon J_u \to \R^d$ is measurable.  For $1 < p < + \infty$, the   space $GSBD^p(\Omega)$ is  given by  
\begin{equation*}
GSBD^p(\Omega):=\{u\in GSBD(\Omega)\colon e(u)\in L^p(\Omega;\R^{d \times d}_{\rm sym}),\, \mathcal{H}^{d-1}(J_u)<\infty\}.
\end{equation*}
For $u\in GSBD^p(\Omega)$, the \emph{approximate gradient} $\nabla u$ exists $\mathcal{L}^d$-a.e.\ in $\Omega$,  see \cite[Corollary 5.2]{FinalKorn}:

\begin{lemma}[Approximate gradient]\label{lemma: approx-grad}
Let $\Omega \subset \R^d$ be   open,     let $1 < p < +\infty$, and $u \in GSBD^p(\Omega)$. Then for $\mathcal{L}^d$-a.e.\   $x_0 \in \Omega$  there exists a matrix  in $\R^{d \times d}$,   denoted by $\nabla u(x_0)$, such that
$$\lim_{\rho \to 0} \  \rho^{-d} \mathcal{L}^d\Big(\Big\{x \in B_\rho(x_0) \colon \,  \frac{|u(x) - u(x_0) - \nabla u(x_0)(x-x_0)|}{|x - x_0|}  > \eps   \Big\} \Big)  = 0 \text{ for all $\eps >0$}, $$
 and ${\rm sym}(\nabla u(x_0)) = e(u)(x_0)$, where $e(u)(x_0)$ denotes the approximate symmetric gradient.      
\end{lemma}

We point out that the result in Lemma \ref{lemma: approx-grad} has already been obtained in \cite{Friedrich:15-4}  for $p=2$, as a consequence of the embedding $GSBD^2(\Omega) \subset (GBV(\Omega))^d$, see \cite[Theorem 2.9]{Friedrich:15-4}.

%

\textbf{Korn inequalities in $GSBD^p$:}    We  recall Korn and Poincar\'e inequalities in $GSBD^p$. In what follows, we say that $a\colon \R^d \to \R^d$ is a \emph{rigid motion} if $a$ is affine with $ e(a) = \frac{1}{2}  ( \nabla a + (\nabla a)^{\mathrm{T}})  = 0$. We start by  Korn and Korn-Poincar\'e  inequalities for functions with small jump set, see  \cite[Theorem~1.1, Theorem~1.2]{FinalKorn}.

\begin{theorem}[Korn inequality for functions with small jump set]\label{th: kornSBDsmall}
Let $\Omega \subset \R^d$ be a bounded Lipschitz domain and let $1 < p < +\infty$. Then there exists a constant $c = c(\Omega,p)>0$ such that for all  $u \in GSBD^p(\Omega)$ there  exists  a set of finite perimeter $\omega \subset \Omega$ with 
\begin{align*}
\mathcal{H}^{d-1}(\partial^* \omega) \le c\mathcal{H}^{d-1}(J_u), \ \ \ \ \mathcal{L}^d(\omega) \le c(\mathcal{H}^{d-1}(J_u))^{d/(d-1)}
\end{align*}
and a  rigid motion $a$ such that
\begin{align}\label{eq: main estmain}
\Vert u - a \Vert_{L^{p}(\Omega \setminus \omega)} + \Vert \nabla u - \nabla a \Vert_{L^{p}(\Omega \setminus \omega)}\le c \Vert e(u) \Vert_{L^p(\Omega)}.
\end{align}
Moreover, there exists $v \in W^{1,p}(\Omega;\R^d)$ such that $v= u$ on $\Omega \setminus \omega$ and
\begin{align*}
\Vert e(v) \Vert_{L^p(\Omega)} \le c \Vert e(u) \Vert_{L^p(\Omega)}.
\end{align*}
\end{theorem} 
 
 Note that in  \cite[Theorem 1.1]{FinalKorn}  $\mathcal{L}^d(\omega) \le c(\mathcal{H}^{d-1}(J_u))^{d/(d-1)}$ has not been  stated explicitly,  but it readily follows from $\mathcal{H}^{d-1}(\partial^* \omega) \le c\mathcal{H}^{d-1}(J_u)$ by the isoperimetric inequality.

\begin{remark}[Scaling invariance on  cubes]\label{rem: Korn-scaling}
If $\Omega= Q_\rho$ for $\rho >0$, then we find $\omega \subset  Q_\rho$ and a rigid  motion $a$ such that
$$\mathcal{H}^{d-1}(\partial^* \omega) \le \bar{c} \mathcal{H}^{d-1}(J_u), \ \ \ \ \mathcal{L}^d(\omega)\ \le \bar{c}\big(\mathcal{H}^{d-1}(J_u)\big)^{d/(d-1)} $$
and 
$$\Vert u - a \Vert^p_{L^{p}(Q_\rho \setminus \omega)}\le \bar{c}\rho^p \Vert e(u) \Vert^p_{L^p(Q_\rho)}, $$
where $\bar{c} = \bar{c}(p) >0$ is independent of the sidelength $\rho$. This follows by a standard rescaling argument. 
\end{remark}

  Note that the above result is indeed only relevant for functions with sufficiently small jump set, as otherwise one can choose $\omega = \Omega$,  and \eqref{eq: main estmain} trivially holds. In other words, for functions with jump set whose  measure  is comparable to the size of the domain, Theorem \ref{th: kornSBDsmall} might  not give any information. A finer result, yet restricted to the two-dimensional setting, is provided by the following  \emph{piecewise Korn-Poincar\'e inequality}.  (For the definition and properties of Caccioppoli partitions we refer to \cite[Section~4.4]{Ambrosio-Fusco-Pallara:2000}.)

\begin{proposition}[Piecewise Korn-Poincar\'e inequality]\label{th: kornpoin-sharp}
Let $\Omega \subset \R^2$ be an open, bounded set with Lipschitz boundary, and  let  $0 < \theta \le \theta_0$ for some $\theta_0$ sufficiently small. Then, there is some $C_{\theta}=C_{\theta}(\theta)>0$   such that the following holds: for  each $u \in GSBD^2(\Omega)$ we find a (finite) Caccioppoli partition $\Omega = R \cup \bigcup^{J}_{j=1} P_j$, and corresponding rigid motions $(a_j)_{j=1}^J$ such that  
\begin{align}\label{eq: kornpoinsharp2}
{\rm (i)} & \ \  \sum\nolimits_{j=1}^{J}\mathcal{H}^1\big( (\partial^* P_j \cap \Omega) \setminus J_u \big) +\mathcal{H}^1\big( (\partial^* R \cap \Omega )\setminus J_u  \big) \le \theta (\mathcal{H}^1(J_u) + \mathcal{H}^1(\partial \Omega)),\notag\\
{\rm (ii)} & \ \ \mathcal{L}^2(R)   \le \theta (\mathcal{H}^1(J_u)+ \mathcal{H}^1(\partial \Omega))^2, \ \ \ \ \ \  \mathcal{L}^2(P_j) \ge \theta^3  \ \ \ \text{for all $j=1,\ldots,J$},     \notag\\
{\rm (iii)}& \ \ \Vert u - a_j \Vert_{L^\infty(P_j)}  \le C_{\theta} \Vert  e(u) \Vert_{L^2(\Omega)} \ \ \ \quad  \text{for all $j=1,\ldots,J$}.
\end{align}
\end{proposition}

\begin{proof}
The statement is a slightly simplified version of \cite[Theorem 4.1]{FriedrichSolombrino}.  We briefly explain how the result  can be obtained therefrom.   We define $\theta_0 \le 1/c$, where $c$ is the constant from  \cite[Theorem 4.1]{FriedrichSolombrino} and apply \cite[Theorem 4.1]{FriedrichSolombrino} for $\theta/c$ in place of $\theta$.   Then, \eqref{eq: kornpoinsharp2}(i) follows from \cite[(18)(i)]{FriedrichSolombrino}, where we denote the  component $P_0$ by $R$. Item \eqref{eq: kornpoinsharp2}(ii) follows from \cite[(17)(i), (18)(ii)]{FriedrichSolombrino}, choosing $\theta_0$ sufficiently small such that $C_\Omega \ge  \theta_0$. Finally, \eqref{eq: kornpoinsharp2}(iii) follows from  \cite[(18)(iii)]{FriedrichSolombrino}, where also a corresponding Korn-type estimate has been proved.  
\end{proof}

 To control the affine mappings appearing  in the above results, we will also make use of the following elementary lemma (see \cite[Lemma~3.4]{FM}). 
		
\begin{lemma}\label{lemma: rigid motion}
Let  $G\in \mathbb{R}^{d \times d}$, $b \in \R^d$.  Let $\delta >0$, $R >0$, and let $\psi \colon \R_+\to \R_+$ be a continuous, strictly increasing function with $\psi(0) = 0$. Consider a measurable, bounded set $E \subset \R^d$  with $E \subset B_R(0)$ and $\mathcal{L}^d(E)\ge \delta$. Then there exists a continuous, strictly increasing function $\tau_\psi: \psi(\R_+) \to \R_+$ with $\tau_\psi(0) = 0$ only depending on $\delta$, $R$, and $\psi$ such that 
\begin{align*}
|G| + |b| \le \tau_\psi\Big( \fint_{E} \psi(|G\,x + b|)\, \mathrm{d}x \Big).
\end{align*}
If $\psi(t) = t^p$, $p \in [1,\infty)$, then $\tau_\psi$ can be chosen as $\tau_\psi(t) = ct^{1/p}$ for  $c=c(p,\delta,R)>0$. Moreover, there exists $c_0>0$ only depending on $\delta$,  $d$,  and $R$ such that 
\begin{align}\label{eq: estimate2}
\Vert G\, x + b \Vert_{L^\infty(B_R(0))}  \le c_0  \Vert  G\,x + b\Vert_{L^1(E)}.
\end{align}
\end{lemma}

\textbf{Approximation:} The following result is a special version of \cite[Theorem 5.1]{FinalKorn}.  (For the definition and properties of $GSBV$ functions we refer to \cite[Section~4.5]{Ambrosio-Fusco-Pallara:2000}.) 

\begin{theorem}[Approximation]\label{th: crismale-density2}
Let $\Omega  \subset  \R^d$  be a bounded Lipschitz domain, and  let $1 < p <+\infty$. Let  $u \in  GSBD^p(\Omega)$. Then there exists a sequence  $(u_n)_n \subset  GSBV^p(\Omega;\R^d) \cap L^p(\Omega;\R^d)$ such that
\begin{align*}
{\rm (i)} & \ \ u_n \to  u  \text{ in measure on } \Omega,\notag\\
{\rm(ii)} & \ \ \Vert e(u_n) - e(u) \Vert_{L^p(\Omega)} \to 0,\notag\\
{\rm (iii)} &  \ \  \mathcal{H}^{d-1}(J_{u_n} \triangle J_u) \to 0.
\end{align*}
Moreover, each $u_n$ lies in $W^{1,p}(\Omega \setminus (\Gamma_n \cup  \overline{\tilde{\omega}_n}))$, where $\Gamma_n$ is closed and the finite union of $C^1$-manifolds, and   $\tilde{\omega}_n$  is a finite union of cubes. 
\end{theorem}

\textbf{Compactness:} We recall the following compactness result in $GSBD^p$ (see \cite[Theorem 1.1]{Crismale}).

\begin{theorem}[Compactness]\label{thm: Vito compactness}
 Let $\Omega \subset \R^d$ be open and bounded.  Let $(u_n)_n$ be a sequence in $GSBD^p(\Omega)$ such that 
\begin{align}\label{eq: hp in V.C.}
\sup\nolimits_{n\in \N} \big(  \Vert e(u_n) \Vert_{L^p(\Omega)} + \mathcal{H}^{d-1}(J_{u_n}) \big) < +\infty.
\end{align}
Then, there exists a subsequence, still denoted $(u_n)_n$, such that $G_\infty :=\left\{x\in \Omega:\, |u_n(x)|\rightarrow \infty   \right\}$ has finite perimeter, and  a function  $u\in GSBD^p(\Omega)$ with $u=0$ in  $G_\infty$  such that 
\begin{align}\label{eq: main properties}
{\rm (i)} & \ \ u_n \to u \ \, \mathcal{L}^d\text{-a.e.\ on $\Omega\setminus G_\infty$},\notag\\
{\rm (ii)} & \ \  e(u_n) \rightharpoonup e(u) \textit{ in } L^p(\Omega\setminus G_\infty; \mathbb{R}^{d\times d}_{ \rm sym}),\notag\\
{\rm (iii)} & \ \ \liminf_{n\rightarrow \infty}\mathcal{H}^{d-1}(J_{u_n})\geq \mathcal{H}^{d-1}\big(J_{u}\cup  (\partial^* G_\infty\cap \Omega)\big). 
\end{align}
\end{theorem}

 A control on \eqref{eq: hp in V.C.} does in general not imply that the sequence converges in measure which is reflected by the presence of the set $G_\infty$. The latter  can be understood as the parts of the domain which are (almost) completely disconnected by the jump set $(J_{u_n})_n$ such that the functions $(u_n)_n$ can take arbitrarily large values on these pieces.  To ensure measure convergence on the \emph{entire domain}, one needs to pass to \emph{modifications} of $(u_n)_n$.

\begin{theorem}[Compactness for modifications]\label{th: comp}
Let $\Omega \subset \Omega' \subset \R^d$ be bounded Lipschitz domains. Let $(\mathcal{E}_n)_n$ be a sequence of functionals of the form \eqref{eq: basic energy} with densities satisfying \eqref{eq: general bound} and \eqref{eq: general bound2-new}. Let $(u^0_n)_n \subset W^{1,p}(\Omega';\R^d)$ be converging in $L^p(\Omega';\R^d)$ to some $u^0 \in W^{1,p}(\Omega';\R^d)$. Consider $(u_n)_n  \subset GSBD^p(\Omega')$ with $u_n = u^0_n$ on $\Omega' \setminus \overline{\Omega}$ and  $\sup_{n \in \N} \mathcal{E}_n(u_n, \Omega')   <+\infty$. 

\noindent Then, we find a  subsequence (not relabeled), modifications $(y_n)_n \subset GSBD^p(\Omega')$ satisfying 
\begin{align}\label{eq:good-en0}
\text{$y_n = u^0_n$ on $\Omega' \setminus \overline{\Omega}$,   \quad \quad \quad  $\mathcal{L}^d\big(\lbrace e(y_n) \neq e(u_n) \rbrace \big) \le \tfrac{1}{n}$}\quad \quad \quad \text{for all $n\in\N$},
\end{align}
and a limiting function  $u \in GSBD^p(\Omega')$ with $u = u^0$ on $\Omega' \setminus \overline{\Omega}$ such that $y_n \to u$ in measure on $\Omega'$ and $e(y_n) \rightharpoonup e(u)$ weakly in $L^p(\Omega';\R^{d\times d}_{\rm sym})$.

\noindent Moreover, if $p=d=2$ and  $(|\nabla u^0_n|^2)_n$ are equiintegrable, then the modifications $(y_n)_n$ can be chosen in such a way that we also have 
\begin{align}\label{eq:good-en}
\mathcal{E}_n(y_n,\Omega') \le \mathcal{E}_n(u_n,\Omega') + \tfrac{1}{n}\quad \quad \quad \text{for all $n\in\N$}.
\end{align}  
\end{theorem}


 \begin{proof}
Up to small adaptions, the case $p=d=2$ has been addressed in   \cite[Theorem~6.1, Remark~6.3]{FriedrichSolombrino}. In  \cite{FriedrichSolombrino}, only a single energy and a single boundary datum was considered, but an inspection of the proof shows that the statement can be extended to the above setting. In fact, the crucial point is that the growth conditions \eqref{eq: general bound} and \eqref{eq: general bound2-new} hold uniformly in $n \in \N$. Moreover, the property $\mathcal{L}^2\big(\lbrace e(y_n) \neq e(u_n) \rbrace \big) \le \tfrac{1}{n}$ was not  noted explicitly in \cite{FriedrichSolombrino}, but follows from the construction, see \cite[(65)-(66)]{FriedrichSolombrino}. We also refer to \cite[Theorem 3.1]{Manuel} for an analogous statement in $GSBV^p$. 

The statement for  $d\ge 2$ and $1 < p <+\infty$ can be found in \cite[Theorem 1.1]{newvito}.  The result is weaker than the one in \cite{FriedrichSolombrino} in the sense that \eqref{eq:good-en}  cannot be guaranteed. We briefly explain that \eqref{eq:good-en0}  is satisfied. Indeed, the modifications $(y_n)_n$ are obtained from $(u_n)_n$ by subtracting piecewise rigid motions $(a^n_j)_j$ associated to a fixed Caccioppoli partition $(P_j)_j$, i.e., $y_n = u_n - \sum_j a^n_j \chi_{P_j}$, see \cite[(1.4)]{newvito}. This even yields 
\begin{align}\label{eq: evenbetter!}
\mathcal{L}^d\big(\lbrace e(y_n) \neq e(u_n) \rbrace \big) =0 \quad \quad \text{for all $n \in \N$}.
\end{align}
As $u_n = u^0_n$ on $\Omega' \setminus \overline{\Omega}$ and $u^0_n \to u^0$,  \cite[(1.5b)]{newvito} allows us to  choose $a^n_j = 0$ for all components $P_j$ intersecting $\Omega' \setminus \overline{\Omega}$. This ensures $y_n = u_n = u^0_n  $ on $\Omega' \setminus \overline{\Omega}$.  
  \end{proof}

\textbf{Lower semicontinuity:}   We start with a definition  from \cite{FPM}.

\begin{definition}[Symmetric joint convexity]\label{def: symm-conv}
We say that $\tau\colon\R^d\times \R^d\times \R^d\rightarrow [0,+\infty)$ is a \emph{symmetric jointly convex function} if 
\begin{align*}
\tau(i,j,\nu)=\sup_{h\in\mathbb{N}}   (g_h(i)-g_h(j)) \cdot \nu \quad \quad \text{for all } (i,j,\nu) \in \R^d\times \R^d\times \R^d \quad  \text{with } i \neq j,
\end{align*}
where $g_h\colon \R^d\to \R^d$ is a  uniformly continuous, bounded, and  conservative vector field for every $h\in\mathbb{N}$.  
\end{definition}

The following result can be found in \cite[Theorem 5.1]{FPM}.

\begin{theorem}[Lower semicontinuity of surface integrals in $GSBD^p$]\label{thm: GSBD LSC}
Let $\tau \colon \R^d\times\R^d\times \R^d \rightarrow [0,+\infty)$  be a symmetric jointly convex function. Then, for every sequence $(u_n)_n\subset GSBD^p(\Omega)$, $p>1$, converging in measure to $u\in GSBD^p(\Omega)$, and satisfying the condition
\begin{align*}
\sup_{n\in\mathbb{N}} \big( \| e(u_n) \|_{L^p(\Omega)}  + \mathcal{H}^{d-1}(J_{u_n}) \big) <+\infty,
\end{align*}
we have that 
\begin{align*}
\int_{J_u} \tau(u^+,u^-,\nu_u)\, {\rm d}\mathcal{H}^{d-1} \leq \liminf_{n\to \infty} \int_{J_{u_n}} \tau(u_n^+,u_n^-, \nu_{u_n})\, {\rm d}\mathcal{H}^{d-1}.
\end{align*}
\end{theorem}

In the present paper, we will use that certain densities depending only on the normal are symmetric jointly convex. More precisely, we have the following  result  (see \cite[Proposition 4.11]{FPM}).

\begin{proposition}\label{prop: only nu2} 
A function $\tau\colon \R^d\times \R^d \times  \R^{d}  \to [c,+\infty)$, $c>0$, of the form  $\tau(i,j,\nu) = \psi(\nu)$ for all  $(i,j,\nu) \in \R^d\times \R^d \times  \R^d  $,  $i \neq j$,  is symmetric jointly convex  if $\psi\colon\R^d \to [0,+\infty)$ is even,   positively  $1$-homogeneous,  and convex.
\end{proposition}

We close   with a lower semicontinuity result for modifications in the setting of Theorem \ref{th: comp}.  It is a special case of \cite[Theorem 1.2]{newvito} for integrands not depending on the jump height. 
  
 \begin{lemma}[Lower semicontinuity of surface integrals for modifications]\label{lemma: vito-lsc}
 Consider the setting of Theorem \ref{th: comp} for a sequence $(u_n)_n$ with corresponding modifications $(y_n)_n$ and limiting function $u$. Furthermore, let $h\colon  \Omega'  \times \mathbb{S}^{d-1}  \to [0,+\infty)$ be a density satisfying \eqref{eq: general bound2-new} such that  $\nu \mapsto h(x,\nu)$ is even and symmetric jointly convex  for all $x \in   \Omega  $.  Moreover, suppose that $h$ is uniformly continuous on $\Omega\times \mathbb{S}^{d-1}$ and that  for $\mathcal{H}^{d-1}$-a.e.\ $x \in \partial \Omega \RRR \cap \Omega'\EEE$ it holds that    $h(x,\nu_\Omega(x)) = \lim_{n\to \infty} h(x_n,\nu_n)$ for sequences $(x_n)_n \subset \Omega$ and $(\nu_n)_n \subset \mathbb{S}^{d-1}$ with $x_n \to x$ and $\nu_n\to \nu_\Omega$, where $\nu_\Omega(x)$ denotes the outer normal at $x \in \partial \Omega \RRR \cap \Omega'\EEE$.  Then, it holds that 
      \begin{align}\label{eq: vitonew2XXX}
\int_{J_u} h(x, \nu_u) \, {\rm d}\mathcal{H}^{d-1} &\le  \liminf_{n \to \infty} \int_{J_{u_n}} h(x, \nu_{u_n}) \, {\rm d}\mathcal{H}^{d-1}.
\end{align}
\end{lemma}

\begin{proof}
 First, we reduce the problem to the case that  $h$ is continuous on the \emph{entire} set $\Omega' \times \mathbb{S}^{d-1}$. We can construct an extension  of $h$ to $\Omega'$,  called  $\tilde h$, \RRR with $\tilde{h} = h$ on $\Omega \times \mathbb{S}^{d-1}$, \EEE  which still satisfies  that  $\nu \mapsto\tilde h(x,\nu)$ is even and symmetric jointly convex  for all $x \in  \Omega' $. This can be done by a local construction,  first extending to the boundary and then  reflecting with respect to $x$ across $\partial \Omega$: clearly, the two properties which hold for fixed $x$ with respect to $\nu$ are preserved.  Moreover, we  have $\tilde{h}(x,\nu_\Omega(x)) =  h(x,\nu_\Omega(x))$ for $\mathcal{H}^{d-1}$-a.e.\ $x \in \partial \Omega \cap \Omega'$. Now it suffices to check \eqref{eq: vitonew2XXX} for $\tilde{h}$ in place of $h$. In fact, as $J_{u_n}, J_u \subset \overline{\Omega} \cap \Omega'$, see  Theorem \ref{th: comp}, it holds that $\nu_{u}(x) = \nu_\Omega(x)$ for $\mathcal{H}^{d-1}$-a.e.\ $x \in J_u \cap \partial \Omega$ (and likewise for $u_n$). 

Now, \eqref{eq: vitonew2XXX} for $\tilde{h}$   can be deduced from \cite[Theorem 1.2]{newvito}: first, observe that $\tilde{h}$ satisfies $(g_1)$--$(g_5)$ therein, where in particular $(g5)$ is trivial as $\tilde{h}$ does not depend on the jump height, and $(g4)$ follows from Theorem \ref{thm: GSBD LSC}. Then, the lower semicontinuity of the surface term follows from \cite[Equation (4.3)]{newvito}, once we clarify the role played by the Caccioppoli partition $(P_j)_j$. To this end, recall that the modifications $(y_n)_n$ are defined as $y_n = u_n - \sum_j a^n_j \chi_{P_j}$, see \cite[(1.4)]{newvito}. By a suitable choice of $(a^n_j)_j$, see \cite[below equation (4.52)]{newvito}, one can ensure that $\bigcup_j \partial^* P_j \cap \Omega \subset J_u$ up to an $\mathcal{H}^{d-1}$-negligible set. Therefore, in our setting, \cite[Equation (4.3)]{newvito} can be simplified to $\frac{{\rm d}\mu}{{\rm d}\mathcal{H}^{d-1}}(x_0) \ge \tilde{h}(x_0,\nu_u(x_0))$ for $\mathcal{H}^{d-1}$-a.e.\ $x_0 \in J_u$. This implies \eqref{eq: vitonew2XXX}.
\end{proof}

\subsection{$\Gamma$-convergence and integral representation on Sobolev functions and piecewise constant functions}\label{sec: prep intrep}

This subsection is devoted to integral representation formulas for bulk and  surface integrals, respectively.

\textbf{Bulk integrals:} Let $1<p<+\infty$, and let $f_n\colon\Omega  \times \R^{d\times d}  \to [0,+\infty)$ be a sequence of  Carath\'eodory functions satisfying  (\ref{eq: general bound}) for some $\alpha,\beta >0$. Let us consider the functionals $\mathcal{F}_n\colon   L^1(\Omega;\R^d)   \times \mathcal{A}(\Omega)\to [0,  +\infty] $ defined by
\begin{align}\label{eq: bulk part functional}
\mathcal{F}_n(u,A):=\begin{cases}
\int_A f_n(x,e(u)(x))\,{\rm d}x\quad &u\in W^{1,p}(\Omega;\R^d),\\
+\infty\quad & \mbox{otherwise}.
\end{cases}
\end{align}

\begin{proposition}\label{prop: bulk int repr}
There exists $\mathcal{F}\colon L^1(\Omega;\R^d)\times \mathcal{A}(\Omega)\to [0,+\infty)$ such that, up to subsequence (not relabeled), the functionals $\mathcal{F}_n(\cdot,A)$ $\Gamma$-converge  in the strong topology of $L^1(\Omega;\R^d)$ to $\mathcal{F}(\cdot,A)$ for every $A\in\mathcal{A}(\Omega)$. Moreover, we have that
\begin{align*}
\mathcal{F}(u,A)=\int_{A}  f_0  (x,e(u)(x))\, {\rm d}x,
\end{align*}
where for all $x \in \Omega$ and  $\xi \in \R^{d \times d}$ the density  $f_0$  is given by  
\begin{align}\label{eq: bulk repr}
 f_0(x,\xi) =  f_0(x,{\rm sym}(\xi)) = \limsup_{\rho \to 0^+} \frac{\textbf{m}^{1,p}_{\mathcal{F}}\left(\bar{\ell}_{\xi},Q_{\rho}(x)  \right)}{\rho^d}.
\end{align}
 Here, $\bar{\ell}_{\xi}$ is defined in \eqref{eq: specified not} and $\textbf{m}^{1,p}_{\mathcal{F}}$ in \eqref{eq: general minimization2}.  Moreover,   $f_0$ is a Charath\'eodory  function satisfying \eqref{eq: general bound} and   it holds that
\begin{align}\label{eq: bulk repr-new}
f_0(x,\xi) = \limsup_{\rho \to 0^+} \liminf_{n \to \infty }\frac{\textbf{m}^{1,p}_{\mathcal{F}_n}\left(\bar{\ell}_{\xi},Q_{\rho}(x)  \right)}{\rho^d} = \limsup_{\rho \to 0^+} \limsup_{n \to \infty }\frac{\textbf{m}^{1,p}_{\mathcal{F}_n}\left(\bar{\ell}_{\xi},Q_{\rho}(x)  \right)}{\rho^d}.
\end{align}
\end{proposition}
\begin{proof}
The proof is standard   (see e.g.\ \cite{BFLM, Buttazzo})   and we only provide the main steps. The proof of the  $\Gamma$-convergence part is based on standard \emph{localization techniques}, see  e.g.\ \cite[Sections 18, 19]{DalMaso:93}. The integral representation result  and   \eqref{eq: bulk repr} follow  by adapting the \emph{global method of relaxation} (see \cite[Theorem~2]{BFLM}) to our setting with a weaker growth condition from below in contrast to \cite{BFLM}, as we only control the symmetric part of the gradients. In fact,  by  using  a Korn-Poincar\'e inequality instead of the classical Poincar\'e inequality, one can follow the arguments in the proof of  \cite[Theorem 2]{BFLM}. 

Finally, \eqref{eq: bulk repr-new}  is also well-known and obtained as a consequence of $\Gamma$-convergence: given $x \in \Omega$ and $\xi \in \R^{d\times d}$,  by the $\Gamma$-liminf inequality and the coercivity of the functionals we get  
\begin{align}\label{eq: N5-1}
\textbf{m}^{1,p}_{\mathcal{F}}\left(\bar{\ell}_{\xi},Q_{\rho}(x)  \right) \le \sup_{\rho' \in  (0,\rho)} \liminf_{n\to 0}\textbf{m}^{1,p}_{\mathcal{F}_n}\left(\bar{\ell}_{\xi},Q_{\rho'}(x)  \right)
\end{align}
for all $\rho>0$. (The passage to smaller cubes $Q_{\rho'}(x) $ is necessary to ensure that boundary values at $\partial Q_\rho(x)$ are preserved under convergence.) On the other hand, the $\Gamma$-limsup inequality and the fact that boundary values can be adjusted by the fundamental estimate show
\begin{align}\label{eq: N5-2}
\limsup_{n\to 0}\textbf{m}^{1,p}_{\mathcal{F}_n}\left(\bar{\ell}_{\xi},  Q_{\rho}(x)  \right) \le \textbf{m}^{1,p}_{\mathcal{F}}\left(\bar{\ell}_{\xi},Q_{\rho}(x)  \right)
\end{align} for all $\rho>0$. By combining \eqref{eq: N5-1}--\eqref{eq: N5-2} and using \eqref{eq: bulk repr} we get \eqref{eq: bulk repr-new}. 
\end{proof}

 \textbf{Surface integrals:} We now address the surface part of functionals defined in  (\ref{eq: basic energy}). 
 Consider  the representation formula \eqref{eq: general minimization3} for functions $u=u_{x,\zeta,0,\nu}$ given in \eqref{eq: jump competitor} for   $x \in \Omega$, $\nu \in \mathbb{S}^{d-1}$, and $\zeta\in\R^d$.    First, it is instrumental to   simplify  \eqref{eq: general minimization3} whenever the density $g$ satisfies \eqref{eq: general bound2-new}. Indeed,  by \cite[Theorem A.1]{Chambolle-Giacomini-Ponsiglione:2007}, each $v \in PR(\Omega)$ can be represented as  $v(x) = \sum\nolimits_{k\in \N} (A_k\, x + b_k) \chi_{P_k}(x)$ for $x \in \Omega$, where $ (A_k)_k \subset  \R^{d\times d}_{\rm skew}$, $(b_k)_k \subset \R^d$, and $(P_k)_k$ denotes a Caccioppoli partition of $\Omega$ (see \cite[Section~4.4]{Ambrosio-Fusco-Pallara:2000}).  Therefore, in view of the fact that $g$ does not depend on the jump height (see \eqref{eq: general bound2-new}), one can check that the minimization problem \eqref{eq: general minimization3} can be restricted to functions where the partition consists of exactly two sets.   More precisely,  for $x \in \Omega$, $\nu \in \mathbb{S}^{d-1}$, and $\zeta\in\R^d$,   \eqref{eq: general minimization3} can be rewritten as
\begin{align} \label{eq: general minimization3-better}
 \mathbf{m}^{PR}_{\mathcal{E}}( u_{x,\zeta,0,\nu}  ,A) =  \mathbf{m}^{PC}_{\mathcal{E}}(\bar{u}_{x,\nu},A) :=  \inf_{v \in PC(\Omega)} \  \lbrace \mathcal{E}(v,A)\colon \ v = \bar{u}_{x,\nu} \ \text{ in a neighborhood of } \partial A \rbrace,
 \end{align}
where  $PC(\Omega) = \lbrace u \in L^1(\Omega;\R^d) \colon \, u =   e_1\chi_{T} \colon T \subset \Omega \text{ with $T$ set of finite perimeter} \rbrace$ denotes the space of piecewise constant functions attaining only the values $0$ and $e_1$,  and $\bar{u}_{x,\nu}$ is defined in \eqref{eq: specified not}.   
 
We now address $\Gamma$-convergence and integral representation of functionals  $\mathcal{G}_n\colon PC(\Omega)\times \mathcal{A}(\Omega)\to [0,+\infty)$ defined by
\begin{align}\label{eq: surface functional}
\mathcal{G}_n(u,A):=\int_{J_u\cap A} g_n(x,\nu_u(x))\, {\rm d}\mathcal{H}^{d-1}(x)
\end{align} 
for all $A\in \mathcal{A}(\Omega)$,  $u\in  PC(A)$, where $g_n:\Omega\times \mathbb{S}^{d-1}\to [0,+\infty)$ is a Borel function satisfying (\ref{eq: general bound2-new}).

\begin{proposition}\label{prop: surface int repr}
There exists $\mathcal{G}\colon PC(\Omega)\times  \mathcal{A}(\Omega)\to [0,+\infty)$ such that, up to subsequence (not relabeled), $\mathcal{G}_n(\cdot,A)$ $\Gamma$-converges with respect to the strong $L^1(\Omega;\R^d)$-convergence  to $\mathcal{G}(\cdot, A)$ for all $A\in\mathcal{A}(\Omega)$. Moreover, for all $u \in PC(\Omega)$ we have that 
\begin{align}\label{eq: jump int. rep.}
\mathcal{G}(u, A)= \int_{A\cap J_u}  g_0  (x,\nu_u(x))\, {\rm d}\mathcal{H}^{d-1}(x),
\end{align} 
 where for all $x \in \Omega$ and  $\nu \in \mathbb{S}^{d-1}$ the density $g_0$  is given by   
\begin{align}\label{eq: jump energy density}
 g_0(x,\nu):= \limsup_{\rho\to 0^+}\frac{\mathbf{m}^{PC}_{\mathcal{G}}(\bar{u}_{x,\nu},Q^\nu_{\rho}(x))}{\rho^{d-1}}.
\end{align}
Moreover,  $g_0$ satisfies (\ref{eq: general bound2-new}) and  it holds that
\begin{align}\label{eq: jump energy density-new}
 g_0(x,\nu) = \limsup_{\rho\to 0^+} \liminf_{n\to \infty}\frac{\mathbf{m}^{PC}_{\mathcal{G}_n}(\bar{u}_{x,\nu},Q^\nu_{\rho}(x))}{\rho^{d-1}} =  \limsup_{\rho\to 0^+} \limsup_{n\to \infty}\frac{\mathbf{m}^{PC}_{\mathcal{G}_n}(\bar{u}_{x,\nu},Q^\nu_{\rho}(x))}{\rho^{d-1}}.
\end{align}
\end{proposition}

\begin{proof}
 We apply   \cite[Theorem 3.2]{AmbrosioBraides} and \cite[Theorem 3]{BFLM}  to obtain the representation \eqref{eq: jump int. rep.}--\eqref{eq: jump energy density}. Finally, \eqref{eq: jump energy density-new} can be derived by $\Gamma$-convergence as explained in \eqref{eq: N5-1}--\eqref{eq: N5-2}, where we employ the fundamental estimate on $PC(\Omega)$  for the inequality analogous to \eqref{eq: N5-2},  see \cite[Lemma 4.4]{AmbrosioBraides}. (We also refer to \cite[Lemmas 6.3, 7.5]{FM} for similar arguments in $PR(\Omega)$.) 
\end{proof}

\begin{proof}[Proof of Proposition \ref{prop: prep}]
Consider a sequence of functionals $(\mathcal{E}_n)_n$ of the form \eqref{eq: basic energy} for densities $(f_n)$ and $(g_n)$ satisfying  \eqref{eq: general bound}   and \eqref{eq: general bound2-new}, respectively, Then,  with the notation in  \eqref{eq: bulk part functional} and  \eqref{eq: surface functional}, thanks  to \eqref{eq: almost the same1}--\eqref{eq: almost the same2},    for all $A \in \mathcal{A}(\Omega)$ it holds that   $\mathcal{E}_n(u,A) =\mathcal{F}_n(u,A)$ for all $u\in W^{1,p}(\Omega;\R^d)$ and $|\mathcal{E}_n(u,A) - \mathcal{G}_n(u,A)| \le \beta \mathcal{L}^d(A)$ for all   $u\in PR(\Omega)$. Now, the statement of Proposition  \ref{prop: prep} follows immediately from  \eqref{eq: bulk repr-new}, \eqref{eq: general minimization3-better}, and, \eqref{eq: jump energy density-new}. 
\end{proof}

\begin{remark}\label{rem: new rem}
For later reference, we point out that we have shown that $f=f_0$ and $g = g_0$, where the densities $f,g$ are given in Proposition \ref{prop: prep}, and $f_0,g_0$ are defined in \eqref{eq: bulk repr} and \eqref{eq: jump energy density}, respectively.   
\end{remark}

The following result   can be found in \cite[Theorem 3.1]{AmbrosioBraides2}. 

\begin{proposition}[Relaxation]\label{prop: surface int repr-relax}
Consider a continuous density $h$ satisfying (\ref{eq: general bound2-new}) and denote the corresponding functional in \eqref{eq: surface functional} by $\mathcal{S}$. Then, the relaxed functional
\begin{align*}
\bar{\mathcal{S}}(u,A):=\inf\left\{\liminf
_{n\to \infty}\mathcal{S}(u_n,A)\colon\, u_n\to u \textit{ in measure on $\Omega$}  \right\},
\end{align*}
for all $u \in PC(\Omega)$ and $A \in \mathcal{A}(\Omega)$ admits an integral representation 
\begin{align*}
\bar{\mathcal{S}}(u,A)=\int_{J_u\cap A}\bar{h}(y,\nu_u(y))\,{\rm d}\mathcal{H}^{d-1}(y),
\end{align*}
where for each $x \in \Omega$, the density $\bar{h}(x,\cdot)$ is the $BV$-elliptic envelope of $h$,   i.e.,  
\begin{align}\label{eq: jump energy density-relax}
 \bar{h}(x,\nu):= \inf_{v \in PC(Q^\nu_1) }  \Big\{ \int_{J_v \cap Q^{\nu}_1}  h(  x,  \nu_v(y)) \, {\rm d}\mathcal{H}^{d-1}(y) \colon \, v =  \bar{u}_{0,\nu} \text{ in a neighborhood of $\partial Q^\nu_1$} \Big\}
\end{align}
for all  $\nu \in \mathbb{S}^{d-1}$, where $\bar{u}_{0,\nu}$ is defined in \eqref{eq: specified not}. 
\end{proposition}

\begin{corollary}\label{cor: surf}
Suppose  that  $h$   is given  as in Proposition \ref{prop: surface int repr-relax} and suppose that $\nu \mapsto h(x,\nu)$ is even for all $x \in \Omega$.
Then, it holds that
\begin{align}\label{eq: jump energy density-relax2}
 \bar{h}(x,\nu):=  \limsup_{\rho\to 0^+}  \frac{\mathbf{m}^{PC}_{\mathcal{S}}(\bar{u}_{x,\nu}, Q^\nu_{\rho}(x)  )}{\rho^{d-1}}
\end{align}
for all $x \in \Omega$ and $\nu \in \mathbb{S}^{d-1}$. Moreover, for each $x \in \Omega$, the function $\bar{h}(x,\cdot)$ is \RRR even and \EEE symmetric jointly convex, as defined in Definition \ref{def: symm-conv}  (as a function independent of the variables $i$ and $j$).  
\end{corollary}

\begin{proof}
First, as $\bar{\mathcal{S}}$ is the $\Gamma$-limit of the constant sequence $\mathcal{S}$,  \eqref{eq: jump energy density}--\eqref{eq: jump energy density-new}  imply  
\eqref{eq: jump energy density-relax2}. We fix $x \in \Omega$, and show that $\bar{h}(x,\cdot)$ is symmetric jointly convex. First, since the functional $\bar{\mathcal{S}}$ is lower semicontinuous in $PC(\Omega)$, we get that $\bar{h}(x,\cdot)$ is convex  by \cite[Theorem 5.11]{Ambrosio-Fusco-Pallara:2000}. It is elementary to check that $\bar{h}(x,\cdot)$ is still even. Eventually, $\nu \mapsto |\nu|\bar{h}(x,\nu/|\nu|)$ can be understood as a positively $1$-homogeneous function for $\nu \in \R^d$. Then Proposition \ref{prop: only nu2}  implies that $\bar{h}(x,\cdot)$ is symmetric jointly convex  (as a function independent of the variables $i$ and $j$). 
\end{proof}

\begin{remark}[$BV$- and $BD$-elliptic envelope]\label{rem: envelope}
In view of \eqref{eq: general minimization3-better}, we observe that the characterizations given in  \eqref{eq: envelope} and  \eqref{eq: jump energy density-relax}   coincide, i.e., in the present setting the $BV$- and $BD$-elliptic envelope are the same. 
\end{remark}

\section{Compactness of $\Gamma$-convergence in  $GSBD^p$}\label{sec: compactness}

This section is devoted to the proof of Theorem \ref{th: gamma}. The result is based on the localization method of $\Gamma$-convergence along  with the recent  integral representation result \cite{Crismale-Friedrich-Solombrino}.  For the first part, the main ingredient is a  fundamental estimate  in $GSBD^p$. We  start by  addressing this tool,  and  afterwards we proceed with the proof of Theorem \ref{th: gamma}.

\subsection{Fundamental estimate in $GSBD^p$}

  We use the following convention in the whole subsection: given $A\in\mathcal{A}(\Omega)$, we may regard every $u \in GSBD^p(A)$ as a measurable function on $\Omega$, extended by $u = 0$ on $\Omega \setminus A$. We start by formulating the fundamental estimate.

 \begin{proposition}[Fundamental estimate in $GSBD^p$]\label{lemma: fundamental estimate}
Let $\Omega \subset \R^d$ be  open,  and let $1 < p <+\infty$.  Let $\eta >0$ and  let $A', A, B \in \mathcal{A}(\Omega)$ with $A' \subset \subset  A$.  Assume   that  $A \setminus A'  \subset B$, or that $B$ has Lipschitz boundary.  Then,  there exists a  function    $\Lambda\colon GSBD^p(A) \times GSBD^p(B) \to  [0,+\infty]$ which is lower semicontinuous   with respect to convergence in measure and satisfies  
\begin{align}\label{eq: Lambda0}
\Lambda(z_1,z_2) \to 0 \ \text{ whenever } \   z_1 -z_2 \to 0 \ \ \ \text{in measure on $(A\setminus A')\cap B$}
\end{align}  
such that the following holds: for every functional $\mathcal{E}$ in \eqref{eq: basic energy}  with densities $f$, $g$ satisfying  \eqref{eq: general bound}--\eqref{eq: general bound2} and for every $u \in GSBD^p(A)$, $v \in GSBD^p(B)$ there exists a function $w \in GSBD^p(A' \cup B)$ such that
\begin{align}\label{eq: assertionfund}
{\rm (i)}& \ \ \mathcal{E} ( w, A' \cup B) \le  (1+ \eta)\big(\mathcal{E}(u,A)  + \mathcal{E}(v, B) \big) + \Lambda(u,v) +\eta, \notag \\ 
{\rm (ii)}& \ \   \Vert \min\lbrace |w - u|, |w-v| \rbrace \Vert^p_{L^p(A' \cup B)}  \le \Lambda(u,v) + \eta\big(\mathcal{E}(u,A)  + \mathcal{E}(v, B)\big) + \eta,\notag \\
{\rm (iii)} & \ \  w = u \text{ on } A' \text{ and } w = v \text{ on } B \setminus A.
\end{align}
\end{proposition}

 \begin{remark}\label{rem: topo} Let us start with some comments on the result:
 
 {\normalfont

(i) A main technique  of the proof is the  Korn inequality for $GSBD^p$-functions with small jump sets, see Theorem \ref{th: kornSBDsmall}. This allows us to establish an $L^p$-control on   $\min\lbrace |w - u|, |w-v| \rbrace$ in  \eqref{eq: assertionfund}(ii).  In contrast, we point out  that each function $u,v,w$ itself might not even be integrable.

(ii) The statement is much easier to prove when \eqref{eq: Lambda0} is replaced by 
$$
\Lambda(z_1,z_2) \to 0 \ \text{ whenever } \   \Vert z_1 -z_2\Vert_{L^p((A\setminus A')\cap B)} \to 0.
$$
This corresponds to a fundamental estimate in $GSBD^p(\Omega) \cap L^p(\Omega;\R^d)$,  see \cite[Lemma 3.7]{Crismale-Friedrich-Solombrino}. The latter in turn is inspired by the original statement in $SBV^p$ formulated in \cite[Proposition~3.1]{Braides-Defranceschi-Vitali}.     The arguments there basically rely on a suitable cut-off construction between the functions $u$ and $v$.  This special case is  not enough for our purposes as it requires $L^p$-integrability of the functions which is not available in our setting.  We also point out that  a truncation argument as \cite[Lemma~3.5]{Braides-Defranceschi-Vitali} is not applicable.  As a remedy,  we use an alternative technique, based on Theorem \ref{th: kornSBDsmall}.
 
(iii)    In contrast to \cite[Lemma 3.7]{Crismale-Friedrich-Solombrino}, we  need the condition $A \setminus A'  \subset   B$,  unless one assumes that $B$ has Lipschitz boundary.  

}
 \end{remark}

 \begin{proof}[Proof of Proposition \ref{lemma: fundamental estimate}]
 We  begin with a short outline of the proof.   We start by partitioning the set  $A\setminus A'$  into `layers'  where we will eventually `join' $u$ and  $v$ by a cut-off construction. These layers are additionally covered by a collection of small  cubes  (Step 1). In each of these cubes, we apply a Korn inequality in $GSBD^p$ (Theorem \ref{th: kornSBDsmall}) on the function $u-v$ (Step 2), and we analyze the corresponding expectional sets (Step 3) and rigid motions (Step 4). Based on this, we introduce modifications of $u$ and $v$ such that their difference lies in $L^p$ (Step 5). Then, we apply a cut-off construction similar to  \cite[Lemma 3.7]{Crismale-Friedrich-Solombrino} or  \cite[Proposition 3.1]{Braides-Defranceschi-Vitali} (Step 6) and obtain the desired function $w$ satisfying \eqref{eq: assertionfund} (Step 7).  

 We will focus on the case where $A \setminus A'  \subset B$. At the end, we will indicate the minor changes to be done when this is not assumed, but instead $B$ has Lipschitz boundary.  To account for this alternative assumption, along the proof  we will write $(A\setminus A')\cap B$ in place of $A\setminus A'$ several times, although the intersection with $B$ is redundant under condition $A \setminus A' \subset B$. 

 \noindent \emph{Step 1: Preliminaries.} Let $\eta >0$, and let the sets  $A', A, B \in \mathcal{A}(\Omega)$ with $A' \subset \subset  A$  and $A \setminus A'   \subset  B$ be  given. In this step, we introduce several parameters and coverings that  we will use throughout the proof.   We fix $k \in \N$ sufficiently large such that 
 \begin{align}\label{eq: k0}
  \frac{12^{p+1} \beta^2 d \bar{c} \, (1+\mathcal{L}^d(A \setminus A'))}{k  \alpha^2 }\le \eta,
\end{align}
where $\bar{c} \ge 1$ denotes the constant from Remark \ref{rem: Korn-scaling}, $d$ the dimension, and  $\alpha, \beta$ are defined in \eqref{eq: general bound}.

Let $A_1,\ldots,A_{k+1}$ be open subsets of $\R^d$ with $A' \subset \subset A_1 \subset \subset \ldots \subset \subset A_{k+1} \subset \subset A$. We also define further open sets $A_i \subset  \subset A_i^+ \subset \subset A_{i+1}^- \subset \subset A_{i+1}$ for $i=1,\ldots,k$, and let 
\begin{align}\label{eq:ST}
 S_i = (A_{i+1} \setminus \overline{A_i})  \cap B, \quad \quad \quad \quad T_i =   (A^-_{i+1} \setminus \overline{A_i^+})  \cap B 
 \end{align}
 for $i=1,\ldots,k$.   As $A \setminus A'  \subset  B$, we get $T_i \subset \subset S_i \subset A \cap B$ for $i=1,\ldots,k$. (The intersection in \eqref{eq:ST}  with $B$ is redundant, but added in order to  highlight that the sets are contained in $A \cap B$.)    Moreover, let  $\varphi_i\in C^\infty(\R^d;  [0,1]  )$ with $\varphi_i=1$ on $A_i^+$ and $\varphi_i = 0 $ on $\R^d \setminus {A_{i+1}^-}$, i.e., 
\begin{align}\label{eq: varohi}
\lbrace 0 < \varphi_i < 1 \rbrace \subset T_i.
\end{align} 

Define  $\psi\colon \R_+ \to [0,1)$ by $\psi(t) := \frac{t}{1+t}$ for $t \ge 0$ and observe that
\begin{align}\label{eq: mes-conv}
\text{$u_n \to u$ in measure on $U \in \mathcal{A}(\Omega)$ if and only if $\int_U \psi(|u_n-u|)\,{\rm d}x \to 0$}.
\end{align}
  We apply  Lemma \ref{lemma: rigid motion} for $\delta = 1/2$ and $R = \sqrt{d}$, and let  $\tau_\psi\colon  [0,1)  \to \R_+$  be the continuous, strictly increasing function with $\tau_\psi(0) = 0$.  As $\tau_\psi$ is uniformly continuous on $[0, 1/2  ]$, we can choose  $\lambda \in  (0,+\infty)$   such that
\begin{align}\label{eq:crazy choice}
{\rm (i)} & \ \   2   \lambda <  1/2, \notag \\  
{\rm (ii)}&\ \ 12^p  \beta \, (1+ \max_{i=1,\ldots,k}  \Vert \nabla \varphi_i \Vert_\infty^p)  \, \mathcal{L}^d(A \setminus A') \max_{t \in [0,1/2]} \big| \tau^p_\psi( t +  2  \lambda)   - \tau^p_\psi( t ) \big| \le \eta/2,
\end{align} 
where here and the following $\tau_\psi^p(\cdot) := (\tau_\psi(\cdot))^p$. The constant $\lambda$ will become relevant later for the definition of $\Lambda$, see in particular \eqref{eq: lambda-def} and \eqref{eq: lambda-def-2}.  Recalling $T_i \subset \subset S_i$,  we pick a further  constant $\rho \in (0,1) $  sufficiently small  such that  
\begin{align}\label{eq: rho choice}
{\rm (i)}  \ \ 2^p \lambda^{-p}  \bar{c}    \rho^{p-1} \le \lambda, \quad  \ \ \           {\rm (ii)} \ \           \rho \le  \min_{i=1\ldots,k} \,  \big\{   {\rm dist}(T_i,S_i),   \Vert  \nabla  \varphi_i\Vert_\infty^{-1} \big\}.
\end{align}
 For each $i=1\ldots,k$, we cover $T_i$ up to set of $\mathcal{L}^d$-negligible measure with a finite number of pairwise disjoint open cubes $\mathcal{Q}^i := (Q^i_j)_j$ with centers $(x^i_j)_j \subset \rho \mathbb{Z}^d \cap T_i$ and sidelength $\rho$.   In view of  \eqref{eq:ST} and  \eqref{eq: rho choice}(ii),  we get
\begin{align}\label{eq: Ball properties}
T_i \subset \bigcup\nolimits_j \overline{Q^i_j} \subset \subset S_i\subset A\cap B.
\end{align} 
We now fix $u \in GSBD^p(A)$ and $v \in  GSBD^p(B)$, where $u$ and $v$ are extended by zero outside $A$ and $B$, respectively. 
 In the following proof, we will assume without loss of generality that
\begin{align}\label{eq: closeuv}
 2  \rho^{-d} \int_{(A\setminus A^\prime)\cap B}\psi(|u-v|) \,  {\rm d}x \le  1/2. 
\end{align}
Indeed, if this  does not hold, we simply set $\Lambda(u, v)=+\infty$ and $w=u\chi_A + v\chi_{B\setminus A}$. Then,  \eqref{eq: assertionfund} is clearly  satisfied.  In the sequel, we will also assume that  $\alpha \le 1 \le \beta$ which is not restrictive.

\noindent \emph{Step 2: Classification of cubes and Korn's inequality.} In this step  of the proof, we distinguish `good and bad cubes' and apply Theorem \ref{th: kornSBDsmall} on good cubes. Corresponding exceptional sets and rigid motions are analyzed in Step 3 and Step 4 below, respectively.  Fix $i =1,\ldots,k$. We define the collection of \emph{bad cubes}, denoted by  $\mathcal{Q}^i_{\rm bad}$, as the subset of cubes $(Q^i_j)_j$ with the property 
\begin{align}\label{eq: bad cubes}
(2\bar{c})^{\frac{d-1}{d}} \mathcal{H}^{d-1}\big( (J_u\cup J_v) \cap Q^i_j\big)    +  \Vert e(u) \Vert^p_{L^p(Q^i_j)} + \Vert e(v) \Vert^p_{L^p(Q^i_j)} \ge \rho^{d-1}.
\end{align}
 We also let $\mathcal{Q}^i_{\rm good} :=  \mathcal{Q}^i \setminus \mathcal{Q}^i_{\rm bad}$.  For each $Q^i_j \in \mathcal{Q}^i_{\rm good}$, we apply Theorem \ref{th: kornSBDsmall} for $u-v$  to obtain  exceptional  sets $\omega^{i}_j$ and rigid motions $a^{i}_j$ such that 
\begin{align}\label{eq: R2main-NNN}
{\rm (i)} & \ \ \mathcal{H}^{d-1}(\partial^* \omega^{i}_j) \le \bar{c}\,\mathcal{H}^{d-1}((J_u \cup J_v) \cap Q^i_j), \ \ \ \ \mathcal{L}^d(\omega^{i}_j)\le \bar{c}\big(\mathcal{H}^{d-1}\big( (J_u \cup J_v) \cap Q^i_j)\big)\big)^{d/(d-1)},\notag\\
{\rm (ii)} & \ \  \Vert u - v - a^{i}_j \Vert^p_{L^{p}(Q^i_j \setminus \omega^{i}_j)}\le \bar{c}\rho^p \, \Vert e(u-v) \Vert^p_{L^p(Q^i_j)}.
\end{align}
(See also Remark \ref{rem: Korn-scaling}.) For each $Q^i_j\in \mathcal{Q}^i_{\rm bad}$, we define the exceptional set $\omega^i_j$ simply by  
\begin{align}\label{eq: exceptional sets}
\omega^i_j  := Q^i_j.
\end{align}

\noindent \emph{Step 3: Korn's inequality and exceptional sets.}
We now show that for each cube $Q^i_j$ we have 
\begin{align}\label{eq: exceptional sets-2}
\mathcal{H}^{d-1}(\partial^* \omega^{i}_j) \le C \big(\mathcal{H}^{d-1}(J_u \cap Q^i_j) + \mathcal{H}^{d-1}(J_v \cap Q^i_j) + \Vert e(u) \Vert^p_{L^p(Q^i_j)} + \Vert e(v) \Vert^p_{L^p(Q^i_j)} \big),
\end{align}
 where  $ C:= 4d\bar{c}$,  and $\bar{c} \ge 1$ denotes again the constant from Remark \ref{rem: Korn-scaling}.  Indeed, for $Q^i_j \in \mathcal{Q}^i_{\rm good}$  this follows directly from \eqref{eq: R2main-NNN}(i). For $Q^i_j \in \mathcal{Q}^i_{\rm bad}$ instead,  we first observe that $\mathcal{H}^{d-1}(\partial^* \omega^i_j) = 2d\rho^{d-1}$, see \eqref{eq: exceptional sets}. Then, by   \eqref{eq: bad cubes} we obtain
\begin{align*}
\frac{1}{ 2d}\mathcal{H}^{d-1}(\partial^* \omega^i_j)&   =  \rho^{d-1} \le  (2\bar{c})^{\frac{d-1}{d}} \mathcal{H}^{d-1}\big( (J_u\cup J_v) \cap Q^i_j\big)  +  \Vert e(u) \Vert^p_{L^p(Q^i_j)} + \Vert e(v) \Vert^p_{L^p(Q^i_j)},
\end{align*}
from which \eqref{eq: exceptional sets-2} follows. Moreover, for $Q^i_j \in \mathcal{Q}^i_{\rm good}$, by  \eqref{eq: bad cubes} and \eqref{eq: R2main-NNN}(i)  we get 
\begin{align}\label{eq: volume good cubes}
\mathcal{L}^d(\omega^i_j) &\le  \bar{c}\big(\mathcal{H}^{d-1}( (J_u \cup J_v) \cap Q^i_j)\big)^{d/(d-1)}    \le \frac{1}{2} \rho^d = \frac{1}{2}\mathcal{L}^{d}(Q^i_j).
\end{align}

\noindent \emph{Step 4: Korn's inequality and rigid motions.} Recall that  $\tau_\psi\colon  [0,1) \to \R_+$  is the function obtained by  Lemma \ref{lemma: rigid motion} for $\delta=1/2$ and $R =  \sqrt{d}$.  For each $z_1 \in  GSBD^p(A)$ and $z_2 \in  GSBD^p(B)$ we define 
\begin{align}\label{eq: lambda-def}
\Lambda_*(z_1,z_2) = 2^p \, \tau^p_\psi\Big(   2  \rho^{-d} \int_{(A\setminus A^\prime)\cap B} \psi(|z_1 - z_2|)\, \mathrm{d}x     +    2  \lambda\Big),
\end{align}
whenever $ 2  \rho^{-d} \int_{(A\setminus A^\prime)\cap B}\psi(|u-v|) \, {\rm d}x \le  1/2 $ and $\Lambda_*(z_1,z_2) = + \infty$ else. (Note that this is well defined by \eqref{eq:crazy choice}(i).  It will lead to a definition of $\Lambda$  in \eqref{eq: lambda-def-2}   that it is consistent with the definition below \eqref{eq: closeuv}.)

The goal of this step is to prove the estimate 
\begin{align}\label{eq:poinca}
\Vert a^{i}_j \Vert^p_{L^p(Q^i_j \setminus \omega^i_j)} \le \mathcal{L}^d(Q^i_j) \,\Lambda_*(u,v)\quad \quad \quad \text{for all } Q^i_j \in \mathcal{Q}^i. 
\end{align}
By definition of $\omega^i_j$, see \eqref{eq: exceptional sets}, it is clear that this needs to be checked only for cubes in $\mathcal{Q}^i_{\rm good}$.  To this end, we first note by \eqref{eq: volume good cubes} that 
\begin{align}\label{eq: rhokdef}
\mathcal{L}^d(Q^i_j \setminus \omega^i_j) \ge \frac{1}{2}\mathcal{L}^d(Q^i_j) = \frac{1}{2}\rho^d.
\end{align}
We write the rigid motions $a^{i}_j$ as $a^{i}_j(x) = A^{i}_j\,x + b^{i}_j$, and denote by $x^i_j$ the center of the cube $Q^i_j$. We can apply Lemma \ref{lemma: rigid motion} for $\delta=1/2$, $R= \sqrt{d}$,  $G = \rho A^i_j$, $b = b^{i}_j + A^i_j\, x^i_j$, and $E = \rho^{-1}(Q^i_j \setminus \omega^i_j - x^i_j)$    to find
\begin{align}\label{eq: estimate1---}
\rho|A^{i}_j| + |b^{i}_j + A^i_j\, x^i_j | & \le  \tau_\psi\Big( \fint_{E} \psi(|G\,x + b|)\, \mathrm{d}x \Big)  = \tau_\psi\Big(  \fint_{Q^i_j \setminus \omega^i_j} \psi(|a^{i}_j |)\, \mathrm{d}x \Big),
\end{align}    
 where in the second step we used a change of variables.  
   We now estimate the integral on the right hand side of  \eqref{eq: estimate1---}. By the triangle inequality, the monotonicity of $\psi$,  and the subadditivity of $\psi$ we get   
$$\int_{Q^i_j \setminus \omega^i_j} \psi(|a^{i}_j| ) \, \mathrm{d}x  \le 
   \int_{Q^i_j \setminus \omega^i_j} \psi(|u -v - a^{i}_j | ) \, \mathrm{d}x  + \int_{Q^i_j \setminus \omega^i_j} \psi(|u  - v| ) \, \mathrm{d}x. $$
    Note that $\psi(t)  = \frac{t}{1+t}  \le \lambda + \lambda^{-p}t^p$ for all $t \ge 0$. Therefore, we get    
 \begin{align}\label{eq: first step-a}
  \int_{Q^i_j \setminus \omega^i_j} \psi(|a^{i}_j| ) \, \mathrm{d}x    & \le  \lambda^{-p}     \Vert u  - v- a^{i}_j \Vert^p_{L^p(Q^i_j \setminus \omega^{i}_j)}  + \lambda\mathcal{L}^d(Q^i_j\setminus \omega^i_j) +   \int_{Q^i_j}   \psi(|u-v|)\, \mathrm{d}x . 
 \end{align}
 For the first addend, we further compute by  \eqref{eq: bad cubes} and \eqref{eq: R2main-NNN}(ii) that
  \begin{align*}
 \Vert u  - v- a^{i}_j \Vert^p_{L^p(Q^i_j \setminus \omega^{i}_j)}  & \le \bar{c} \rho^p \Vert e(u-v)\Vert_{L^p(Q^i_j)}^p \le   2^{p-1}  \bar{c}\rho^p\Big(   \Vert e(u) \Vert^p_{L^p(Q^i_j)} + \Vert e(u) \Vert^p_{L^p(Q^i_j)}\Big)  \\
  & \le   2^{p-1}  \bar{c}\rho^p \rho^{d-1},
 \end{align*}
 where we used that $Q^i_j \in \mathcal{Q}^i_{\rm good}$. This along with \eqref{eq: rho choice}(i), \eqref{eq: rhokdef}, and  \eqref{eq: first step-a} yields
\begin{align}\label{eq: first step-a3}
  \fint_{Q^i_j \setminus \omega^i_j} \psi(|a^{i}_j| ) &\le  \frac{1}{\rho^d/2}    \lambda^{-p} 2^{p-1}  \bar{c} \, \rho^{p + d-1}  + \lambda +    \frac{1}{\rho^d/2}     \int_{Q^i_j}   \psi(|u-v|) \le 2 \rho^{-d}  \int_{Q^i_j}   \psi(|u-v|) + 2\lambda.
\end{align}
A simple calculation also yields
\begin{align}\label{eq: first step-a2}
\Vert a^{i}_j \Vert^p_{L^p(Q^i_j \setminus \omega^i_j)}  &\le  \mathcal{L}^d(Q^i_j) \sup_{x \in Q^i_j} |A^i_j\,x + b^i_j|^p  =   \mathcal{L}^d(Q^i_j) \sup_{x \in  Q_{\rho}} |A^i_j\,(x+x^i_j) + b^i_j|^p \notag \\ &  \le \mathcal{L}^d(Q^i_j) \,  2^{p-1} \, \big((|A^i_j|\rho)^p  + |b^{i}_j + A^i_j\, x^i_j |^p   \big).  
\end{align}
Now we obtain \eqref{eq:poinca} by using that $\tau_\psi$ is increasing and by  combining \eqref{eq: estimate1---}, \eqref{eq: first step-a3}, and \eqref{eq: first step-a2}.

\noindent \emph{Step 5: Modifications of  $u$.}  
 In this step  of the proof, we will modify the function $u$  on $S_i$  (recall \eqref{eq:ST})  such that its  difference to $v$ restricted to  $T_i$ lies in  $L^p$. For each  $i=1,\ldots,k$, we define  $\omega^i = \bigcup_j \omega^i_j$,  and we  note that $\omega^i \subset \bigcup_j Q^i_j \subset \subset S_i$ by \eqref{eq: Ball properties}. We    introduce the function 
\begin{align}\label{eq: defu,v}
u_i = u \chi_{A \setminus \omega^i} + v\chi_{\omega^i} \in GSBD^p(A).
\end{align}
We now prove the estimates
\begin{align}\label{eq: modified u,v}
{\rm (i)} & \ \  \mathcal{E}(u_i, S_i)  \le  \big( 1 + C\beta   \alpha^{-1}   \big) \big( \mathcal{E}(u,S_i) +   \mathcal{E}(v,S_i)\big),\notag\\
{\rm (ii)} & \ \  \Vert u_i- v\Vert^p_{L^p(T_i)} \le   C 4^{p-1} \alpha^{-1} \rho^p \big( \mathcal{E}(u,S_i) +   \mathcal{E}(v,S_i)\big) + 2^{p-1} \mathcal{L}^d(S_i) \Lambda_*(u,v),
\end{align}
 where $C = 4d\bar{c}$ is the constant of \eqref{eq: exceptional sets-2}.  To prove (i), we first use \eqref{eq: general bound2}  to get 
\begin{align}\label{eq: first-to-ui}
\mathcal{E}(u_i,  S_i)  & \le \mathcal{E}(u,  S_i)  + \mathcal{E}(v,S_i) +  \beta\mathcal{H}^{d-1}(\partial^* \omega^i).
\end{align}
By  \eqref{eq: Ball properties}  and  \eqref{eq: exceptional sets-2}  we then compute
\begin{align*}
\beta\mathcal{H}^{d-1}(\partial^* \omega^i) & \le   \beta \sum\nolimits_j  \mathcal{H}^{d-1}(\partial^* \omega^{i}_j)  \\
& \le     C\beta \sum\nolimits_j \Big( \mathcal{H}^{d-1}(J_u \cap Q^i_j) + \mathcal{H}^{d-1}(J_v \cap Q^i_j) + \Vert e(u) \Vert^p_{L^p(Q^i_j)} + \Vert e(v) \Vert^p_{L^p(Q^i_j)} \Big)      \\
& \le  C\beta  \Big(\mathcal{H}^{d-1}(J_u \cap S_i) + \mathcal{H}^{d-1}(J_v \cap S_i) + \Vert e(u) \Vert^p_{L^p(S_i)} + \Vert e(v) \Vert^p_{L^p(S_i)}\Big),
\end{align*}
 where we used that the cubes are pairwise disjoint.  Then, (i) follows from  \eqref{eq: first-to-ui} and  the lower  bound in  \eqref{eq: general bound}--\eqref{eq: general bound2}.   We now address \eqref{eq: modified u,v}(ii).   To this end, for each cube $Q^i_j$, by using \eqref{eq: R2main-NNN}(ii) and   \eqref{eq:poinca} we get
 \begin{align*}
\Vert u - v \Vert^p_{L^p(Q^i_j \setminus \omega^i_j)} &\le 2^{p-1} \Vert u - v - a^{i}_j \Vert^p_{L^p(Q^i_j \setminus \omega^{i}_j)} +  2^{p-1} \Vert a^{i}_j \Vert^p_{L^p(Q^i_j \setminus \omega^{i}_j)}   \\ & \le 
 2^{p-1} \bar{c}\, \rho^p\Vert e(u-v)\Vert^p_{L^p(Q^i_j)} +   2^{p-1} \Vert a^{i}_j \Vert^p_{L^p(Q^i_j \setminus \omega^{i}_j)}  \\
 & \le 4^{p-1} \bar{c}\,\rho^p \Big( \Vert e(u) \Vert^p_{L^p(Q^i_j)} + \Vert e(v) \Vert^p_{L^p(Q^i_j)} \Big) + 2^{p-1}\mathcal{L}^d(Q^i_j) \, \Lambda_*(u,v).
 \end{align*} 
Then, summing over all cubes and using   \eqref{eq: Ball properties} as well as \eqref{eq: defu,v} we derive
 \begin{align*}
\Vert u_i- v\Vert^p_{L^p(T_i)} &\le  \sum\nolimits_j \Vert u - v \Vert^p_{L^p(Q^i_j \setminus \omega^i_j)} \\&\le  4^{p-1}  \bar{c} \rho^p\Big( \Vert e(u) \Vert^p_{L^p(S_i)} + \Vert e(v) \Vert^p_{L^p(S_i)} \Big) + 2^{p-1} \mathcal{L}^d(S_i) \,\Lambda_*(u,v).
 \end{align*}
In view of  \eqref{eq: general bound}  and $C \ge \bar{c}$,  this shows \eqref{eq: modified u,v}(ii), and concludes this step of the proof.

\noindent \emph{Step 6: Cuff-off construction.} We now perform a cut-off to join the functions $u_i$ and $v$.  Recalling  $u_i$ defined in \eqref{eq: defu,v} and the functions $\varphi_i$ introduced before \eqref{eq: varohi}, we  define  the functions  $w_i := \varphi_i u_i + (1-\varphi_i)v \in GSBD^p(A' \cup B)$ for $i=1,\ldots,k$. By \eqref{eq:ST}, $\omega^i \subset \subset S_i$, and  \eqref{eq: defu,v} we get 
\begin{align}\label{eq: wi}
\mathcal{E}(w_i,A' \cup B) &\le \mathcal{E}\big(u_i, (A' \cup B) \cap A_i^+\big) + \mathcal{E}\big(v, B \setminus {A_{i+1}^-}\big)
+ \mathcal{E}(w_i,S_i) \notag \\
&  \le \mathcal{E}(u,A) + \mathcal{E}(u_i,S_i) + \mathcal{E}(v,B) + \mathcal{E}(w_i,S_i).  
\end{align}
The  second   term has already been estimated in \eqref{eq: modified u,v}(i). We now address the last term. By using the upper bounds in \eqref{eq: general bound}--\eqref{eq: general bound2} we  compute  (by $\odot$ we denote the symmetrized vector product)
\begin{align*}
\mathcal{E}(w_i,S_i) &  \le   \int_{S_i} \beta\big(1+|e(w_i)|^p\big)\, {\rm d}x + \beta \mathcal{H}^{d-1}(J_{w_i} \cap S_i) \\
& \le \beta\mathcal{L}^d(S_i) + \beta  \int_{S_i} \big|\varphi_i e(u_i) + (1-\varphi_i)e(v) + \nabla \varphi_i \odot (u_i-v) \big|^p\, {\rm d}x + \beta \mathcal{H}^{d-1}(J_{w_i} \cap S_i) \\
& \le  \beta\mathcal{L}^d(S_i) + 3^{p-1}\beta  \int_{S_i} \big(|e(u_i)|^p + |e(v)|^p + |\nabla \varphi_i|^p |u_i-v|^p \big)\,  {\rm d}x  \\
& \ \ \   + \beta \big(\mathcal{H}^{d-1}(J_{u_i}  \cap S_i) +  \mathcal{H}^{d-1}(J_{v}  \cap S_i) \big).  
\end{align*}
Using the lower bounds \eqref{eq: general bound}--\eqref{eq: general bound2}, \eqref{eq: varohi},  and \eqref{eq: modified u,v}(i) we then get
\begin{align*}
\mathcal{E}(w_i,S_i)& \le 3^{p-1}\beta\alpha^{-1} \big( \mathcal{E}(u_i,S_i) + \mathcal{E}(v,S_i)  \big)  + 3^{p-1}\beta \Vert \nabla \varphi_i \Vert_\infty^p \Vert u_i- v\Vert^p_{L^p(T_i)} +   \beta\mathcal{L}^d(S_i)\\
 &\le 3^{p-1}\beta\alpha^{-1}(2+C\beta\alpha^{-1}) \big( \mathcal{E}(u,S_i) + \mathcal{E}(v,S_i)  \big)  + 3^{p-1}\beta \Vert \nabla \varphi_i \Vert_\infty^p \Vert u_i- v\Vert^p_{L^p(T_i)} +   \beta\mathcal{L}^d(S_i). 
\end{align*}
By   \eqref{eq: rho choice}(ii),  \eqref{eq: modified u,v},  \eqref{eq: wi}, and the fact that $C\beta   \alpha^{-1}  \ge \beta   \alpha^{-1}   \ge 1$  we thus derive  after some computation 
\begin{align}\label{eq: last step}
\mathcal{E}(w_i,A' \cup B) & \le \mathcal{E}(u,A) +   \mathcal{E}(v,B) +  2C\beta   \alpha^{-1}   \big( \mathcal{E}(u,S_i) +   \mathcal{E}(v,S_i)\big)  +   \mathcal{E}(w_i,S_i)\notag\\ 
& \le \mathcal{E}(u,A) +   \mathcal{E}(v,B) +  (2C +  C  3^{p}  + C12^{p-1}) (\beta \alpha^{-1})^2  \big( \mathcal{E}(u,S_i) + \mathcal{E}(v,S_i)  \big) \notag \\ & \ \ \  + 6^{p-1} \beta \Vert \nabla \varphi_i \Vert_\infty^p \mathcal{L}^d(S_i) \Lambda_*(u,v)  +   \beta\mathcal{L}^d(S_i).
\end{align}

\noindent \emph{Step 7: Definition of $w$ and $\Lambda$.} We finally define $w$ and $\Lambda$, and we show estimate \eqref{eq: assertionfund}. Recalling that the sets $(S_i)_{i=1}^k$ are pairwise disjoint and contained in  $(A \setminus A') \cap B  \subset A \cap B $, see \eqref{eq:ST},   we can choose $i_0 \in \lbrace 1, \ldots, k \rbrace$ such that 
\begin{align}\label{eq: miscela}
 \mathcal{E}(u,S_{i_0}) + \mathcal{E}(v,S_{i_0}) + \frac{1}{2}\mathcal{L}^d(S_{i_0})& \le \frac{1}{k} \sum\nolimits_{i=1}^k  \big( \mathcal{E}(u,S_i) + \mathcal{E}(v,S_i) + \frac{1}{2}\mathcal{L}^d(S_i) \big) \notag \\&  \le  \frac{1}{k}  \Big( \mathcal{E}(u,A) + \mathcal{E}(v,B) + \frac{1}{2}\mathcal{L}^d\big( A \setminus A' \big)\Big).
\end{align}
 Recall $C = 4d\bar{c}$.  As $\alpha \le 1$, we have   $12^{p+1}d\bar{c}\beta/(2\alpha^2) \ge 1$.  Then, by \eqref{eq: k0}, \eqref{eq: last step}, and \eqref{eq: miscela} the function  $w := w_{i_0}$ satisfies   
\begin{align}\label{eq: penultimate}
\mathcal{E}(w,A' \cup B )&\le  \mathcal{E}(v,A) +   \mathcal{E}(v,B) +  \frac{12^{p+1}d\bar{c} \beta^2}{\alpha^2} \Big( \mathcal{E}(u,S_{i_0}) + \mathcal{E}(v,S_{i_0}) + \frac{1}{2} \mathcal{L}^d(S_{i_0})  \Big)    + M \Lambda_*(u,v)\notag \\
& \le  (1+\eta)\big( \mathcal{E}(v,A) +   \mathcal{E}(v,B)\big)  + M \Lambda_*(u,v) + \eta/2,
\end{align}
where for shorthand we have defined $M:= 6^{p-1} \beta  (1+ \max_{i=1,\ldots,k} \Vert \nabla \varphi_i \Vert_\infty^p ) \,  \mathcal{L}^d( A \setminus A'  )$.  In a similar fashion,  as $\beta \ge 1$,  by analogous estimates, taking \eqref{eq: modified u,v}(ii),  \eqref{eq: k0}  and  \eqref{eq: miscela} into account,  we get  
\begin{align}\label{eq: difference for later}
\Vert u_{i_0}- v\Vert^p_{L^p(T_{i_0})} \le \eta\big( \mathcal{E}(v,A) +   \mathcal{E}(v,B)\big)  + M \Lambda_*(u,v) + \eta/2.
\end{align}
  We let 
\begin{align}\label{eq: lambda-def-2} 
\Lambda(z_1,z_2) = M 2^p \tau^p_\psi\Big(  2  \rho^{-d} \int_{(A\setminus A^\prime)\cap B} \psi(|z_1 - z_2|)\, \mathrm{d}x \Big), 
\end{align}
whenever $ 2  \rho^{-d} \int_{(A\setminus A^\prime)\cap B}\psi(|u-v|)\,  {\rm d}x  \le  1/2 $ and $\Lambda(z_1,z_2) = + \infty$ else. (Note that this  is consistent with the definition below \eqref{eq: closeuv}.) 
 Then,   $\Lambda$ is  lower semicontinuous by Fatou's lemma and the fact that $\tau_\psi$ is continuous and increasing. Moreover,  in view of \eqref{eq: mes-conv},  we easily check that \eqref{eq: Lambda0} is satisfied since $\tau_\psi(0)=0$, see Lemma \ref{lemma: rigid motion}.
Eventually,  by  \eqref{eq:crazy choice}(ii) and \eqref{eq: lambda-def} we find  
\begin{align}\label{eq: and this}
|\Lambda(u,v) - M\Lambda_*(u,v)| \le \eta/2.
\end{align}
 This along with \eqref{eq: penultimate} yields \eqref{eq: assertionfund}(i).  Recalling that $w = \varphi_{i_0} u_{i_0} + (1-\varphi_{i_0})v$ and that $\lbrace u \neq u_{i_0} \rbrace \subset S_{i_0} \subset  A \setminus A' $ as well as $\varphi_{i_0} =1$ on $A'$ and  $\varphi_{i_0} =0$ outside $A$, we get  \eqref{eq: assertionfund}(iii). Moreover, in view of  \eqref{eq: defu,v}, $w=v$ on $\omega^{i_0}$,  and the fact that $\lbrace 0 < \varphi_{i_0} <1 \rbrace \subset T_{i_0} \subset S_{i_0}$  (see \eqref{eq: varohi}),  we compute
\begin{align*}
 \Vert \min\lbrace |w - u|, |w-v| \rbrace \Vert_{L^p(A' \cup B)} & =   \Vert \min\lbrace |w - u_{i_0}|, |w-v| \rbrace \Vert_{L^p(S_{i_0} \setminus \omega^{i_0})} \le \Vert u_{i_0} - v  \Vert_{L^p(T_{i_0})}. 
\end{align*}
 Finally, by \eqref{eq: difference for later} and \eqref{eq: and this} we get that  \eqref{eq: assertionfund}(ii) holds true. This concludes the proof  whenever $A \setminus A'  \subset  B$.
 
 If $B$ has Lipschitz boundary, the condition $A \setminus A'  \subset  B$ is dispensable. In fact, we can still cover each set $T_i$  defined in \eqref{eq:ST}  with cubes of sidelength $\rho$, see Step 1. These cubes, however, are not necessarily contained in $B$. Still, we can apply Korn's inequality on the cubes in Step 2 by extending $v = 0$ outside $B$, at the expense of an additional term $ C  \mathcal{H}^{d-1}(\partial B \cap Q^i_j)$ on the right hand side of the estimate on $\mathcal{H}^{d-1}(\partial^* w^i_j)$,  see \eqref{eq: exceptional sets-2}. This implies an additional addend $ C\beta \mathcal{H}^{d-1}(\partial B \cap S_i)$ on the right hand side of \eqref{eq: modified u,v}(i). Eventually, in \eqref{eq: miscela}, this leads to an additional addend on the right hand side of the form $\frac{1}{k}\mathcal{H}^{d-1}(\partial B \cap (A \setminus A')   )$. This can be made arbitrarily small for $k$ sufficiently large.   
\end{proof}

\subsection{Proof of Theorem \ref{th: gamma}}

We consider a sequence of functionals $(\mathcal{E}_n)_n$  of the  form \eqref{eq: basic energy}. We start by proving some  properties of the $\Gamma$-liminf and $\Gamma$-limsup with respect to the topology of the convergence in measure. To this end, we define 
\begin{align}\label{eq: liminf-limsup}
\mathcal{E}'(u,A)&:=\Gamma-\liminf_{n \to \infty} \mathcal{E}_n(u,A)   = \inf \big\{ \liminf_{n \to \infty} \mathcal{E}_n(u_n,A):   \  u_n \to  u \hbox{ in measure on }  A  \big\}, \notag \\
\mathcal{E}''(u,A) &:= \Gamma-\limsup_{n \to \infty} \mathcal{E}_n(u,A)  = \inf \big\{  \limsup_{n \to \infty} \mathcal{E}_n(u_n,A):   \ u_n \to  u \hbox{ in measure on }  A \big\}
\end{align} 
for all $u \in GSBD^p(\Omega)$ and $A \in \mathcal{A}(\Omega)$.

\begin{lemma}[Properties of $\Gamma$-liminf and  $\Gamma$-limsup]\label{eq: liminflimsup-prop}
 Let $\Omega \subset \R^d$  be open.  Let $\mathcal{E}_n\colon  GSBD^p(\Omega)   \times \mathcal{A}(\Omega) \to [0,\infty)$ be a sequence of functionals as in (\ref{eq: basic energy}), where we assume that $f_n$ and $g_n$ satisfy  (\ref{eq: general bound}) and (\ref{eq: general bound2}), respectively, for all $n\in\mathbb{N}$. Define $\mathcal{E}'$ and $\mathcal{E}''$ as in \eqref{eq: liminf-limsup}. For brevity, we write  $ \mathcal{I}(u,A):=   \|e(u)\|_{L^p(A)}^p+ \mathcal{H}^{d-1}(J_u \cap A)$. Then we have
\begin{align}\label{eq: infsup0}
{\rm (i)} & \ \ 
\mathcal{E}'(u,A) \le \mathcal{E}'(u,B), \ \ \ \ \ \  \mathcal{E}''(u,A) \le \mathcal{E}''(u,B) \ \ \ \text{ whenever } A \subset B, \notag \\
{\rm (ii)} & \ \ 
 \alpha \mathcal{I}(u,A)  \le \mathcal{E}'(u,A)    \le  \mathcal{E}''(u,A) \le \beta\mathcal{I}(u,A) + \beta\mathcal{L}^d(A), \notag \\
{\rm (iii)} & \ \ 
\mathcal{E}'(u,A)  = \sup\nolimits_{B \subset \subset A} \mathcal{E}'(u,B),  \ \ \ \ \mathcal{E}''(u,A)  = \sup\nolimits_{B \subset \subset A} \mathcal{E}''(u,B)  \ \ \  \text{ whenever } A \in \mathcal{A}(\Omega),  \notag \\
{\rm (iv)} & \ \ 
\mathcal{E}'(u,A\cup B) \le  \mathcal{E}'(u,A) + \mathcal{E}''(u,B), \notag \\
& \ \   \mathcal{E}''(u,A\cup B) \le  \mathcal{E}''(u,A) + \mathcal{E}''(u,B)  \ \ \  \text{ whenever } A,B \in \mathcal{A}(\Omega),
\end{align}
where  $\alpha, \beta$ appear in (\ref{eq: general bound}) and (\ref{eq: general bound2}).
\end{lemma}

%

\begin{proof}[Proof of Lemma \ref{eq: liminflimsup-prop}]
 Apart from (iii), the proof is standard. For convenience of the reader, however, we describe the arguments here  to some extent. First, property (i) follows from the fact that all  $\mathcal{E}_n(u,\cdot)$ are measures.   The upper bound  in (ii)   follows  by  (\ref{eq: general bound})--(\ref{eq: general bound2}) and by  taking the constant sequence $u_n = u$ in \eqref{eq: liminf-limsup}. For the lower bound in (ii), let us consider an (almost) optimal sequence $(w_n)_n$ in  \eqref{eq: liminf-limsup}. By the growth conditions (\ref{eq: general bound}) and (\ref{eq: general bound2}) we get  that   \eqref{eq: hp in V.C.} is satisfied,  i.e., we can apply Theorem \ref{thm: Vito compactness}.  Since the sequence $(w_n)_n$ converges in measure to $u$, the set  $G_\infty$  has $\mathcal{L}^d$-negligible measure.  Then,  (\ref{eq: general bound})--(\ref{eq: general bound2}) along with \eqref{eq: main properties}(ii),(iii)   imply the lower bound. 
 
 As a preparation for (iii) and (iv), we show the following:   for all sets $D, E,F \in \mathcal{A}(\Omega)$, $E \subset \subset F \subset \subset D$,  we have
\begin{align}\label{eq: infsup2}
\mathcal{E}'(u,D) \le \mathcal{E}'(u,F) +  \mathcal{E}''  (u,D \setminus \overline{E}), \ \ \ \ \ \ \mathcal{E}''(u,D) \le \mathcal{E}''(u,F) + \mathcal{E}''(u,D \setminus \overline{E}). 
\end{align} 
Indeed, let $(u_n)_n, (v_n)_n \subset GSBD^p(\Omega)$ be sequences converging in measure to $u$ on $F$  and $D \setminus \overline{E}$, respectively, such that
\begin{align}\label{eq: infsup1}
\mathcal{E}''(u,F) = \limsup\nolimits_{n \to \infty} \mathcal{E}_n(u_n,F), \ \ \ \ \ \ \ \mathcal{E}''(u,D \setminus \overline{E}) = \limsup\nolimits_{n \to \infty} \mathcal{E}_n(v_n,D \setminus \overline{E}).
\end{align}
 Fix $\eta >0$.  We apply Proposition \ref{lemma: fundamental estimate}  for   $A = F$, $B =  D  \setminus \overline{E}$,   and some $A' \in \mathcal{A}(\Omega)$ with $E \subset \subset A' \subset \subset F$.  Note that we indeed have $A' \subset \subset A$ and $A \setminus A'  \subset   B$.  We  get a function $ w^\eta_n  \in GSBD^p(D)$ satisfying (see \eqref{eq: assertionfund}(i))
\begin{align}\label{eq: infsup3}
\mathcal{E}_n (  w^\eta_n,  D) \le  (1+ \eta)\big(\mathcal{E}_n(u_n,F)  + \mathcal{E}_n(v_n, D\setminus \overline{E}) \big) +  \Lambda_\eta  (u_n,v_n) +\eta,
\end{align}
 where $\Lambda_\eta$ is the function given in \eqref{eq: Lambda0}. (We include $\eta$ in the notation to highlight the fact that the definition of $\Lambda_\eta$ depends on $\eta$.)  We observe that $u_n-v_n$ tends to $0$ in measure on $F\setminus \overline{E}$.  Hence,  we get $\Lambda_\eta(u_n, v_n)\to 0$  by \eqref{eq: Lambda0}.  By a diagonal argument we can find a sequence $(\eta_n)_n \subset (0,+\infty)$ such that $\eta_n \to 0$ and $\Lambda_{\eta_n}(u_n, v_n)\to 0$ as $n \to \infty$.   We now define   $\bar{w}_n:= w_n^{\eta_n}$ for $n \in \N$.  Recall that  $(u_n)_n$ and $(v_n)_n$  converge in measure to $u$ on $F$  and $D \setminus \overline{E}$, respectively.  In view of  \eqref{eq: assertionfund}(ii),(iii),  we get that $\bar{w}_n$  converges in measure to $u$ on $D$. Then, by using \eqref{eq: liminf-limsup},  \eqref{eq: infsup1}--\eqref{eq: infsup3},  and $\eta_n\to 0$  we obtain 
$$\mathcal{E}''(u,D) \le \limsup\nolimits_{n \to \infty} \mathcal{E}_n (  \bar{w}_n,  D) \le \mathcal{E}''(u,F) + \mathcal{E}''(u,D \setminus \overline{E}). $$
 This implies  the second estimate in   \eqref{eq: infsup2}.  A similar argument  yields the first one.

  Let us now show (iii),  i.e., the inner regularity of  $\mathcal{E}'$ and $\mathcal{E}''$.  By  \eqref{eq: infsup0}(ii) and \eqref{eq: infsup2} we get 
$$
\mathcal{E}''(u,D) \le \mathcal{E}''(u,F) +  \beta \mathcal{I}(u,D\setminus \overline{E}) + \beta\mathcal{L}^d(D \setminus \overline{E}). 
$$
Since we can choose $D$ and $E$ in such a way that  $\mathcal{I}(u,D\setminus \overline{E})$ and $\mathcal{L}^d(D \setminus \overline{E})$   can be taken arbitrarily small, and  $\mathcal{E}''(u,\cdot)$ is an increasing set function, we obtain $\mathcal{E}''(u,D)  \le \sup\nolimits_{F \subset \subset D} \mathcal{E}''(u,F)$. This shows (iii) for  $\mathcal{E}''$. The proof for $\mathcal{E}'$ is similar. 

 Finally, we show (iv).  Observe that the inequalities are clear if $A \cap B = \emptyset$. Let  $A,B \in \mathcal{A}(\Omega)$  with nonempty intersection. Given $\eps >0$, one can choose $M \subset \subset M' \subset \subset A$ and $N \subset \subset N' \subset \subset B$ such that  $M,M',N,N' \in \mathcal{A}(\Omega)$, $M' \cap N' = \emptyset$, and  $\mathcal{I}(u,(A\cup B) \setminus (\overline{M \cup N} ))  + \mathcal{L}^d((A\cup B) \setminus (\overline{M \cup N} ))  \le \eps$,  see \cite[Proof of Lemma 5.2]{AmbrosioBraides} for details. Then using, \eqref{eq: infsup0}(i),(ii) and \eqref{eq: infsup2}  we get 
\begin{align*}
\mathcal{E}''(u,A \cup B) & \le  \mathcal{E}''(u,M' \cup N') + \mathcal{E}''(u, (A\cup B) \setminus \overline{M \cup N} ) \le  \mathcal{E}''(u,M') +  \mathcal{E}''(u,N') + \beta  \eps  \\ & \le \mathcal{E}''(u,A) +  \mathcal{E}''(u,B) + \beta  \eps,
\end{align*}
where we also used $\mathcal{E}''(u,M' \cup N') \le \mathcal{E}''(u,M') +  \mathcal{E}''(u,N')$ which holds due to $M' \cap N' = \emptyset$. Since $\eps$ was arbitrary, the statement follows. The proof for $\mathcal{E}'$ is again  similar. 
\end{proof}

We can now prove Theorem \ref{th: gamma}.

\begin{proof}[Proof of Theorem \ref{th: gamma}]  
We apply a compactness result for $\bar{\Gamma}$-convergence, see \cite[Theorem 16.9]{DalMaso:93}, to find  an increasing sequence of integers $(n_k)_k$ such that  the objects $\mathcal{E}'$ and $\mathcal{E}''$ defined in \eqref{eq: liminf-limsup} with respect to  $(n_k)_k$ satisfy
$$(\mathcal{E}')_-(u,A) = (\mathcal{E}'')_-(u,A)  $$
for all $u \in GSBD^p(\Omega)$ and $A \in \mathcal{A}(\Omega)$, where $(\mathcal{E}')_-$ and $(\mathcal{E}'')_-$ denote the inner regular envelope.   By   \eqref{eq: infsup0}(iii) we know that $\mathcal{E}'$ and $\mathcal{E}''$ are inner regular, and thus they both coincide with their respective inner regular envelopes.   This shows that  the $\Gamma$-limit, denoted by $\mathcal{E}:= \mathcal{E}' = \mathcal{E}''$, exists for all $u \in GSBD^p(\Omega)$ and all $A \in \mathcal{A}(\Omega)$. 

 We now check that $\mathcal{E}$ satisfies the  assumptions  of the integral representation result \cite[Theorem~2.1]{Crismale-Friedrich-Solombrino}. First, the definition in \eqref{eq: liminf-limsup} and the locality of each $\mathcal{E}_n$ show that $\mathcal{E}(\cdot, A)$ is local for any $A \in \mathcal{A}(\Omega)$, i.e.,   $\mathcal{E}(u,A) = \mathcal{E}(v,A)$ for all $u,v \in GSBD^p(\Omega)$ satisfying $u=v$ $\mathcal{L}^d$-a.e.\ in $A$. Moreover, in view of  \cite[Remark 16.3]{DalMaso:93}, $\mathcal{E}(\cdot,A)$ is lower semicontinuous with respect to convergence in measure on $\Omega$ for any $A \in \mathcal{A}(\Omega)$. Next, we check that  $\mathcal{E}(u,\cdot)$ can be extended to a Borel measure for any $u \in  GSBD^p(\Omega)$. Indeed,  $\mathcal{E}$ is increasing, superadditive, and inner regular, see \cite[Proposition 16.12 and Remark 16.3]{DalMaso:93}. Moreover, by \eqref{eq: infsup0}(iv) we find that $\mathcal{E}$ is subadditive. Then,  by De  Giorgi-Letta (see \cite[Theorem 14.23]{DalMaso:93}), $\mathcal{E}(u,\cdot)$ can be extended to a Borel measure. Eventually, by \eqref{eq: infsup0}(ii) we get that 
$$\alpha \bigg(\int_{ B} |e(u)|^p \, {\rm d}x    +   \mathcal{H}^{d-1}(J_u \cap B)\bigg) \le \mathcal{E}(u,B) \le \beta \bigg(\int_{ B} (1 + |e(u)|^p)  \,  {\rm d}x  +   \mathcal{H}^{d-1}(J_u \cap B)\bigg).$$  
for every   $u \in GSBD^p(\Omega)$ and every Borel set  $B \subset \Omega$. Therefore, $\mathcal{E}$ satisfies the  assumptions  of \cite[Theorem 2.1]{Crismale-Friedrich-Solombrino}, and we conclude that $\mathcal{E}$ admits a representation of the form \eqref{eq: representation}--\eqref{eq: g^epsilon_infty}. (The minimization problems \eqref{eq: f^epsilon_infty}--\eqref{eq: g^epsilon_infty} can be formulated both in terms of balls and cubes, see \cite[Theorem 2.1, Remark 2.2(iv)]{Crismale-Friedrich-Solombrino}.) 
\end{proof}

\section{Two approximation results}\label{sec: apprix}

 In this section we present two approximation results for $GSBD^p$ functions which are instrumental for the proof of  Theorem  \ref{thm: main thm d=2},  but also of independent interest.

\subsection{Approximation of $GSBD^p$ functions by $W^{1,p}$}

In the first result we show that a sequence of  $GSBD^p$ functions with asymptotically vanishing jump sets can be approximated by a sequence of equiintegrable Sobolev functions.  The result is a variant in $GSBD^p$ of similar results for Sobolev functions \cite{Fonseca} and $BV$-functions  \cite{Larsen}. It crucially relies  on the recent Korn inequality  stated in  Theorem \ref{th: kornSBDsmall}. Later we will apply this result in the blow-up around points with approximate gradient.

\begin{lemma}[Approximation of $GSBD^p$ functions by $W^{1,p}$]\label{lem: our Larsen}
Let $\Omega\subset \R^d$ be an open, bounded set with Lipschitz boundary, and let $1<p<+\infty$. Let $u\colon\Omega \to \R^d$ be a measurable function, and let $(u_n)_{n}\subset GSBD^p(\Omega)$ be a sequence satisfying
\begin{align}\label{eq: properties-part0}
{\rm (i)} & \ \ \sup_{n\in\mathbb{N}}\| e(u_n)\|_{L^p(\Omega)} <\infty   ,\notag\\
{\rm (ii)} & \ \  \lim_{n\to \infty}\mathcal{H}^{d-1}(J_{u_n})=0   ,\notag\\
{\rm (iii)} & \ \ u_n \to u \text{ in measure on $\Omega$ as $n \to +\infty$.}
\end{align} 
Then, there exists a sequence $(w_n)_n\subset W^{1,p}(\Omega;\R^d)$ such that $(|\nabla w_n|^p)_n$ is equiintegrable, and 
\begin{align}\label{eq: Larsen properties}
{\rm (i)} & \ \  \lim_{n\to \infty} \mathcal{L}^d\left(\left\{w_n\neq u_n \right\}\cup \left\{e(w_n)\neq e(u_n)   \right\}  \right)=0,\notag\\
{\rm (ii)}  & \ \  \lim_{n\to \infty} \Vert w_n - u \Vert_{L^p(\Omega)} =0.  
\end{align}
\end{lemma}

\begin{proof}
 We first introduce the sequence $(w_n)_n$, and afterwards we study its properties.

\noindent \emph{Step 1: Existence of a bounded sequence in $W^{1,p}$.} Let us  prove that there exists a  bounded  sequence $(v_n)_n\subset W^{1,p}(\Omega;\R^d)$ such that 
\begin{align}\label{eq: step1 Larsen}
\mathcal{L}^d\left(\left\{v_n\neq u_n \right\}\cup \left\{e(v_n)\neq e(u_n)   \right\}  \right)\to 0.
\end{align}
By Theorem \ref{th: kornSBDsmall},  for every $n\in\mathbb{N}$ there exists $v_n \in W^{1,p}(\Omega;\R^d)$ and a set of finite perimeter $\omega_n  \subset \Omega  $ such that $u_n=v_n$ in $\Omega\setminus \omega_n$, and $\mathcal{L}^d(\omega_n)  \le  c (\mathcal{H}^{d-1}(J_{u_n}))^{d/(d-1)}$, where the constant $c$ depends only on $p$, $d$, and $\Omega$. Thus, by construction  and \eqref{eq: properties-part0}(ii),  (\ref{eq: step1 Larsen}) holds true.  We now show that  the sequence $(v_n)_n$ is bounded in $W^{1,p}(\Omega;\R^d)$.   To this end, we first apply Korn's and Poincar\'e's inequality  for Sobolev functions to get
\begin{align}\label{eq: KKKbound}
\| v_n- a_n \|_{L^p(\Omega)} +   \| \nabla v_n- A_n \|_{L^p(\Omega)}  \le  C  \| e(v_n)\|_{L^p(\Omega)}  \le C  \| e(u_n)\|_{L^p(\Omega)},
\end{align}
where each $a_n$  is a rigid motions, i.e.,  $a_n(x):= A_nx + b_n$ for   $A_n\in\R^{d\times d}_{\rm skew}$   and $b_n\in\R^d$. It  now suffices to prove  
\begin{align}\label{A_n e b_n bdd}
\sup\nolimits _{n\in\mathbb{N}}\left(|A_n|+|b_n| \right) < +\infty.
\end{align}
 Indeed,  this also shows $ \sup_{n\in\mathbb{N}} \|a_n\|_{L^p(\Omega)}<\infty$,  and then   along with  \eqref{eq: properties-part0}(i), \eqref{eq: KKKbound}, and the triangle inequality we obtain the boundedness of $(v_n)_n$  in $W^{1,p}(\Omega;\R^d)$.

Let us show \eqref{A_n e b_n bdd}.  Since $(v_n)_n$ converges in measure to $u$   (see \eqref{eq: properties-part0}(iii)),  by \cite[Remark 2.2]{FriedrichSolombrino} there exists  a strictly  increasing concave function  $\psi\colon \R_+ \to \R_+$  with  $\psi(0)=0$  and $\psi(\R_+) = \R_+$,  such that,  up to  passing  to a subsequence  $(v_{n_k})_k$, we have 
\begin{align}\label{eq: step1 Larsen psi}
\sup_{k\in\mathbb{N}}\int_{\Omega}\psi(|v_{n_k}|)\, {\rm d}x \leq 1.
\end{align}
 As $\psi \ge 0$ is increasing and concave, it is also subadditive. Therefore, we get by \eqref{eq: step1 Larsen psi} and the triangle inequality that 
\begin{align*}
\int_{\Omega}\psi(|a_{n_k}|)\, {\rm d}x &\leq  \int_{\Omega}\psi(|v_{n_k}-a_{n_k}|)\, {\rm d}x +  \int_{\Omega}\psi(|v_{n_k}|)\, {\rm d}x \leq \int_{\Omega}\psi(|v_{n_k}-a_{n_k}|)\, {\rm d}x +1. 
\end{align*}
Then, by using Jensen's inequality for concave functions,  H\"older's inequality,  and the monotonicity of $\psi$   we obtain  
\begin{align*}
\fint_{\Omega}\psi(|a_{n_k}|)\, {\rm d}x  \leq \psi \left( \fint_{\Omega}|v_{n_k}-a_{n_k}|\,{\rm d} x  \right)+  \frac{1}{\mathcal{L}^d(\Omega)}  \leq \psi \left(    \big(\mathcal{L}^d(\Omega)\big)^{-\frac{1}{p}}  \|v_{n_k}-a_{n_k}\|_{L^p(\Omega)}  \right)+   \frac{1}{\mathcal{L}^d(\Omega)}.  
\end{align*}
 This along with \eqref{eq: properties-part0}(i) and \eqref{eq: KKKbound} shows that $\sup_{k \in \N} \int_{\Omega}\psi(|a_{n_k}|)\, {\rm d}x < + \infty$.    Thus, Lemma \ref{lemma: rigid motion}  yields \eqref{A_n e b_n bdd}  (for the subsequence), and therefore  $(v_{n_k})_k$  is bounded in $W^{1,p}(\Omega;\R^d)$.  We now show that actually the whole sequence is bounded: by  Rellich's theorem,  \eqref{eq: properties-part0}(iii), and \eqref{eq: step1 Larsen} we get that a further subsequence converges strongly in $L^p(\Omega;\R^d)$ to $u$. By Urysohn's property we deduce that the \emph{whole} sequence $(v_n)_n$ converges to $u$ in $L^p(\Omega;\R^d)$. In particular, $\sup_{n \in \N} \Vert v_n\Vert_{L^p(\Omega)} <+\infty$  which along with \eqref{eq: properties-part0}(i) and \eqref{eq: KKKbound} shows $\sup_{n \in \N}\Vert a_n \Vert_{L^p(\Omega)}<+\infty$. Then Lemma \ref{lemma: rigid motion} applied for $\psi(t) = t^p$ yields \eqref{A_n e b_n bdd} for the whole sequence.   This concludes  Step 1  of the proof. 


\emph{Step 2: Equiintegrability  and \eqref{eq: Larsen properties}.}  Thanks   to Step 1, the sequence   $(v_n)_n$ is  bounded in $W^{1,p}(\Omega;\R^d)$. Therefore,   we can apply \cite[Lemma 1.2]{Fonseca}  (see also \cite[Lemma 2.1]{Larsen}), and    we deduce that there exists  a  sequence $(w_n)_n\subset W^{1,p}(\Omega;\R^d)$ such that $(|\nabla w_n|^p)_n$ is equiintegrable, and 
\begin{align}\label{eq: antoher}
\mathcal{L}^d\left(\left\{w_n\neq v_n \right\}\cup \left\{e(w_n)\neq e(v_n)   \right\}  \right)\to 0.
\end{align}
By (\ref{eq: step1 Larsen}) we get that     \eqref{eq: Larsen properties}(i) holds true.  We  recall that  $\sup_{n \in \N} \Vert \nabla w_n \Vert_{L^p(\Omega)} <+\infty$ and $v_n \to u$ in measure on $\Omega$. This along with  Poincar\'e's inequality, Rellich's theorem, and \eqref{eq: antoher} shows that  \eqref{eq: Larsen properties}(ii) holds true.   This concludes the proof  of the lemma.
\end{proof}




\subsection{Approximation of jump sets by boundary of partitions}

The goal of the subsection is to prove the following result which allows to approximate the jump  sets   of $GSBD^p$ functions by the boundary of partitions.  To this end, we define the sets 
\begin{align}\label{eq: half-cub not}
Q^{\nu,\pm}_1 = Q^\nu_1 \cap \lbrace  x \in \R^d \colon \, \pm   x \cdot \nu    \ge 0\rbrace,
\end{align}
where here and in the following $\pm$ is a placeholder for both $+$ and $-$.  Recall the definition of   $u_{x_0,a,b,\nu}$ in  \eqref{eq: jump competitor}.

\begin{lemma}[Approximation of jump sets by boundary of partitions]\label{lemma: partition}
Let $d= 2$ and $p \ge 2$. Let $\zeta \in \R^2  \setminus \lbrace 0 \rbrace  $ and  $\nu \in \mathbb{S}^1$. Let $(u_n)_n \subset GSBD^p(Q^\nu_1)$ be a sequence satisfying
\begin{align}\label{eq: properties-part}
{\rm (i)} & \ \ \lim_{n \to \infty} \Vert e(u_n) \Vert_{L^p(Q^\nu_1)} =0,\notag\\
{\rm (ii)} & \ \  \sup_{n \in \N} \mathcal{H}^1(J_{u_n}) <+\infty,\notag\\
{\rm (iii)} & \ \ u_n \to u_{0,\zeta,0,\nu} \text{ in measure on $Q^\nu_1$ as $n \to +\infty$.}
\end{align} 
Then, there exists a sequence of neighborhoods $N_n \subset Q^\nu_1$ of $\partial Q_1^\nu$ and pairwise disjoint sets $S^+_n$ and  $S^-_n$ with $\mathcal{L}^2\big(Q^\nu_1 \setminus (S^+_n \cup S^-_n)\big) = 0 $ for all $n \in \N$  such that
\begin{align}\label{eq: properties-part2}
{\rm (i)} & \ \ \lim_{n \to  \infty} \mathcal{L}^2\big( S^\pm_n \triangle Q^{\nu,\pm}_1   \big) = 0,\notag\\
{\rm (ii)} &  \ \  N_n \cap Q^{\nu,\pm}_1 \subset S^\pm_n  \ \ \ \text{for all $n \in \N$},\notag\\
{\rm (iii)} & \ \  \lim_{n \to \infty }\mathcal{H}^1\big( ( \partial^* S^+_n \cap \partial^* S^-_n )           \setminus  J_{u_n}     \big) = 0.
\end{align}
\end{lemma}
Later we will apply this result in the blow-up around jump points. Indeed, assumptions \eqref{eq: properties-part} are satisfied in the blow-up around $\mathcal{H}^1$-a.e.\ jump point. The result states that,  up to an asymptotically small set,  the jump set  covers the boundary of a partition  which consists of two sets  approximating the upper and lower half cubes $Q^{\nu,+}_1$ and $Q^{\nu,-}$, see \eqref{eq: properties-part2}(i),(iii). Condition \eqref{eq: properties-part2}(ii) will ensure that the piecewise constant functions $v_n := \zeta \chi_{S_n^+}$ satisfy  $v_n = u_{0,\zeta,0,\nu}$ in a neighborhood of $\partial Q^\nu_1$.  


\begin{remark}[Dimension $d=2$, exponent $p$]\label{rem: d,p}
(i) We emphasize that the result holds in \emph{dimension two only} due to the application of a piecewise Korn-Poincar\'e inequality, see Proposition \ref{th: kornpoin-sharp} and \cite{FriedrichSolombrino}, which has been derived only for $d=2$. The result, crucially based on similar results of this type \cite{Friedrich:15-5, Friedrich:15-4}, heavily relies on combining different components of the jump set by segments via an explicit construction: a technique whose extension to higher dimension seems to be very difficult. 

(ii) Let us also mention that, for simplicity, Proposition \ref{th: kornpoin-sharp} has been derived  in $GSBD^2$ and not in $GSBD^p$ for general $1 < p< +\infty$. Therefore, we can derive Lemma \ref{lemma: partition} only for $p \ge 2$. (For $p>2$ we can resort to $GSBD^2$ via H\"older's inequality.) Although a generalization of Proposition~\ref{th: kornpoin-sharp} to general $1 < p < + \infty$ would be possible without significant changes in the proof, we refrain from entering into such minor issues and prefer to address the problem in this slightly less general fashion.  
\end{remark}

Besides the actual proof   of Lemma \ref{lemma: partition}, we also give a simplified proof in Appendix \ref{proof1XXX}  which works under the assumption that each $J_{u_n}$ consists of a \emph{bounded number of closed, continuous curves}. Our motivation to present this simpler version of the  proof is twofold. On the one hand, it provides elementary self-contained arguments  avoiding the deep and complicated result from  \cite{FriedrichSolombrino}. On the other hand, the construction of combining different parts of the jump set  represents  (in a simplified way) a main technique used in \cite{Friedrich:15-4, Friedrich:15-5, FriedrichSolombrino} being essential in the proof of  Proposition~\ref{th: kornpoin-sharp}.  We now proceed with the  proof of Lemma \ref{lemma: partition}, and refer the reader to Appendix \ref{proof1XXX} for the simplified proof.

\begin{proof}[Proof of Lemma \ref{lemma: partition}]
Let $(u_n)_n  \subset  GSBD^p(Q^\nu_1)$ for $p \ge 2$ be given. By H\"older's inequality we clearly have $(u_n)_n \subset GSBD^2(Q^\nu_1)$ with 
\begin{align}\label{eq: l2-nnound}
\lim_{n\to \infty} \Vert e(u_n) \Vert_{L^2(Q^\nu_1)} = 0
\end{align}
due to \eqref{eq: properties-part}(i). We define 
\begin{align}\label{eq: bound-NNN}
C_0 :=     \sup\nolimits_{n \in \N} \mathcal{H}^1(J_{u_n}) + \mathcal{H}^1(\partial Q^\nu_1)< + \infty.   
\end{align}
Our strategy relies on applying Proposition \ref{th: kornpoin-sharp}  for  $u_n$ and for fixed $0< \theta <\min\lbrace \theta_0,\frac{1}{2}\rbrace$, where $\theta_0$ is the constant from  Proposition \ref{th: kornpoin-sharp}. At the end of the proof, we pass to the limit $\theta \to 0$ and perform a diagonal argument. For simplicity, we do not indicate the $\theta$-dependence of the objects explicitly in the notation.  We start by using \eqref{eq: properties-part}(iii) to find a sequence $(\eta_n)_n \subset (0,+\infty)$ with $\eta_n \to 0$ and some $n_\theta \in \N$ depending on $\theta$    such that the sets  
\begin{align}\label{eq: Bdef-NNN}
B_n^- = \lbrace x\in Q^{\nu,-}_1 \colon \, |u_n(x)| < \eta_n\rbrace, \quad \quad \quad \quad B_n^+ = \lbrace x\in Q^{\nu,+}_1 \colon \, |u_n(x) - \zeta | < \eta_n\rbrace, 
\end{align}
satisfy  for all $n \ge n_\theta$ that 
\begin{align}\label{eq: bound2-NNN}
\mathcal{L}^2\big( Q^{\nu,\pm}_1 \setminus B_n^\pm \big) \le \theta^8/4.
\end{align}
By Proposition \ref{th: kornpoin-sharp} for $\theta^2$ in place of $\theta$  and by \eqref{eq: bound-NNN}  we obtain partitions $Q^\nu_1 = R_{n} \cup \bigcup^{J_{n}}_{j=1} P^{n}_j$ and  corresponding rigid motions $(a^{n}_j)_{j=1}^{J_{n}}$ such that
\begin{align}\label{eq: kornpoinsharp-NNN}
{\rm (i)} & \ \ \mathcal{H}^1\big( (\partial^* R_{n} \cap Q^\nu_1 )\setminus J_{u_n}  \big) +   \sum\nolimits_{j=1}^{J_{n}}\mathcal{H}^1\big( (\partial^* P^{n}_j \cap Q^\nu_1) \setminus J_{u_n} \big) \le  C_0\theta^2,\notag\\
{\rm (ii)} & \ \ \mathcal{L}^2(R_{n})  \le C_0^2 \theta^2, \ \ \ \ \ \  \mathcal{L}^2(P_j^{n}) \ge \theta^6  \ \ \ \text{for all $j=1,\ldots,J_{n}$}.      \notag\\
{\rm (iii)}& \ \    \max_{1\le j \le J_{n}}\Vert u_n - a^{n}_j \Vert_{L^\infty(P^{n}_j)}  \le C_{\theta^2} \Vert e(u_n) \Vert_{L^2(Q^\nu_1)}.
\end{align}
\emph{Step 1: Components essentially contained in half-cubes.} In this step we show that, by possibly passing to a larger  $n_\theta \in \N$ only depending on $\theta$, for $n \ge n_\theta$ we have for all $j=1,\ldots,J_n$ that 
\begin{align}\label{eq: step1}
\mathcal{L}^2(P^n_j \cap B^-_n)=0 \quad \quad \text{or} \quad \quad \mathcal{L}^2(P^n_j \cap B^+_n)=0.
\end{align}

To see this, by  \eqref{eq: bound2-NNN} and  \eqref{eq: kornpoinsharp-NNN}(ii), for each $j = 1,\ldots,J_n$ we get 
 \begin{align}\label{eq: bound3-NNN}
\max\big\{ \mathcal{L}^2\big( B_n^+ \cap P^n_j \big),  \mathcal{L}^2\big(  B_n^- \cap P^n_j \big)     \big\} \ge \theta^6/4.
\end{align}
We now first assume that the maximum is attained for $B_n^-$ and show $\mathcal{L}^2(P_j^n \cap B_n^+) = 0$. Afterwards, we briefly indicate the changes if the maximum is attained for $B_n^+$. Writing $a^n_j$ as $a^n_j(x) = A^n_j \, x + b^n_j$, we get by Lemma \ref{lemma: rigid motion} for   $\delta =\theta^6/4$  and $R = \sqrt{2}/2$,  see  \eqref{eq: estimate2}, that 
\begin{align}\label{eq: bound3-NNNnewbo}
\Vert A^n_j \, x +  b^n_j \Vert_{L^\infty(Q^\nu_1)}  \le c_\theta  \Vert  A^n_j  \, x +b^n_j \Vert_{ L^1  (B^-_n \cap P^n_j )},
\end{align} 
 where $c_\theta>0$ is a constant depending on $\theta$. In view of \eqref{eq: Bdef-NNN} and  \eqref{eq: kornpoinsharp-NNN}(iii),  this implies 
 \begin{align}\label{eq: bound3-NNNNN}
 \Vert a^n_j \Vert_{L^\infty(Q^\nu_1)} & \le c_\theta  \Vert a^n_j \Vert_{L^\infty(B^-_n \cap P^n_j )} \le c_\theta \big( \Vert a^n_j - u_n \Vert_{L^\infty(P^n_j )} +  \Vert u_n \Vert_{L^\infty(B^-_n)}\big) \notag\\
 & \le   c_\theta \big(C_{\theta^2} \Vert e(u_n) \Vert_{L^2(Q^\nu_1)} + \eta_n\big).
 \end{align}
Now, suppose by contradiction that  $\mathcal{L}^2(P_j^n \cap B_n^+) > 0$. Then  \eqref{eq: Bdef-NNN},  \eqref{eq: kornpoinsharp-NNN}(iii), and \eqref{eq: bound3-NNNNN} give
\begin{align*}
|\zeta| - \eta_n \le \Vert u_n \Vert_{L^\infty(B_n^+ \cap P^n_j)} \le  \Vert u_n - a^n_j \Vert_{L^\infty(P^n_j)} +  \Vert a^n_j \Vert_{L^\infty(Q^\nu_1)} \le  (c_\theta+1) \big(C_{\theta^2} \Vert e(u_n) \Vert_{L^2(Q^\nu_1)} + \eta_n\big).
\end{align*}
By \eqref{eq: l2-nnound} and the fact that $\eta_n \to 0$ this yields a contradiction for $n$ sufficiently large only depending on $\theta$. If the maximum in \eqref{eq: bound3-NNN} is attaind for $B_n^+$ instead, we can show $\mathcal{L}^2(P_j^n \cap  B_n^-)  = 0$ by a similar argument, where we repeat the argument in \eqref{eq: bound3-NNN}--\eqref{eq: bound3-NNNNN} for $A^n_j \, x +  (b^n_j -\zeta)$ instead of $A^n_j \, x +  b^n_j$. We omit the details. 

\noindent \emph{Step 2: Cutting of components.}
We define the set
\begin{align}\label{eq: Vtheta}
V_\theta = \lbrace x \in Q^\nu_1 \colon | x \cdot \nu | \le \theta \rbrace.
\end{align} 
Let $n \ge n_\theta$.  In this step, up to sets of  negligible $\mathcal{L}^2$-measure,  we cut the components $R_n$ and $(P^n_j)^{J_n}_{j=1}$ into sets
$$R_n = R^+_n \cup R^-_n, \quad \quad \quad P^n_j = P^{n,+}_j \cup P^{n,-}_j \quad \text{for all }  j=1,\ldots,J_n,$$
such that 
\begin{align}\label{eq: cutting}
{\rm (i)} & \ \  R_n^\pm \cup \bigcup\nolimits_{j=1}^{J_n} P^{n,\pm}_j \subset Q^{\nu,\pm}_1 \cup V_\theta, \notag \\     
{\rm (ii)} &   \ \   \sum\nolimits_{j=1}^{J_n}  \mathcal{H}^1(\partial^* P^{n,\pm}_j \setminus \partial^* P_j^n) + \mathcal{H}^1(\partial^* R_n^\pm \setminus \partial^* R_n) \le (C_0^2 + 1) \theta. 
\end{align}

To this end, let us fix $P^n_j$. We can assume without restriction that $\mathcal{L}^2(P^n_j \cap B^+_n) = 0$ by Step 1, see \eqref{eq: step1}. (The other case can be treated in a similar fashion.) By Fubini's theorem,  $\theta \le \frac{1}{2}$,  and \eqref{eq: bound2-NNN} we find that 
$$\int_0^\theta \mathcal{H}^1\big( P_j^n    \cap  L(s)\big)\, {\rm d}s \le \int_0^{1/2} \mathcal{H}^1\big( (Q^{\nu,+}_1 \setminus B_n^+)    \cap  L(s)\big)\, {\rm d}s \le \mathcal{L}^2(Q^{\nu,+}_1 \setminus B_n^+) \le \theta^8/4,$$
where $L(s) := \lbrace x\in Q^\nu_1\colon  x\cdot \nu    = s \rbrace$. Therefore, we find $s^n_j \in (0,\theta)$ such that  the sets $P^{n,+}_j = P^n_j \cap \lbrace x\in Q^\nu_1\colon x \cdot \nu  > s^n_j \rbrace $ and $P^{n,-}_j = P^n_j \cap \lbrace x\in Q^\nu_1\colon  x \cdot \nu   < s^n_j \rbrace $ satisfy
$$\mathcal{H}^1\big(\partial^* P^{n,\pm}_j \setminus \partial^* P_j^n  \big) = \mathcal{H}^1\big( P_j^n    \cap  L(s^n_j)\big) \le \theta^{-1} \theta^8/4 \le \theta^7.$$
Clearly, by construction we also have $P^{n,\pm}_j \subset Q^{\nu,\pm}_1 \cup V_\theta$. We repeat this construction for each $P^n_j$. As $\ J_n \le \theta^{-6}$, see \eqref{eq: kornpoinsharp-NNN}(ii), we then get 
$$  \sum\nolimits_{j=1}^{J_n}  \mathcal{H}^1\big(\partial^* P^{n,\pm}_j \setminus \partial^* P_j^n\big) \le \# J_n \,  \theta^7 \le \theta.$$
By a similar construction we can define $R_n = R_n^- \cup R_n^+$ such that $R_n^\pm \subset Q^{\nu,\pm}_1 \cup V_\theta$ and 
$$\mathcal{H}^1(\partial^* R_n^\pm \setminus \partial^*R_n) \le \theta^{-1} \mathcal{L}^2(R_n) \le C_0^2 \theta, $$ 
where the last step follows from \eqref{eq: kornpoinsharp-NNN}(ii). This shows \eqref{eq: cutting}  and concludes the proof of Step 2.

\noindent \emph{Step 3: Definition of $S_n^+$ and $S_n^-$.} We now define the sets  $S_n^+$ and $S_n^-$ and establish \eqref{eq: properties-part2}. For each $0 < \theta <   \min\lbrace \theta_0,\frac{1}{2}\rbrace $ and each $n \ge n_\theta$, we  first define the sets
$$\hat{S}_{n,\theta}^- = \bigcup\nolimits_{j=1}^{J_n} P^{n,-}_j \cup R_n^-, \quad \quad \quad \quad \hat{S}_{n,\theta}^+ = \bigcup\nolimits_{j=1}^{J_n} P^{n,+}_j \cup  R_n^+.  $$
(We include $\theta$ in the notation to highlight the dependence of the definition on $\theta$.) By \eqref{eq: kornpoinsharp-NNN}(i) and \eqref{eq: cutting}(ii) we note that
\begin{align}\label{eq: combining crack}
\mathcal{H}^1\big(  (\partial^*  \hat{S}_{n,\theta}^- \cap  \partial^*  \hat{S}_{n,\theta}^+  ) \setminus J_{u_n} \big) & \le  \sum\nolimits_{j=1}^{J_n}  \mathcal{H}^1(\partial^* P^{n,+}_j \setminus \partial^* P_j^n)  + \mathcal{H}^1(\partial^* R_n^+ \setminus \partial^* R_n)\notag \\ & \ \ \  + \mathcal{H}^1\big( (\partial^* R_{n} \cap Q^\nu_1 )\setminus J_{u_n}  \big) +   \sum\nolimits_{j=1}^{J_{n}}\mathcal{H}^1\big( (\partial^* P^{n}_j \cap Q^\nu_1) \setminus J_{u_n} \big)\notag \\  & 
\le (C_0^2 + 1) \theta +  C_0 \theta^2 \le (C_0 + 1)^2 \theta.
\end{align}
Define also the neighborhood $N_\theta = Q^{\nu}_1 \setminus Q^{\nu}_{1-\theta}$. The sets $\hat{S}_{n,\theta}^-$ and  $\hat{S}_{n,\theta}^+$  do  possibly not satisfy  \eqref{eq: properties-part2}(ii),  and therefore we introduce the sets     
\begin{align}\label{eq: Sdef-NNN}
S_{n,\theta}^+  = \big(\hat{S}_{n,\theta}^+    \cup  \big(N_\theta \cap  Q^{\nu,+}_1\cap V_\theta\big) \big) \setminus  \big(N_\theta \cap  Q^{\nu,-}_1\cap V_\theta\big), \notag \\  S_{n,\theta}^-  = \big( \hat{S}_{n,\theta}^-    \cup  \big(N_\theta \cap  Q^{\nu,-}_1 \cap V_\theta\big) \big)\setminus  \big(N_\theta \cap  Q^{\nu,+}_1 \cap V_\theta\big).
\end{align}
Clearly, we have  $\mathcal{L}^2\big(Q^\nu_1 \setminus (S^+_{n,\theta} \cup S^-_{n,\theta})\big) = 0$ and we can check that   
      \begin{align}\label{eq: properties-part3-NNN}
{\rm (i)} & \ \  \mathcal{L}^2\big( S^\pm_{n,\theta} \triangle Q^{\nu,\pm}_1   \big) \le 2\theta,\notag\\
{\rm (ii)} &  \ \  N_\theta \cap  Q^{\nu,\pm}_1  \subset S^\pm_{n,\theta},\notag\\
{\rm (iii)} & \ \  \mathcal{H}^1\big( ( \partial S^+_{n,\theta} \cap \partial S^-_{n,\theta} )           \setminus  J_{u_n}     \big) \le c\theta,
\end{align}
for a constant $c>0$ independent of $\theta$ and $n$.  In fact,  (i) and (ii) follow from \eqref{eq: Vtheta},  \eqref{eq: cutting}(i), and \eqref{eq: Sdef-NNN}. To see  (iii), we observe that   \eqref{eq: Sdef-NNN}  implies    
$$\mathcal{H}^1\big( ( \partial S^+_{n,\theta} \cap \partial S^-_{n,\theta} )           \setminus  J_{u_n}     \big) \le \mathcal{H}^1\big( ( \partial \hat{S}^+_{n,\theta} \cap \partial \hat{S}^-_{n,\theta} )           \setminus  J_{u_n}     \big)  +   \sum\nolimits_{j=\pm} \mathcal{H}^1(\partial(N_\theta \cap  Q^{\nu, j}_1\cap V_\theta)),$$
Since $\mathcal{H}^1(\partial(N_\theta \cap  Q^{\nu, \pm}_1\cap V_\theta)) \le c\theta$ for a universal $c>0$,  the estimate then follows from \eqref{eq: combining crack}.

Finally, we obtain the  sets $S^+_{n}$ and $S^-_n$ satisfying \eqref{eq: properties-part2} from  \eqref{eq: properties-part3-NNN} by a  diagonal argument.   
\end{proof}

\section{Identification of the $\Gamma$-limit}\label{sec: identi}

This section is devoted to the proof of the results announced Subsection \ref{sec: ident}. To prove  Theorem~\ref{thm: main thm d=2}, we need to recover the estimates  \eqref{eq: f_infty=f}--\eqref{eq: g_infty=g-relax} (the implications about $\Gamma$-convergence then follow immediately by Theorem \ref{th: gamma}).  This we will we do in Subsections \ref{sec: sub2}-- \ref{sec: sub4} below. Eventually, in Subsection \ref{sec:cor} we prove the corollaries of Theorem   \ref{thm: main thm d=2}.

\begin{remark}\label{rem: recov}
In the proof of \eqref{eq: f_infty=f}--\eqref{eq: g_infty=g-relax}, we will use the following   general property of recovery sequences: if $(u_n)_n$ is a recovering sequence for $u$ with respect to $\mathcal{E}_n(\cdot, \Omega)$, then $(u_n)_n$ is optimal for $u$ with respect to $\mathcal{E}_n(\cdot, A)$ for every $A\in \mathcal{A}(\Omega)$ such that $\mathcal{E}(u,\partial A)=0$.  This follows from the fact that $\mathcal{E}(u,\cdot)$ is a Radon measure for each $u \in  GSBD^p(\Omega)$. 
\end{remark}

\subsection{Bulk part: Proof of  \eqref{eq: f_infty=f}}\label{sec: sub2}
We start with the bulk density  by   showing two inequalities.

\noindent \emph{Step 1: $f_\infty(x,e(u)(x)) \leq f(x,e(u)(x))$ for $\mathcal{L}^d$-a.e.\ $x\in\Omega$.}   First,   in view of  \eqref{eq: general minimization}  and \eqref{eq: general minimization2}, we get $\mathbf{m}_{\mathcal{E}}(\bar{\ell}_{ \xi },Q_{\rho}(x)) \le \mathbf{m}^{1,p}_{\mathcal{E}}(\bar{\ell}_{ \xi },Q_{\rho}(x))$  for all $\xi \in \R^{d \times d}$,  where we recall the notation $\bar{\ell}_{\xi}(y) = \xi y $ for $y \in \R^d$ in \eqref{eq: specified not}.  Then (\ref{eq: f^epsilon_infty-new}) implies 
\begin{align}\label{es: step1-1}
f_\infty(x,{\rm sym}(\xi))&=\limsup_{\rho\to 0^+}\frac{\mathbf{m}_{\mathcal{E}}(\bar{\ell}_{\xi}, Q_{\rho}(x))}{\rho^d}  \le \limsup_{\rho\to 0^+}\frac{\mathbf{m}^{1,p}_{\mathcal{E}}(\bar{\ell}_{\xi},Q_{\rho}(x))}{\rho^d}. 
\end{align}
On the other hand,  by Remark \ref{rem: new rem} we get that the density $f$ in \eqref{eq: bulk assumption} coincides with $f_0$ given in \eqref{eq: bulk repr}.    Thus, by \eqref{eq: bulk repr}  and \eqref{eq: almost the same1}  we find 
\begin{align}\label{es: step1-2}
f(x,{\rm sym}(\xi)) =   \limsup_{\rho \to 0^+}   \frac{\mathbf{m}^{1,p}_{\mathcal{E}}(\bar{\ell}_{\xi},Q_{\rho}(x))}{\rho^d}. 
\end{align} 
By combining \eqref{es: step1-1}--\eqref{es: step1-2} we obtain $f_\infty(x, {\rm sym}(\xi)  ) \leq f(x, {\rm sym}(\xi) )$.  This concludes Step 1.

\noindent  \emph{Step 2: $f_\infty(x,e(u)(x)) \geq f(x,e(u)(x))$ for $\mathcal{L}^d$-a.e.\ $x\in\Omega$.} By Remark \ref{rem: invariance}(i) and the Radon-Nikod\'ym Theorem we have for $\mathcal{L}^d$-a.e.\ $x \in \Omega$ that
\begin{align}\label{eq: our (4.5) G.P.}
f_\infty(x,e(u)(x))=\lim_{\rho\to 0^+ }\frac{\mathcal{E}(u,Q_\rho(x))}{\rho^d}<\infty.
\end{align}
Let $(u_n)_n$ be a recovering sequence for $\mathcal{E}(u,\Omega)$. This along with the growth condition \eqref{eq: general bound2-new} yields that the sequence $(\mathcal{H}^{d-1}(J_{u_n}))_n$ is uniformly bounded. Thus,  up to a subsequence (not relabeled),  there exists a finite positive Radon measure $\mu$ such that
\begin{align}\label{eq: our (4.6) G.P.}
\mu_n:=\mathcal{H}^{d-1}\mres J_{u_n}\rightharpoonup \mu\quad \textit{\emph{ weakly$^*$
 in the sense of measures}}.
\end{align}
Let us notice that for $\mathcal{L}^d$-a.e.\ $x \in \Omega$ we have that 
\begin{align}\label{eq: blow-up measure}
\limsup_{\rho \to 0^+}\frac{\mu\big(\overline{Q_\rho(x)}\big)}{\rho^{d-1}}=0.
\end{align}
Indeed, by contradiction we suppose that there exists a Borel set $B\subset \Omega$ with $\mathcal{L}^d(B)>0$ and $t >0$ such that 
$$
\limsup_{\rho \to 0^+}\frac{\mu\big(\overline{Q_\rho(x)}\big)}{\rho^{d-1}}> t,
$$
for all $x\in B$. Then, as a consequence of \cite[Theorem 2.56]{Ambrosio-Fusco-Pallara:2000} we would get
$\mu \mres B\geq t \mathcal{H}^{d-1}\mres B$
implying that $\mu(B)=\infty$. But this is a contradiction since $\mu$ is finite. Since $u\in GSBD^p(\Omega)$, the approximate gradient $\nabla u (x)  $ exists for $\mathcal{L}^d$-a.e.\ $x\in\Omega$, see Lemma \ref{lemma: approx-grad}. As the main inequality needs to hold for $\mathcal{L}^d$-a.e.\ $x\in\Omega$, we can suppose that \eqref{eq: our (4.5) G.P.} and \eqref{eq: blow-up measure} are satisfied at $x$ and that the approximate gradient $\nabla u (x)$ exists. Since $\mathcal{E}(u,\cdot)$ is a Radon measure, there exists a subsequence $(\rho_i)_i\subset (0,\infty)$ with $\rho_i \searrow 0$ as $i\to \infty$ such that $\mathcal{E}(u,\partial Q_{\rho_i}(x))=0$  for every $i\in\mathbb{N}$, and  such that  the representation formula \eqref{eq: bulk assumption} holds along $(\rho_i)_i$, i.e., 
\begin{align}\label{eq: 4.20 GP-bulki}
 f(x, e(u)(x) ):= \lim_{ i\to \infty}   \limsup_{n \to \infty}    \frac{1}{\rho_i^{d}}{\mathbf{m}^{1,p}_{\mathcal{E}_n}( \bar{\ell}_{\nabla u(x)},Q_{\rho_i}(x))}.
\end{align}
By  Remark \ref{rem: recov}, for every $i\in\mathbb{N}$ there exists $n_i\in\mathbb{N}$ such that for every $n\geq n_i$, we get
\begin{align}\label{eq: our (4.8) G.P.}
\frac{\mathcal{E}(u,Q_{\rho_i}(x))}{\rho^d_i} &\geq \frac{\mathcal{E}_n(u_n,Q_{\rho_i}(x))}{\rho^d_i}-\frac{1}{i}.
\end{align}
We define the functions 
\begin{align*}
\tilde{u}^i_n(y) := \frac{u_n(x+\rho_iy)- u_n(x)  }{\rho_i} \quad \quad \text{and} \quad \quad  \tilde{u}^i(y): =  \frac{u(x+\rho_iy)-u(x)}{\rho_i}  \quad \quad \text{ for $y \in Q_1$, }
\end{align*}
and note that $\tilde{u}^i_n\to   \tilde{u}^i$ in measure on $Q_1$ as $n \to \infty$  since $u_n \to u$ in measure on $\Omega$.  We now show   by a diagonal argument that, up to passing to larger $n_i \in  \mathbb{N}$, the sequence $v_i := \tilde{u}^i_{n_i}$  satisfies 
\begin{align}\label{eq: our (4.9) G.P.}
v_i \to  \bar{\ell}_{\nabla u(x)} \quad \text{in measure on $Q_1$},   
\end{align}
(recall  that $\bar{\ell}_{\nabla u(x)}(y) = \nabla u(x) y$ for $y \in \R^d$, see   \eqref{eq: specified not}) and 
 \begin{align}\label{eq: our (4.10) G.P.}
\lim_{i\to\infty}\mathcal{H}^{d-1}(J_{v_i})=0.
\end{align}
Indeed, recall that $\tilde{u}^i_n\to   \tilde{u}^i$ in measure on $Q_1$ as $n \to \infty$. Since $u$ is approximately differentiable at $x$, we also have $\tilde{u}^i \to \bar{\ell}_{\nabla u(x)}$ in measure on $Q_1$ as $i \to \infty$, cf.\ Lemma \ref{lemma: approx-grad}. Consequently, \eqref{eq: our (4.9) G.P.} can be achieved. Moreover,  by  a change of variables   and by recalling (\ref{eq: our (4.6) G.P.}) we get 
\begin{align*}
\limsup_{n\to\infty}  \mathcal{H}^{d-1}(J_{\tilde{u}^i_{n}})= \limsup_{n\to\infty} \frac{\mathcal{H}^{d-1}(J_{u_n}\cap Q_{\rho_i}(x))}{\rho_i^{d-1}}\leq   \limsup_{n\to \infty} \frac{\mu_n\big( \overline{Q_{\rho_i}(x)}  \big)}{\rho_i^{d-1}}  \le  \frac{ \mu \big( \overline{Q_{\rho_i}(x)} \big)}{\rho_i^{d-1}}.
\end{align*}
Thus, by  \eqref{eq: blow-up measure} we get
\begin{align}\label{eq: ce serve per applicare Larsen ii)}
\limsup_{i\to\infty}\,\limsup_{n\to\infty}\mathcal{H}^{d-1}(J_{\tilde{u}^i_{n}}) =0.
\end{align}
Then, by a diagonal argument, (\ref{eq: our (4.10) G.P.})  can be ensured. Note by \eqref{eq: 4.20 GP-bulki}  that  we can choose $(n_i)_i$  such that additionally  we have 
\begin{align}\label{eq: 4.20 GP-better-bulki}
 f(x, e(u)(x)):= \lim_{ i\to \infty}   \frac{1}{\rho_i^{d}}{\mathbf{m}^{1,p}_{\mathcal{E}_{n_i}}( \bar{\ell}_{\nabla u(x)},Q_{\rho_i}(x))}.
\end{align}
By \eqref{eq: basic energy},  \eqref{eq: our (4.8) G.P.},  and  by the change of variables $y' = x+\rho_i y$   we get 
\begin{align}\label{eq: our (4.8) G.P.-extra}
\int_{Q_1}f_{n_i}\big(x+\rho_i y,e(v_i)(y)\big)\,{\rm d}y = \frac{\int_{Q_{\rho_i}(x)}f_{n_i}( y',  e(u_{n_i})( y'  ))\,{\rm d}  y'  }{\rho^d_i} \le 
\frac{\mathcal{E}(u,Q_{\rho_i}(x))}{\rho^d_i}  +\frac{1}{i}.
\end{align}
In addition, taking into consideration (\ref{eq: our (4.5) G.P.})   we get
\begin{align}\label{eq: our (4.11) G.P.}
\limsup_{i\to\infty}\int_{Q_1}f_{ n_i}\big(x+\rho_i y, e(v_i)( y )\big)\,{\rm d}y \le \lim_{i\to\infty}\frac{\mathcal{E}(u,Q_{\rho_i}(x))}{\rho_i^d} = f_\infty(x,e(u)(x)).
\end{align}
Let  us  observe that the sequence $(v_i)_i \subset GSBD^p(  Q_1) $ satisfies the assumptions of Lemma \ref{lem: our Larsen}. Indeed, by  \eqref{eq: our (4.11) G.P.} and  the growth condition (\ref{eq: general bound}) we have that (\ref{eq: properties-part0})(i) holds true. Thanks to \eqref{eq: our (4.10) G.P.}, we  get (\ref{eq: properties-part0})(ii), while condition (\ref{eq: properties-part0})(iii) is a consequence of  (\ref{eq: our (4.9) G.P.}). Thus, by Lemma \ref{lem: our Larsen} applied to the sequence $(v_i)_i \subset GSBD^p( Q_1) $ there exists a sequence $(w_i)_i\subset  W^{1,p}(Q_1;\R^d) $ such that (\ref{eq: Larsen properties}) holds true, and $(|\nabla w_i |^p)_i$ is equiintegrable. In particular, from this latter fact,  (\ref{eq: Larsen properties})(i), and \eqref{eq: general bound} we get
\begin{align}\label{ce serve sicuro}
\limsup_{i\to\infty}\int_{Q_1}f_{n_i}\big(x+\rho_i y, e(v_i)(  y  )\big)\,{\rm d}y= \limsup_{i\to\infty}\int_{Q_1}f_{n_i}\big(x+\rho_i y, e(w_i)(y  )\big)\,{\rm d}y.
\end{align} 
Moreover, by (\ref{eq: Larsen properties})(ii) and \eqref{eq: our (4.9) G.P.} we get that
\begin{align}\label{eq: strong conv of w_i to l}
\lim_{i\to \infty}\|w_i - \bar{\ell}_{\nabla u(x)}  \|_{L^p(Q_1)} = 0. 
\end{align}
%
%
We now modify the sequence $(w_i)_i$ in such a way  that  it will attain the boundary datum of the function $\bar{\ell}_{\nabla u(x)}$  in  a neighborhood of $\partial Q_1$  (see \cite[Proof of Theorem 5.2(b), Step 2]{Sto lavoro GSBV} for a similar argument). By \cite[Theorem 19.4]{DalMaso:93} we know that the family of functionals $\tilde{\mathcal{F}}_i$  defined by
\begin{align}\label{eq: tildF}
\tilde{\mathcal{F}}_i(u,A) = \int_A  f_{n_i}(x+\rho_i y, e(u)(y) ) \, {\rm d}y 
\end{align}
for $A \in \mathcal{A}(Q_1)$ and  $u \in W^{1,p}(A;\R^d)$  satisfies uniformly the Fundamental Estimate (see \cite[Chapter~18]{DalMaso:93}):  for fixed $0<\eps<1$ there exists a constant $C(\eps)>0$, and a sequence $ (z_i)_i  \subset W^{1,p}(Q_1;\R^d)$ with $z_i = \bar{\ell}_{\nabla u(x)}$ in a neighborhood of $\partial Q_1$ for all $i \in \N$ such that
\begin{align*}
\tilde{\mathcal{F}}_i(z_i,Q_1)\leq (1+\eps)\left(\tilde{\mathcal{F}}_i(w_i,Q_1) + \tilde{\mathcal{F}}_i(\bar{\ell}_{\nabla u(x)},Q_1\setminus  \overline{Q_{1-\eps}} )  \right) + C(\eps)\|w_i - \bar{\ell}_{\nabla u(x)}  \|^{ p }_{L^p( Q_1  )} + \eps
\end{align*}
(more precisely, $z_i:= \varphi w_i + (1-\varphi) \bar{\ell}_{\nabla u(x)}$, where $\varphi$ lies in a finite collection of cut-off functions $\varphi\in C^{\infty}_c(Q_1;  [0,1]  )$,  with   $\varphi=1$ in $Q_{1-\eps}$).  Since $\mathcal{L}^d(Q_1\setminus Q_{1-\eps})\leq d\, \eps$, thanks to the growth condition (\ref{eq: general bound}) and  (\ref{eq: strong conv of w_i to l}),  we get that 
\begin{align}\label{eq: fund est Maso2}
 \limsup\nolimits_{i\to \infty}  \tilde{\mathcal{F}}_i(z_i,Q_1)\leq (1+\eps) \limsup\nolimits_{i\to \infty} \tilde{\mathcal{F}}_i(w_i,Q_1) + d\eps(1+\eps)\beta(1+|e(u)(x)|^p) + \eps.
\end{align}
Then, by (\ref{eq: our (4.11) G.P.}), (\ref{ce serve sicuro}), \eqref{eq: tildF}, and \eqref{eq: fund est Maso2} we derive
\begin{align}\label{eq: piece on the right}
\limsup\nolimits_{i\to \infty} \tilde{\mathcal{F}}_i(z_i,Q_1) \leq (1+\eps) f_{\infty}(x,e(u)(x)) +   d\eps(1+\eps)\beta(1+|e(u)(x)|^p) + \eps.
\end{align}
On the other hand, by a change of  variables  we have that 
\begin{align}\label{eq: ce serve a sinistra}
\tilde{\mathcal{F}}_{i}(z_i,Q_1)= \int_{Q_1}f_{ n_i} (x+\rho_i y', e(z_i)(y'))\, {\rm d}  y'  = \frac{1}{\rho_i^d}\int_{Q_{\rho_i}(x)}f_{n_i} ( y,  e(\tilde{z}_i)(y))\, {\rm d}y,
\end{align}
 where $ \tilde{z}_i \in W^{1,p}(Q_{\rho_i}(x);\R^d)$ is defined by $\tilde{z}_i(y):= \rho_iz_i((y-x)/\rho_i)+\bar{\ell}_{\nabla u(x)}x$ for $y\in Q_{\rho_i}(x)$. Since $z_i = \bar{\ell}_{\nabla u(x)}$ in a neighborhood of $\partial Q_1$, we get  $\tilde{z}_i =\bar{\ell}_{\nabla u(x)}$  in a neighborhood of $\partial Q_{\rho_i}(x)$. Thus,    \eqref{eq: general minimization2},  \eqref{eq: almost the same1},  and  \eqref{eq: ce serve a sinistra} imply 
$$\frac{1}{\rho_i^d}\textbf{m}^{1,p}_{ \mathcal{E}_{n_i}}\left(\bar{\ell}_{ \nabla u(x)  },Q_{\rho}(x)  \right) \le \frac{1}{\rho_i^d}\int_{Q_{\rho_i}(x)}f_{n_i} (y, e(\tilde{z}_i)(y))\, {\rm d}y = \tilde{\mathcal{F}}_{i}(  {z}_i,  Q_1). $$
This along with (\ref{eq: piece on the right}) and the arbitrariness of $\eps$ yields 
\begin{align*}
\limsup_{i\to\infty}\frac{1}{\rho_i^d}\textbf{m}^{1,p}_{\mathcal{E}_{n_i}}\left(\bar{\ell}_{ \nabla u(x)},Q_{\rho}(x)  \right) \leq f_{\infty}(x,e(u)(x)).
\end{align*}
Thanks to \eqref{eq: 4.20 GP-better-bulki},  we obtain the desired inequality $f(x,e(u)(x))\leq f_{\infty}(x,e(u)(x))$. This concludes the proof.


\subsection{Surface part in Theorem \ref{thm: main thm d=2}: Proof of \eqref{eq: g_infty=g}}\label{sec: sub3} 

The proof is again split into two inequalities. The first one is obtained similarly as in Subsection \ref{sec: sub2}.  The other inequality is the only point where we need to restrict ourselves to dimension $d=2$, due to the application of Lemma  \ref{lemma: partition}.

\noindent \emph{Step 1: $g_\infty(x,[u](x),\nu_u(x)) \leq g(x,\nu_u(x))$ for $\mathcal{H}^{d-1}$-a.e.\ $x\in J_u$.}  With the notation in \eqref{eq: jump competitor}, we set $\hat{u}^x := {u}_{x,[u](x),0,\nu_u(x)}$ for brevity.    First, in view of \eqref{eq: general minimization}  and   \eqref{eq: general minimization3},  we get  $\mathbf{m}_{\mathcal{E}}(  \hat{u}^x,  Q^{\nu_u(x)}_{\rho}(x)) \le \mathbf{m}^{PR}_{\mathcal{E}}(\hat{u}^x,Q^{\nu_u(x)}_{\rho}(x))$. Then (\ref{eq: f^epsilon_infty-new}) entails  
\begin{align}\label{es: step1-1XXXX}
g_\infty(x,[u](x),\nu_u(x))&=\limsup_{\rho\to 0^+}\frac{\mathbf{m}_{\mathcal{E}}(\hat{u}^x,Q^{\nu_u(x)}_{\rho}(x))}{\rho^{d-1}}  \le \limsup_{\rho\to 0^+}\frac{\mathbf{m}^{PR}_{\mathcal{E}}(\hat{u}^x,Q^{\nu_u(x)}_{\rho}(x))}{\rho^{d-1}}. 
\end{align}
On the other hand,  by Remark \ref{rem: new rem} we get that the density $g$ in \eqref{eq: surface assumption} coincides with $g_0$ given in \eqref{eq: jump energy density}.   Therefore,  \eqref{eq: almost the same2}, \eqref{eq: general minimization3-better}, and  \eqref{eq: jump energy density} yield   
\begin{align}\label{es: step1-2XXXX}
g(x,\nu_u(x)) =   \limsup_{\rho \to 0^+}   \frac{\mathbf{m}^{PC}_{\mathcal{E}}(\bar{u}_{x,\nu_u(x)},Q^{\nu_u(x)}_{\rho}(x))}{\rho^{d-1}} =  \limsup_{\rho \to 0^+}   \frac{\mathbf{m}^{PR}_{\mathcal{E}}(\hat{u}^{x},Q^{\nu_u(x)}_{\rho}(x))}{\rho^{d-1}}, 
\end{align} 
where we used the notation in \eqref{eq: specified not}.    By  \eqref{es: step1-1XXXX}--\eqref{es: step1-2XXXX}  we obtain $g_\infty(x,[u](x),\nu_u(x)) \leq g(x,\nu_u(x))$.

\noindent \emph{Step 2: $g_\infty(x,[u](x),\nu_u(x)) \geq g(x,\nu_u(x))$ for $\mathcal{H}^{d-1}$-a.e.\ $x\in J_u$.}  We recall that in this step we explicitly use that $d=2$.  Still, for later purposes in Subsection \ref{sec: sub4}, we write $d$ instead of $2$ whenever the arguments hold in every dimension.  In view of \eqref{eq: funci-old}  and the Radon-Nikod\'ym Theorem,  we get  for $\mathcal{H}^{ d-1}$-a.e.\ $x \in J_u$ that 
\begin{align}\label{eq: Radon Nykodim step 4}
g_\infty(x,[u](x),\nu_u(x))= \lim_{\rho\to 0^+}\frac{1}{\rho^{d-1}}{\mathcal{E}(u,Q^{\nu_u(x)}_\rho(x))}<+\infty.
\end{align}
Moreover, we choose  $(\rho_i)_i \subset (0,\infty)$ such that $\rho_i \searrow 0$, $\mathcal{E}(u,\partial Q^{\nu_u(x)}_{\rho_i}(x))=0$ for every $i\in\mathbb{N}$, and such that  the representation formula \eqref{eq: surface assumption} holds along $(\rho_i)_i$, i.e.,   
\begin{align}\label{eq: 4.20 GP}
 g(x,\nu_u  (x)  ):= \lim_{ i\to \infty  }   \limsup_{n \to \infty}   \frac{1}{\rho_i^{d-1}}{\mathbf{m}^{PR}_{\mathcal{E}_n}( \bar{u}_{x,\nu_u(x)},Q^{\nu_u(x)}_{\rho_i}(x))}.
\end{align}
 In what follows,  since  the main inequality needs to hold for $\mathcal{H}^{d-1}$-a.e.\ $x \in J_u$, we can fix $x \in J_u$  such that  the above properties hold.  We  write $\bar{\nu}$ in place of $\nu_u(x)$ for notational simplicity.   

Let $(u_n)_n\subset GSBD^p(\Omega)$ be a recovery sequence for $\mathcal{E}(u,\Omega)$.  As  $\mathcal{E}(u,\partial Q^{\bar{\nu}}_{\rho_i}(x))=0$,  by  Remark~\ref{rem: recov}  we get   that for every $i \in \N$ there exists $n_i \in \N$ such that for all $n \ge n_i$ it holds that     
\begin{align}\label{eq: 4.21 GP-old}
\frac{1}{\rho_i^{d-1}}{\mathcal{E}(u,Q^{\bar{\nu}}_{\rho_i}(x))}&\ge \frac{1}{\rho_i^{d-1}}{\mathcal{E}_n(u_n,Q^{\bar{\nu}}_{\rho_i}(x))} - \frac{1}{i}.
\end{align}
 We define $\tilde{u}^i_n(y)=u_{n}(x+\rho_i y )$ and $\tilde{u}^i(y)=u(x+\rho_i y )$  for $y \in Q^{\bar{\nu}}_1$ and note that $\tilde{u}^i_n \to \tilde{u}^i$ in measure on $ Q^{\bar{\nu}}_1$ since $u_n \to u$ in measure on $\Omega$. Since $x\in J_u$ is an approximate jump point, we find $\tilde{u}^i \to u_{0,u^+(x),u^-(x), \bar{\nu}}$ in measure for $i \to \infty$, see \eqref{0106172148}  and recall notation \eqref{eq: jump competitor}.  Thus, by a diagonal argument, up to passing to larger $n_i \in \N$, we can suppose that $n_i \to +\infty$ as $i\to \infty$ and
\begin{align}\label{eq: 4.21 GP-new}
v_i \to  u_{0,u^+(x),u^-(x),\bar{\nu} } \quad \text{ in measure on $Q^{\bar{\nu}}_1$},
\end{align}
where we define $v_i \in GSBD^p(Q^{\bar{\nu}}_1)$ by $v_i = \tilde{u}^i_{n_i}$. Note that  by \eqref{eq: 4.20 GP}  we can choose $(n_i)_i$  such that additionally  
\begin{align}\label{eq: 4.20 GP-better}
 g(x, \bar{\nu}  ):= \lim_{ i\to \infty  }   \frac{1}{\rho_i^{ d-1}}{\mathbf{m}^{PR}_{\mathcal{E}_{n_i}}( \bar{u}_{x,\bar{\nu}},Q^{\bar{\nu}}_{\rho_i}(x))}
\end{align}
holds.  By a  change of variables, and by using   \eqref{eq: basic energy}  and \eqref{eq: 4.21 GP-old}   we get  
\begin{align}\label{eq: 4.21 GP}
\int_{J_{v_i}\cap Q^{\bar{\nu}}_{1}}g_{ n_i  }(x+\rho_iy,\nu_{v_i} (y)  )\,{\rm d}\mathcal{H}^{d-1}(y)& =  \frac{1}{\rho_i^{d-1}}{\int_{J_{u_{n_i}}\cap Q^{\bar{\nu}}_{\rho_i}(x)}g_{n_i}(z,\nu_{u_{n_i}} (z)  )\,{\rm d}\mathcal{H}^{d-1}(z)} \notag\\
&  \le \frac{1}{\rho_i^{d-1}}{\mathcal{E}(u,Q^{\bar{\nu}}_{\rho_i}(x))} + \frac{1}{i}.
\end{align}
 We also define    $\tilde{v}_i(y):= v_i(y)-u^-(x)$ for $y \in Q_1^{\bar{\nu}}$.  We check that $(\tilde{v}_i)_i$   satisfies the assumptions of Lemma \ref{lemma: partition}.  First, to see \eqref{eq: properties-part}(ii), we use \eqref{eq: general bound2-new}, \eqref{eq: Radon Nykodim step 4}, and \eqref{eq: 4.21 GP} to find 
\begin{align}\label{eq: sup on jump} 
\sup_{i \in \N} \mathcal{H}^{d-1}(J_{\tilde{v}_i} \cap Q^{\bar{\nu}}_{1}) = \sup_{i \in \N} \mathcal{H}^{d-1}(J_{{v}_i} \cap Q^{\bar{\nu}}_{1})  \le \frac{1}{\alpha} \sup_{i \in \N} \Big(  \frac{1}{ \rho_i^{d-1}}{\mathcal{E}(u,Q^{\bar{\nu}}_{\rho_i}(x))} + \frac{1}{i}\Big) < +\infty.
\end{align}
  Now we show  \eqref{eq: properties-part}(i), namely $\lim_{i\to \infty}\|e(\tilde{v}_i) \|_{L^p(Q_1^{\bar{\nu}})}=0$. Indeed, using the growth condition (\ref{eq: general bound}),  \eqref{eq: basic energy},   and a  change of variables   we get  
\begin{align*}
\int_{Q^{\bar{\nu}}_1}|e(\tilde{v}_i)(y) |^p\,{\rm d}y &= \frac{\rho_{i}^p}{\rho_{i}^d}\int_{Q^{\bar{\nu}}_{\rho_{i}}(x)}|e(u_{n_i})( z ) |^p\,{\rm d}  z  \leq \frac{\rho_i^{p-1} }{\alpha} \frac{1}{\rho_i^{d-1}}  {\mathcal{E}_{n_i}\big(u_{n_i},Q^{\bar{\nu}}_{\rho_{i}}(x)\big)}.
\end{align*} 
 Then, by \eqref{eq: Radon Nykodim step 4},  \eqref{eq: 4.21 GP-old}, and the fact that $\rho_i \to 0$ as $i \to +\infty$ we conclude $\lim_{i\to \infty}\|e(\tilde{v}_i) \|_{L^p(Q_1^{\bar{\nu}})}=0$. Finally,  \eqref{eq: properties-part}(iii)   holds   for $\zeta = u^+(x)-u^-(x)$  by the definition of $\tilde{v}_i$ and \eqref{eq: 4.21 GP-new}.  Thus, thanks to Lemma \ref{lemma: partition},  we can define a sequence of piecewise constant functions   $(w_i)_i \subset PC(Q^{\bar{\nu}}_1)$   by
\begin{align*}
w_i(y):= 
\begin{cases}
0 \quad &y\in S_i^-,\\
 e_1  \quad &y\in S_i^+,
\end{cases}
\end{align*}
where $S_i^{\pm}$ are given in Lemma \ref{lemma: partition},  and $e_1 = (1,0)$. (From now on, we explicitly set $d=2$.)   Let us observe  some  properties of the sequence $(w_i)_i$. First of all, by (\ref{eq: properties-part2})(i) we get that $w_i\to \bar{u}_{0,\bar{\nu}}$   strongly in $L^1(Q_1^{ \bar{\nu}};\R^2)$.   Thanks to property (\ref{eq: properties-part2})(ii), we have that $w_i = \bar{u}_{0,\bar{\nu}}$ in a neighborhood of $\partial Q^{\bar{\nu}}_1$. Finally, by property (\ref{eq: properties-part2})(iii)  it holds that 
\begin{align*}
\mathcal{H}^{1}(J_{w_i}\setminus J_{\tilde{v}_i}) = \mathcal{H}^{1}(J_{w_i}\setminus J_{{v}_i})  \leq  \eta_i, 
\end{align*}
with $\lim_{i\to \infty} \eta_i=0$. This along with  (\ref{eq: general bound2-new}) implies
\begin{align}\label{eq: 4.24 GP}
 \int_{ J_{w_i} \cap  Q^{\bar{\nu}}_1} g_{n_i}(x+\rho_{i} y,\nu_{w_i})\, {\rm d}\mathcal{H}^{1}(y) \le  \int_{J_{v_i} \cap Q^{\bar{\nu}}_1} g_{n_i}(x+\rho_{i} y,\nu_{v_i})\, {\rm d}\mathcal{H}^{1}(y)  +   \beta {\eta}_i. 
\end{align}
Due to (\ref{eq: Radon Nykodim step 4}), (\ref{eq: 4.21 GP}), and (\ref{eq: 4.24 GP})  we get
\begin{align}\label{eq: 4.24XXX GP}
\limsup_{i \to \infty} \int_{ J_{w_i} \cap  Q^{\bar{\nu}}_1} g_{n_i}(x+\rho_{i} y,\nu_{w_i})\, {\rm d}\mathcal{H}^{1}(y) \le \limsup_{i \to \infty}  \frac{1}{\rho_i}{\mathcal{E}(u,Q^{\bar{\nu}}_{\rho_i}(x))} =  g_\infty(x,[u](x),\bar{\nu}).
\end{align}
By defining $ \tilde{w}_i(z)  = w_i(  (z-x)/\rho_i   )  $ for $z \in Q_{\rho_{i}}^{\bar{\nu}}(x)$ and rescaling back to $Q_{\rho_{i}}^{\bar{\nu}}(x)$ we obtain
$${\int_{ J_{w_i} \cap  Q^{\bar{\nu}}_1} g_{n_i}(x+\rho_{i} y,\nu_{w_i}  (y)   )\, {\rm d}  \mathcal{H}^{1}(y)   = \frac{1}{\rho_i} \int_{ J_{\tilde{w}_i} \cap  Q_{\rho_{i}}^{\bar{\nu}}(x)} g_{n_i}(z,\nu_{\tilde{w}_i}  (z)   )\,  {\rm d}\mathcal{H}^{1}(z).}$$
 Then \eqref{eq: 4.24XXX GP} yields 
\begin{align}\label{eq: 4.24XXX GP-last}
 \limsup_{i \to \infty}  \frac{1}{\rho_i} \int_{ J_{\tilde{w}_i} \cap  Q_{\rho_{i}}^{\bar{\nu}}(x)} g_{n_i}(z,\nu_{\tilde{w}_i}  (z)  )\, {\rm d}\mathcal{H}^{1}(z) \le g_\infty(x,[u](x),\bar{\nu}). 
 \end{align}
Observe that by construction $\tilde{w}_i= u_{x,\bar{\nu}}$ in a neighbourhood of $\partial Q_{\rho_i}^{ \bar{\nu}}(x)$. Therefore, by   \eqref{eq: general minimization3}  and \eqref{eq: almost the same2}   we find
$$  \mathbf{m}^{PR}_{\mathcal{E}_{n_i}}(u_{x,\bar{\nu}},  Q_{\rho_{i}}^{\bar{\nu}}(x) )  \le \int_{ J_{\tilde{w}_i} \cap  Q_{\rho_{i}}^{\bar{\nu}}(x)} g_{n_i}(z,\nu_{\tilde{w}_i}  (z)  )\, {\rm d}\mathcal{H}^{1}(z) + \beta \rho_i^2. $$
This along with \eqref{eq: 4.24XXX GP-last} shows 
$$ \limsup_{i \to \infty}  \frac{1}{\rho_i}  \mathbf{m}^{PR}_{\mathcal{E}_{n_i}}(u_{x,\bar{\nu}},  Q_{\rho_{i}}^{\bar{\nu}}(x) ) \le  g_\infty(x,[u](x),\bar{\nu}). $$
By \eqref{eq: 4.20 GP-better}, writing again  $\nu_u(x)$ in place of $\bar{\nu}$, we conclude $ g(x,\nu_u(x)) \le  g_\infty(x,[u](x),\nu_u(x))$. \eop

\subsection{Surface part in Theorem \ref{thm: main thm d=2}: Proof of \eqref{eq: g_infty=g-relax}}\label{sec: sub4}
We proceed  with the proof of  \eqref{eq: g_infty=g-relax}. Recall that $g_n = h$ for all $n \in \N$, and therefore $\mathcal{E}_n$ takes the form 
$$ \mathcal{E}_n(u,A) = \int_{A}  f_n\big(x, e(u)(x) \big) \, {\rm d}x + \int_{J_u \cap A} h\big(x,\nu_u(x)\big) \, {\rm d} \mathcal{H}^{d-1}. $$
We  argue  along the lines of the proof in Subsection \ref{sec: sub3} and replace the application of Lemma~\ref{lemma: partition} by Proposition \ref{prop: surface int repr-relax} and the lower semicontinuity result in  Theorem \ref{thm: GSBD LSC}.

\noindent\emph{Step 1: $g_\infty\big(x,[u](x),\nu_u(x)\big) \le \bar{h} (x,\nu_u(x))$ for $\mathcal{H}^{d-1}$-a.e.\ $x\in J_u$.}   We proceed exactly as in Step~1 in  Subsection \ref{sec: sub3}  with the only difference that in place of \eqref{eq: surface assumption} and Proposition \ref{prop: surface int repr}  we use \eqref{eq: jump energy density-relax2}.

\noindent \emph{Step 2: $g_\infty\big(x,[u](x),\nu_u(x)\big) \ge \bar{h} (x,\nu_u(x))$ for $\mathcal{H}^{d-1}$-a.e.\ $x\in J_u$.} We first follow the lines of Step~2 in Subsection \ref{sec: sub3}. We may suppose that $x \in J_u$ satisfies 
\begin{align}\label{eq: Radon Nykodim step 4-relax}
g_\infty(x,[u](x),\nu_u(x))= \lim_{\rho\to 0^+}\frac{1}{\rho^{d-1}}{\mathcal{E}(u,Q^{\nu_u(x)}_\rho(x))}<+\infty.
\end{align}
By $(u_n)_n$ we denote a recovery sequence for $\mathcal{E}(u,\Omega)$. In a similar fashion to \eqref{eq: 4.21 GP-old}--\eqref{eq: 4.21 GP-new}, we find  sequences  $(\rho_i)_i$ and    $(n_i)_i$    with $\rho_i \to 0$  and  $n_i \to +\infty$ as $i \to \infty$ such that 
\begin{align}\label{eq: 4.21 GP-old-relax}
\frac{1}{\rho_i^{d-1}}{\mathcal{E}(u,Q^{\bar{\nu}}_{\rho_i}(x))}&\ge \frac{1}{\rho_i^{d-1}}{\mathcal{E}_{n_i}(u_{n_i},Q^{\bar{\nu}}_{\rho_i}(x))} - \frac{1}{i},
\end{align}
where we again use the shorthand $\bar{\nu} = \nu_u(x)$, and 
\begin{align}\label{eq: 4.21 GP-new-relax}
v_i \to  u_{0,u^+(x),u^-(x),\bar{\nu} } \quad \text{ in measure on $Q^{\bar{\nu}}_1$},
\end{align}
where $v_i(y) := u_{n_i}(x + \rho_i y)$ for $y \in Q^{\bar{\nu}}_1$.  
 By  \eqref{eq: 4.21 GP-old-relax} and a  change of variables  we get 
\begin{align*}
\frac{1}{\rho_i^{d-1}}{\mathcal{E}(u,Q^{\bar{\nu}}_{\rho_i}(x))} &  \ge  \frac{1}{\rho_i^{d-1}}{\int_{J_{u_{n_i}}\cap Q^{\bar{\nu}}_{\rho_i}(x)} h(z,\nu_{u_{n_i}}  (z)  )\,{\rm d} \mathcal{H}^{d-1}(z)}-\frac{1}{i}\nonumber \\
&= \int_{J_{v_i}\cap Q^{\bar{\nu}}_{1}}h(x+\rho_i y,\nu_{v_i}(y))\,{\rm d}\mathcal{H}^{d-1}(y)-\frac{1}{i}.
\end{align*}
 As in \eqref{eq: sup on jump},  by using that $h$ satisfies \eqref{eq: general bound2-new},  we have $\sup_{i \in \N} \mathcal{H}^{d-1}(J_{v_i}\cap Q^{\bar{\nu}}_{1}) < +\infty$. By continuity of the function $h$,  we thus find a sequence $(\eta_i)_i \subset (0,\infty)$  with $\eta_i \to 0$ such that 
\begin{align}\label{eq: this cont is needed}
\frac{1}{\rho_i^{d-1}}{\mathcal{E}(u,Q^{\bar{\nu}}_{\rho_i}(x))} \ge \int_{J_{v_i}\cap Q^{\bar{\nu}}_{1}}h(x,\nu_{v_i}(y))\,{\rm d}\mathcal{H}^{d-1}(y)-\eta_i.
\end{align}
 As $\bar{h}(x,\cdot) \le {h}(x,\cdot)$ by \eqref{eq: jump energy density-relax},  we further get
\begin{align}\label{eq: LLL}
\frac{1}{\rho_i^{d-1}}{\mathcal{E}(u,Q^{\bar{\nu}}_{\rho_i}(x))} \ge \int_{J_{v_i}\cap Q^{\bar{\nu}}_{1}}\bar{h}(x,\nu_{v_i}(y))\,{\rm d}\mathcal{H}^{d-1}(y)-\eta_i.
\end{align}
Since $v_i$ converges in measure to $\hat{u} :=  u_{0,u^+(x),u^-(x),\bar{\nu} }$  on $Q^{\bar{\nu}}_1$ by  \eqref{eq: 4.21 GP-new-relax} and the density  $\bar{h}(x,\cdot)$  is symmetric jointly convex (see Corollary \ref{cor: surf}),  Theorem \ref{thm: GSBD LSC} along with \eqref{eq: LLL} and the fact that $\eta_i \to 0$ yields
$$\liminf_{i \to \infty} \frac{1}{\rho_i^{d-1}}{\mathcal{E}(u,Q^{\bar{\nu}}_{\rho_i}(x))} \ge \liminf_{i \to \infty}  \int_{J_{v_i}\cap Q^{\bar{\nu}}_{1}}  \bar{h}  (x,\nu_{v_i}(y))\,{\rm d}\mathcal{H}^{d-1}(y) \ge   \int_{J_{\hat{u} }   \cap Q^{\bar{\nu}}_{1}}  \bar{h} (x,\nu_{\hat{u}}(y)) \, {\rm d}\mathcal{H}^{d-1}(y). $$
The definition  of $\hat{u}$ implies   $\int_{J_{\hat{u} }   \cap Q^{\bar{\nu}}_{1}} \bar{h}  (x,\nu_{\hat{u}}(y)) \, {\rm d}\mathcal{H}^{d-1}(y) =  \bar{h}  (x,\bar{\nu})$. Now, by recalling the notation $\bar{\nu} = \nu_u(x)$ and by using \eqref{eq: Radon Nykodim step 4-relax} we conclude $g_\infty\big(x,[u](x),\nu_u(x)\big) \ge \bar{h} (x,\nu_u(x))$. \eop

With the results of Subsections \ref{sec: sub2}--\ref{sec: sub4}, Theorem \ref{thm: main thm d=2} is now completely proved. We close this subsection with the proof of Remark \ref{rem: not-cont}.

\begin{proof}[Proof of Remark \ref{rem: not-cont}]
First, as $h$ is continuous on $D$, identity \eqref{eq: g_infty=g-relax} clearly holds for $\mathcal{H}^{d-1}$-a.e.\ $x \in J_u \cap D$. Since $J_u \subset \overline{D} \cap \Omega$, it remains to address the case $x \in \partial D\cap  \Omega$. As $D$ has Lipschitz boundary, the outer normal $\nu_D(x)$ exists for $\mathcal{H}^{d-1}$-a.e.\ $x\in \partial D$. Then, as $J_u \subset \overline{D}$, we deduce that $\nu_u(x) = \nu_D(x) =: \nu_x$ for $\mathcal{H}^{d-1}$-a.e.\ $x \in J_u \cap \partial D$, i.e., we need to show
\begin{align*}
g_\infty\big(x,[u](x),\nu_x \big) = \bar{h} (x,\nu_x) \quad \quad \text{for $\mathcal{H}^{d-1}$-a.e.\ $x \in J_u \cap \partial D$}.
\end{align*}
As before, we split the proof into two inequalities.

\noindent\emph{Step 1: $g_\infty\big(x,[u](x),\nu_x \big) \le \bar{h} (x,\nu_x)$.} We first set $\hat{u}^x := {u}_{x,[u](x),0,\nu_x}$ (recall the notation in \eqref{eq: jump competitor}). In view of \eqref{eq: general minimization}  and   \eqref{eq: general minimization3}, we get  $\mathbf{m}_{\mathcal{E}}(  \hat{u}^x, Q^{\nu_x}_{\rho}(x)) \le \mathbf{m}^{PR}_{\mathcal{E}}(\hat{u}^x,Q^{\nu_x}_{\rho}(x))$. Then (\ref{eq: f^epsilon_infty-new}) gives  
\begin{align}\label{es: step1-1XXXX---}
g_\infty(x,[u](x),\nu_x)&=\limsup_{\rho\to 0^+}\frac{\mathbf{m}_{\mathcal{E}}(\hat{u}^x,Q^{\nu_x}_{\rho}(x))}{\rho^{d-1}}  \le \limsup_{\rho\to 0^+}\frac{\mathbf{m}^{PR}_{\mathcal{E}}(\hat{u}^x,Q^{\nu_x}_{\rho}(x))}{\rho^{d-1}}. 
\end{align}
At the end of the step we will show that, for given $\eps>0$, we can find $\bar{x} \in D$ such that 
\begin{align}\label{eq: main point}
{\rm (i)} \ \ & \Big| \limsup_{\rho \to 0+} \frac{\mathbf{m}^{PR}_{\mathcal{E}}(\hat{u}^x,Q^{\nu_x}_{\rho}(x))}{\rho^{d-1}}   - \limsup_{\rho \to 0+} \frac{\mathbf{m}^{PR}_{\mathcal{E}}(\hat{u}^{\bar{x}},Q^{\nu_x}_{\rho}(\bar{x}))}{\rho^{d-1}} \Big| \le C\eps \alpha^{-1} \Vert h \Vert_\infty, \\ \notag
{\rm (ii)} \ \ & |\bar{h}(x,\nu_x) - \bar{h}(\bar{x},\nu_x)| \le  C\eps \alpha^{-1} \Vert h \Vert_\infty 
\end{align} 
for some universal $C>0$. Then we conclude as follows:  by    \eqref{eq: almost the same2}, \eqref{eq: general minimization3-better},  \eqref{eq: jump energy density-relax2},  and the fact that $\bar{x} \in D$  we have 
\begin{align}\label{es: step1-2XXX}
\bar{h}(\bar{x},\nu_x) =   \limsup_{\rho \to 0^+}   \frac{\mathbf{m}^{PC}_{\mathcal{E}}(\bar{u}_{\bar{x},\nu_x},Q^{\nu_x}_{\rho}(\bar{x}))}{\rho^{d-1}} =  \limsup_{\rho \to 0^+}   \frac{\mathbf{m}^{PR}_{\mathcal{E}}(\hat{u}^{\bar{x}},Q^{\nu_x}_{\rho}(\bar{x}))}{\rho^{d-1}}.
\end{align} 
By  \eqref{es: step1-1XXXX---}, \eqref{eq: main point}(i), and \eqref{es: step1-2XXX}  we obtain $g_\infty(x,[u](x),\nu_x) \leq \bar{h}(\bar{x},\nu_x) + C\eps \alpha^{-1} \Vert h \Vert_\infty$. Then \eqref{eq: main point}(ii) yields $g_\infty(x,[u](x),\nu_x)   \le \bar{h}(x,\nu_x) + 2C\eps \alpha^{-1} \Vert h \Vert_\infty$. Since $\eps$ was arbitrary, we obtain the desired inequality.

To conclude this step, we need to show \eqref{eq: main point}. We only prove \eqref{eq: main point}(i) as  \eqref{eq: main point}(ii) follows along similar lines. For convenience, as in Proposition \ref{prop: surface int repr-relax}, we denote the surface integral with density $h$ by $\mathcal{S}$.  Let $\eps >0$ and suppose without restriction that $1/\eps \in \N$. By uniform continuity of $h$ in $D$ we can choose $\bar{x} \in D$ and $\delta = \delta(\eps) >0$ such that  
\begin{align}\label{h:cont}
\Vert h(y,\cdot) - h(\bar{y},\cdot) \Vert_\infty \le \eps \quad \quad \text{for all $y \in Q^{\nu_x}_{ \delta } (x) \cap D$, \  $\bar{y} \in Q^{\nu_x}_{\delta}(\bar{x}) \cap D$}.
\end{align}
Pick $0< \rho_0\le  \delta $ sufficiently  small such that $Q^{\nu_x}_{\rho_0}(\bar{x}) \subset D$ and choose $v \in PC(Q^{\nu_x}_{\rho_0}(\bar{x}))$ with $v = \bar{u}_{\bar{x},\nu_x}$ in a neighborhood of $\partial Q^{\nu_x}_{\rho_0}(\bar{x})$ such that    
\begin{align}\label{eq: minprob}
\mathcal{S}(v, Q^{\nu_x}_{\rho_0}(\bar{x})  ) \le \mathbf{m}^{PC}_{\mathcal{S}}(\bar{u}_{\bar{x},\nu_x},Q^{\nu_x}_{\rho_0}(\bar{x}))+ \rho^d_0.
\end{align}
Let us observe that
\begin{align}\label{eq: minprob2}
\mathcal{H}^{d-1}(J_v) \le  \alpha^{-1} \Vert h \Vert_\infty\rho_0^{d-1}.
\end{align}
Indeed, using $\bar{u}_{\bar{x},\nu_x}$ as a competitor, we get  $\mathbf{m}^{PC}_{\mathcal{S}}(\bar{u}_{\bar{x},\nu_x},Q^{\nu_x}_{\rho_0}(\bar{x})) \le \Vert h \Vert_\infty \rho_0^{d-1}$. This along with the lower bound in \eqref{eq: general bound2-new} shows \eqref{eq: minprob2}.

We now construct a competitor for $ \mathbf{m}^{PC}_{\mathcal{S}}(\bar{u}_{x,\nu_x},Q^{\nu_x}_{\rho}(x))$ for all $\rho \le \rho'$, where $0<\rho' \le  \delta$ is chosen sufficiently small such that 
\begin{align}\label{eq: Lippi}
 D \cap Q^{\nu_x}_{\rho}(x) \supset  \lbrace y \in Q^{\nu_x}_\rho  (x)  \colon \,   (y-x) \cdot \nu_x \le -\eps \rho  \rbrace.
  \end{align}
Now, we define $v^\rho \in PC(Q_\rho^{\nu_x}(x))$ by $v^\rho =  \bar{u}_{{x},\nu_x}$  on $Q^{\nu_x}_{\rho}(x) \setminus Q^{\nu_x}_{\rho(1-\eps)}(x)$ and on $Q^{\nu_x}_{\rho(1-\eps)}(x)$ we set  
\begin{align*}
v^\rho(y) = \begin{cases}
e_1 & \text{if } (y-x) \cdot \nu_x  > -\eps \rho,\\
v( \bar{x} +  \rho_0 (\eps \rho)^{-1}   (y-x_n)) & \text{if }  -2\eps\rho <   (y-x) \cdot \nu_x < - \eps\rho  \text{ and }  y  \in Q_n,\\
0 & \text{if } (y-x) \cdot \nu_x < -2\eps\rho,
\end{cases}
\end{align*}
where $(Q_n)_n$ denotes a partition of the set $\lbrace y\in Q^{\nu_x}_{\rho(1-\eps)}(x)\colon -2\eps\rho <   (y-x) \cdot \nu_x < - \eps\rho \rbrace$ consisting of $\eps^{1-d}$ cubes with sidelength $\eps\rho$,  and $x_n$ indicates the center of $Q_n$. Since $v = \bar{u}_{\bar{x},\nu_x}$ in a neighborhood of $\partial Q^{\nu_x}_{\rho_0}(\bar{x})$,  the functions $v^\rho$ have  $\mathcal{H}^{d-1}$-negligible jump on $\partial Q_n$.  Hence, we get
\begin{align*}
\mathcal{S}(v^\rho,Q^{\nu_x}_{\rho}(x))& \le   \sum\nolimits_{n=1}^{\eps^{1-d}}  \mathcal{S}(v^\rho,Q_n) + \Vert h \Vert_\infty\mathcal{H}^{d-1}\big(J_{v^\rho} \cap (Q^{\nu_x}_{\rho}(x) \setminus Q^{\nu_x}_{\rho(1-\eps)}(x))\big).
\end{align*}
Then,  due to $J_{v^\rho} \cap Q^{\nu_x}_{\rho(1-\eps)}(x) \subset  D $ (see \eqref{eq: Lippi}), a scaling argument and \eqref{h:cont}   imply
\begin{align*}
\mathcal{S}(v^\rho,Q^{\nu_x}_{\rho}(x))  \le  \rho^{d-1} \rho_0^{-(d-1)} \big( \mathcal{S}(v, Q^{\nu_x}_{\rho_0}(\bar{x})  )  + \eps \mathcal{H}^{d-1}(J_v ) \big)  + C\Vert h\Vert_\infty \eps \rho^{d-1},
\end{align*}
where   $C>0$ is a universal constant. This along  with \eqref{eq: minprob2} shows
 \begin{align*}
\mathcal{S}(v^\rho,Q^{\nu_x}_{\rho}(x))  \le  \rho^{d-1} \rho_0^{-(d-1)} \mathcal{S}(v, Q^{\nu_x}_{\rho_0}(\bar{x})  )   + C\eps\alpha^{-1}\Vert h\Vert_\infty \rho^{d-1}.
\end{align*}
By \eqref{eq: minprob} and  the fact that $v^\rho = \bar{u}_{x,\nu_x}$ in a neighborhood of $\partial Q^{\nu_x}_{\rho}(x)$  we conclude
$$ \limsup_{\rho \to 0+} \frac{\mathbf{m}^{PC}_{\mathcal{S}}(\bar{u}_{x,\nu_x},Q^{\nu_x}_{\rho}(x))}{\rho^{d-1}} \le  \limsup_{\rho \to 0+} \frac{\mathbf{m}^{PC}_{\mathcal{S}}(\bar{u}_{\bar{x},\nu_x},Q^{\nu_x}_{\rho}(\bar{x}))}{\rho^{d-1}} + C\eps\alpha^{-1}\Vert h\Vert_\infty.$$
On the other hand, \eqref{eq: h-inequaly} and \eqref{h:cont} immediately yield the converse inequality  
$$ \limsup_{\rho \to 0+} \frac{\mathbf{m}^{PC}_{\mathcal{S}}(\bar{u}_{x,\nu_x},Q^{\nu_x}_{\rho}(x))}{\rho^{d-1}} \ge  \limsup_{\rho \to 0+} \frac{\mathbf{m}^{PC}_{\mathcal{S}}(\bar{u}_{\bar{x},\nu_x},Q^{\nu_x}_{\rho}(\bar{x}))}{\rho^{d-1}} - C\eps\alpha^{-1}\Vert h\Vert_\infty.$$
This along with \eqref{eq: almost the same2} and  \eqref{eq: general minimization3-better} shows \eqref{eq: main point}(i). In view of the characterization \eqref{eq: jump energy density-relax}, the proof of \eqref{eq: main point}(ii) can be obtained along similar lines. (Indeed, it is even easier since instead of \eqref{h:cont} we only need $\Vert h(x,\cdot) - h(\bar{x},\cdot) \Vert_\infty \le \eps$.)

\noindent\emph{Step 2: $g_\infty\big(x,[u](x),\nu_x \big) \ge \bar{h} (x,\nu_x)$.}  We follow Step 2 of the previous proof by employing a slightly different continuity argument in \eqref{eq: this cont is needed}. More precisely, by the uniform continuity of $h$ in $D$ and property \eqref{eq: h-inequaly}, for given $\eps >0$, we can choose $\bar{x} \in D$ such that \eqref{eq: main point}(ii) holds and  \eqref{eq: this cont is needed} is replaced by
$$
\liminf_{i \to \infty}\frac{1}{\rho_i^{d-1}}{\mathcal{E}(u,Q^{ {\nu}_x}_{\rho_i}(x))} \ge \liminf_{i \to \infty} \int_{J_{v_i}\cap Q^{{\nu}_x}_{1}}h(\bar{x},\nu_{v_i}(y))\,{\rm d}\mathcal{H}^{d-1}(y)-\eps.
$$
Since $\bar{x} \in D$, we can proceed as in the previous proof to find, $g_\infty\big(x,[u](x),\nu_x\big) \ge \bar{h} (\bar{x},\nu_x) -\eps$. This along with \eqref{eq: main point}(ii) and the arbitrariness of $\eps$ shows $g_\infty\big(x,[u](x),\nu_x\big) \ge \bar{h} (x,\nu_x)$. 
\end{proof}

\begin{remark}\label{rem: lastone}
The proof shows that $\bar{h}$ is uniformly continuous on $D$ and that for $\mathcal{H}^{d-1}$-a.e.\ $x \in \partial D \RRR \cap \Omega\EEE$ it holds that  $\bar{h}(x,\nu_D(x)) = \lim_{n\to \infty} \bar{h}(x_n,\nu_n)$ for sequences $(x_n)_n \subset D$ and $(\nu_n)_n \subset \mathbb{S}^{d-1}$ with $x_n \to x$ and $\nu_n\to \nu_D$. 
\end{remark}

\subsection{Proof of Corollaries \ref{cor: hom} and \ref{cor: relax}.}\label{sec:cor}
We deduce the announced corollaries of Theorem   \ref{thm: main thm d=2}.

\begin{proof}[Proof of Corollary \ref{cor: hom}]
 We only need to show that the limits in \eqref{eq: hom-lim} coincide  with $f$ and $g$ in \eqref{eq: bulk assumption}--\eqref{eq: surface assumption}. Then the result follows from Theorem \ref{thm: main thm d=2}.  The argument used to prove this property is standard and for this reason we omit it.  For instance, we can follow closely the proof of  \cite[Theorem 3.11]{Sto lavoro GSBV}. 
\end{proof}

\begin{proof}[Proof of Corollary \ref{cor: relax}]
Thanks to   Theorem~\ref{thm: main thm d=2}(iii),  applied to the constant sequence of functionals  $\mathcal{E}_n=\mathcal{E}$  given by (\ref{eq: basic energy}) corresponding to $f$ and $g$, we get that 
\begin{align*}
\bar{\mathcal{E}}(u,A)=\int_A  \tilde{f}  \big(x,e(u)(x)\big)\,{\rm d}x + \int_{J_u\cap A}\bar{g}\big(x,\nu_u(x)\big)\, {\rm d}\mathcal{H}^{d-1}(x)
\end{align*}
for all $u\in GSBD^p(\Omega)$ and $A\in\mathcal{A}(\Omega)$, where  $\tilde{f}$  is defined as in (\ref{eq: bulk assumption}), and $\bar{g}$  is  the $BD$-elliptic envelope of $g$ defined in \eqref{eq: envelope}.  By Remark \ref{rem: new rem} we get that $\tilde{f}$ coincides with the density $f_0$ in Proposition \ref{prop: bulk int repr}. Since in the case of a constant sequence  the $\Gamma$-limit $\mathcal{F}$ is simply the relaxation of the integral functional with density $f$, it turns out that $f_0 = \tilde{f}$ is the quasiconvex envelope (with respect to the second variable) of $f$, see \cite[Theorem 9.8]{Dacorogna}. 
\end{proof}



\section{Minimization problems for given boundary data}\label{sec: mini-proof}

 This section is devoted to the proofs of the results announced in Subsection \ref{sec: mini}.  Before coming to the proof of Proposition \ref{lemma: gamma bdy}, we state an auxiliary lower semicontinuity result.

\begin{lemma}[Lower semicontinuity]
In the setting of Proposition \ref{lemma: gamma bdy}, consider a sequence $(v_n)_n\subset GSBD^p(\Omega')$ with  $\sup_{n \in \N}(\Vert e(v_n) \Vert_{L^p(\Omega')} + \mathcal{H}^{d-1} (J_{v_n}))<+\infty$ such that $v_n \to v$ in measure  on $\Omega'$ for some $v \in GSBD^p(\Omega')$. Then, for all $A \in \mathcal{A}(\Omega')$ it holds that  
\begin{align}\label{eq: separation}
{\rm (i)} & \ \ \int_A f'(x, e(v)(x)) \, {\rm d}x \le  \liminf_{n \to \infty} \int_A f'_n(x, e(v_n)(x)) \, {\rm d}x, \notag \\
{\rm (ii)}& \ \  \int_{J_v \cap A} g'(x,  \nu_v) \, {\rm d}\mathcal{H}^{d-1}  \le  \liminf_{n \to \infty} \int_{J_{v_n} \cap A} g'_n(x,  \nu_{v_n}) \, {\rm d}\mathcal{H}^{d-1}.
\end{align} 

\end{lemma}

\begin{proof}
The statement follows by  repeating the argument in  \cite[Proposition 4.3]{GP} which we detail here for the convenience of the reader.    First, we notice that by Theorem \ref{thm: main thm d=2}(i) the bulk density $f'_\infty$ of $\mathcal{E}'$ agrees with the function $f'$ given by \eqref{eq: bulk assumption}. Moreover, the surface density $g'_\infty$ agrees with $g'$ given in \eqref{eq: surface assumption} by assumption in Proposition \ref{lemma: gamma bdy}.  As $f'$ and $g'$ are characterized by  \eqref{eq: bulk assumption} and \eqref{eq: surface assumption}, respectively, we find that the energies $\mathcal{E}^{k,l}_n$ with densities $kf_n'$ and $lg_n'$ for $k,l \in \N$ $\Gamma$-converge to $\mathcal{E}^{k,l}$ with densities $kf'$ and $lg'$. Therefore, we obtain
\begin{align*}
\int_A  kf'(x, e(v)(x)) \, {\rm d}x &+ \int_{J_v \cap A} lg'(x,  \nu_v) \, {\rm d}\mathcal{H}^{d-1} \\ & \le \liminf_{n \to \infty}  \Big( \int_A kf'_n(x, e(v_n)(x)) \, {\rm d}x + \int_{J_{v_n} \cap A} lg'_n(x,  \nu_{v_n}) \, {\rm d}\mathcal{H}^{d-1} \Big).
\end{align*}    
In view of \eqref{eq: general bound}, \eqref{eq: general bound2-new},  and   $\sup_{n\in\N}(\Vert e(v_n) \Vert_{L^p(\Omega')} + \mathcal{H}^{d-1} (J_{v_n}))<+\infty$,   by dividing the estimate by $k$ and sending $k \to + \infty$, we obtain \eqref{eq: separation}(i). In a similar fashion, \eqref{eq: separation}(ii) follows by dividing first by $l$ and by letting   $l \to + \infty$. 
\end{proof}

\begin{proof}[Proof of Proposition \ref{lemma: gamma bdy}]
We follow the lines of similar results in the $GSBV^p$ setting, see \cite[Lemma 7.1]{GP} and \cite[Lemma~4.3]{Manuel}. Our focus lies on the adaptations necessary to our $GSBD^p$ setting, including more delicate constructions for extensions and fundamental estimates. (In particular, besides Proposition \ref{lemma: fundamental estimate},  we use the recent extension result  \cite{Matteo} and approximations from \cite{FinalKorn}.) To keep the exposition self-contained, however, we provide all details of the proof.   

We start by noticing that the $\Gamma$-liminf is immediate from the $\Gamma$-convergence of $\mathcal{E}'_n$ to $\mathcal{E}'$ and the fact that the boundary condition on $\Omega' \setminus \overline{\Omega}$ is preserved under the convergence in measure. We now address the  $\Gamma$-limsup.   As  $\mathcal{E}'_n$ $\Gamma$-converges to $\mathcal{E}'$, there exists a recovery sequence $(u_n)_n$ for $u$, i.e., $u_n \to u$ in measure on $\Omega'$   and $\lim_{n\to \infty} \mathcal{E}'_n(u_n) = \mathcal{E}'(u)$. By Remark \ref{rem: recov} we get that $\mathcal{E}'_n(u_n,A) \to \mathcal{E}'(u,A)$ for each  $A \in \mathcal{A}(\Omega')$  with $\mathcal{E}'(u,\partial A) = 0$. This along with  \eqref{eq: separation} shows  
\begin{align}\label{eq: separation-recovery}
{\rm (i)}& \ \ \int_A f'(x, e(u)(x)) \, {\rm d}x  =  \lim_{n \to \infty} \int_A f'_n(x, e(u_n)(x)) \, {\rm d}x, \notag \\
{\rm (ii)} & \ \  \int_{J_u \cap A} g'(x, \nu_u) \, {\rm d}\mathcal{H}^{d-1}  =  \lim_{n \to \infty} \int_{J_{u_n} \cap A} g'_n(x, \nu_{u_n}) \, {\rm d}\mathcal{H}^{d-1},
\end{align}
whenever $\mathcal{E}'(u,\partial A) = 0$.

\noindent \emph{Step 1: Definition of the recovery sequence.} We need to modify the sequence $(u_n)_n$ to ensure that the boundary conditions are satisfied on $\Omega' \setminus \overline{\Omega}$. This is subject of this step. In Step 2 we will estimate the energy of the modified sequence to see that it is indeed a recovery sequence. At the end of the proof in Step 3, we will check the following auxiliary properties: 
\begin{align}\label{eq: good approx}
{\rm (i)} & \ \ u_n - u^0_n  \to 0 \text{ in measure on $\Omega' \setminus \overline{\Omega}$},  \notag \\
{\rm (ii)} & \ \ e(u_n) - e(u^0_n)  \to 0 \text{ strongly in } L^p(\Omega' \setminus \overline{\Omega};\R^{d \times d}_{\rm sym}), \notag\\
{\rm (iii)}& \ \  \mathcal{H}^{d-1}\big(J_{u_n} \cap (\Omega' \setminus {\Omega}) \big) \to 0. 
\end{align}
For the moment, we suppose that \eqref{eq: good approx} holds.  We can find a  Lipschitz  neighborhood $U \supset\supset \Omega' \setminus \overline{\Omega}$ and an extension $(y_n)_n\subset GSBD^p(U)$ satisfying $y_n =u_n - u^0_n $ on $\Omega' \setminus \overline{\Omega}$ such that for $n \to \infty$ 
\begin{align}\label{eq: good approx2}
 {\rm (i)} \ \   \Vert e(y_n) \Vert_{L^p(U)} + \mathcal{H}^{d-1}(J_{y_n} \cap U) \to 0, \quad \quad \quad  {\rm (ii)} \ \ y_n \to 0  \text{ in measure on $U$}. 
\end{align}
Indeed, by the extension result \cite[Theorem 1.1]{Matteo} we can choose $(y_n)_n\subset GSBD^p(U)$ with  $y_n|_{U \cap \Omega} \in W^{1,p}(U\cap \Omega;\R^d)$  such that  $y_n =u_n - u^0_n $ on $\Omega' \setminus \overline{\Omega}$ and 
$$\Vert e(y_n) \Vert_{L^p(U)} + \mathcal{H}^{d-1}(J_{y_n} \cap U) \le C\Vert e(u_n- u^0_n) \Vert_{L^p(\Omega' \setminus \overline{\Omega})} + C\mathcal{H}^{d-1}(J_{u_n} \cap (\Omega' \setminus \overline{\Omega}) ),$$
 where $C>0$ depends only on $\Omega,\Omega'$ and $p$.   Then, \eqref{eq: good approx2}(i) follows from \eqref{eq: good approx}(ii),(iii).  We now show \eqref{eq: good approx2}(ii). To this end, we use \eqref{eq: good approx2}(i) and apply Theorem \ref{th: kornSBDsmall} on $(y_n)_n$ to  find sets $(\omega_n)_n \subset U$ with $\mathcal{L}^d(\omega_n) \to 0$ and rigid motions $(a_n)_n$ such that $\Vert y_n - a_n \Vert_{L^p(U\setminus \omega_n)} \to 0$. By \eqref{eq: good approx}(i) and $y_n =u_n -  u^0_n$ on $\Omega' \setminus \overline{\Omega}$  we get $a_n \to 0$ in measure on $\Omega' \setminus \overline{\Omega}$. As $a_n$ is affine, this also yields $a_n \to 0$ in measure on $U$. This along with $\Vert y_n - a_n \Vert_{L^p(U\setminus \omega_n)} \to 0$ and $\mathcal{L}^d(\omega_n) \to 0$ shows \eqref{eq: good approx2}(ii).

%
%

\begin{figure}
\centering
\def\svgwidth{14cm}
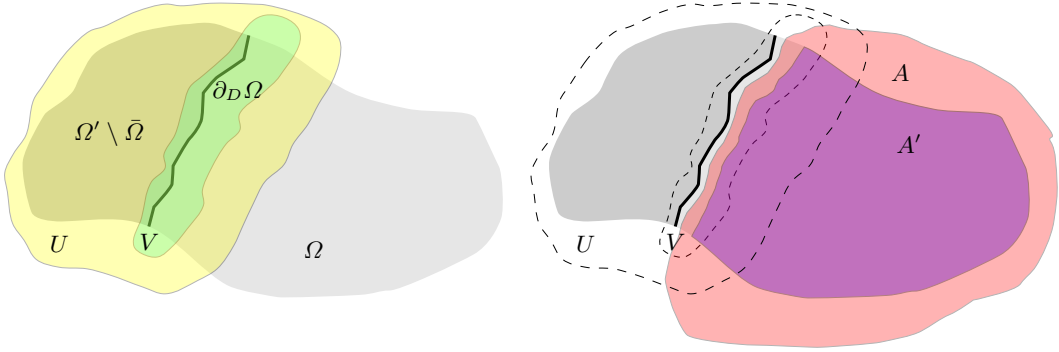
\caption{Illustration of the  various sets used in Step 1 of the proof.}
\label{fig:disegno}
\end{figure}

Let $\eps > 0$ and choose $V$ open with $V \supset \overline{\partial_D\Omega}$, $V \subset U$, $\mathcal{E}'(u,\partial (V \cap \Omega')) = 0$, $\mathcal{L}^d(V)\le \eps$, and $\int_{V\cap \Omega'} f'(x,e(u)(x))\, {\rm d}x< \eps$. (Here,  $\partial_D \Omega = \Omega' \cap \partial \Omega$.) Then by \eqref{eq: separation-recovery}(i) we also obtain  
\begin{align}\label{eq: limiting eps}
 \lim_{n \to \infty}  \int_{V \cap \Omega'} f'_n(x,e(u_n)(x)) \, {\rm d}x < \eps.
\end{align}
Our goal is to apply  the fundamental estimate. To this end, choose  $A \subset \R^d$ such that $A \cap \Omega' \subset \Omega$ and $A \supset \supset \Omega \setminus {V}$, and choose  $A' \subset \subset A$ such that $\Omega \setminus {V} \subset A' \subset \Omega$ (see for instance Picture \ref{fig:disegno}).  Moreover, we let $B = U \cap \Omega'$, and we note that $A' \cup B = \Omega'$ since $V \subset U$. Define the functional $\mathcal{I}$ by  $\mathcal{I}(w,A) = \Vert e(w) \Vert^p_{L^p(A)} + \mathcal{H}^{d-1}(J_w \cap A)$  for all $w \in GSBD^p(\Omega')$ and $A \in \mathcal{A}(\Omega')$. We apply Proposition \ref{lemma: fundamental estimate} and Remark \ref{rem: topo}(iii)  for $\eta>0$ and $\mathcal{I}$, for the functions $u \equiv 0$ and $v = y_n$, to  find a sequence $(\varphi^\eta_n)_n \subset GSBD^p(\Omega')$ with $\varphi^\eta_n =  y_n$ on $\Omega' \setminus \overline{\Omega} \subset B \setminus A$ and $\varphi^\eta_n = 0$ on $\Omega \setminus V \subset A'$  (see \eqref{eq: assertionfund}(iii)) such that by \eqref{eq: assertionfund}(i) we have 
$${\mathcal{I}(\varphi^\eta_n,\Omega') \le (1+\eta)\big(\mathcal{I}( 0,  A) + \mathcal{I}(y_n,B) \big) + \Lambda(u,y_n) + \eta.}$$
As   $ \Lambda( 0,  y_n) \to 0$ for $n \to \infty$ by \eqref{eq: good approx2}(ii) and \eqref{eq: Lambda0}, $\mathcal{I}(0,A) = 0$, and $\mathcal{I}(y_n,B) \to 0$ by \eqref{eq: good approx2}(i), we find $\limsup_{n \to \infty}\mathcal{I}(\varphi^\eta_n,\Omega') \le \eta$. By sending $\eta \to 0$ we can  choose a suitable diagonal sequence $(\varphi_n)_n$, still satisfying $\varphi_n =  y_n$ on $\Omega' \setminus \overline{\Omega}$ and $\varphi_n = 0$ on $\Omega \setminus V$ such that  $\mathcal{I}(\varphi_n,\Omega') \to 0$. This in turn implies   
\begin{align}\label{eq: last g2-nochmal}
\Vert e(\varphi_n) \Vert_{L^p(\Omega')} + \mathcal{H}^{d-1}(J_{\varphi_n} \cap \Omega') \to 0.
\end{align}
In a similar fashion, by \eqref{eq: assertionfund}(ii) and \eqref{eq: good approx2}(ii) we find 
\begin{align}\label{eq: last g2}
 \varphi_n \to 0 \quad \text{in measure on $\Omega'$}. 
\end{align} 
Now, we let $\tilde{u}_n  := u_n - \varphi_n$. Then $\tilde{u}_n = u_n - y_n$ on  $\Omega' \setminus \overline{\Omega}$,  and therefore $\tilde{u}_n =  u^0_n$ on $\Omega' \setminus \overline{\Omega}$ as $y_n = u_n -  u^0_n$ on $\Omega' \setminus \overline{\Omega}$.  Moreover, $\tilde{u}_n = u_n$ on $\Omega \setminus V$ as $\varphi_n = 0 $ on $\Omega \setminus V$. We also observe that $\tilde{u}_n \to u$ in measure on $\Omega'$ by \eqref{eq: last g2} and the fact that ${u}_n \to u$ in  measure on $\Omega'$. 

\noindent \emph{Step 2: Estimate on the energy.} To conclude that $(\tilde{u}_n)_n$ is a recovery sequence, it   remains to  estimate the energy $\tilde{\mathcal{E}}_n'(\tilde{u}_n)$. As $\mathcal{H}^{d-1}(J_{\varphi_n} \cap \Omega') \to 0$  by \eqref{eq: last g2-nochmal},  we find by \eqref{eq: general bound2-new} that  
\begin{align}\label{eq: last g}
\limsup_{n \to \infty} \int_{J_{\tilde{u}_n}} g'_n(x, \nu_{\tilde{u}_n}) \,{\rm d} \mathcal{H}^{d-1} \le \limsup_{n \to \infty} \int_{J_{{u}_n}} g'_n(x,\nu_{{u}_n}) \, {\rm d} \mathcal{H}^{d-1}.  
\end{align}
By  \eqref{eq: f ext}  and $\tilde{u}_n = u_n$ on $\Omega \setminus V$ we further get   
\begin{align*}
\int\limits_{\Omega'} |f'_n(x,e(u_n)) - f'_n(x,e(\tilde{u}_n)) |\, {\rm d}x \le  \int\limits_{V \cap \Omega} \big(f_n(x,e(u_n)) + f_n(x,e(\tilde{u}_n)) \big)  + \alpha \hspace{-0.1cm} \int\limits_{\Omega' \setminus \Omega} ||e(u_n)|^p -  |e(u_n^0)|^p|.
\end{align*}
The rightmost term converges to zero for $n \to \infty$ by \eqref{eq: good approx}(ii). By  \eqref{eq: general bound}, \eqref{eq: limiting eps}--\eqref{eq: last g2-nochmal},   $\mathcal{L}^d(V)\le \eps$,  and the definition $\tilde{u}_n  := u_n - \varphi_n$ we derive 
 \begin{align*}
\limsup_{n \to \infty}  \int_{V \cap \Omega} \big(f_n(x,e(u_n)) + f_n(x,e(\tilde{u}_n))\big)  \, {\rm d}x  &\le \beta \mathcal{L}^d( V) +  2^{p-1}\beta \limsup_{n \to \infty}  \int_{V\cap \Omega} |e(\varphi_n)|^p \, {\rm d}x \\ &  \ \ \ + (1+2^{p-1}\beta \alpha^{-1}) \limsup_{n \to \infty}  \int_{V\cap \Omega} f_n(x,e(u_n)) \, {\rm d}x \\
& \le \beta\eps + (1+2^{p-1}\beta \alpha^{-1})\eps.
  \end{align*}
Thus, by   \eqref{eq: last g}  and the fact that $ \tilde{\mathcal{E}}_n'(u_n) =  \mathcal{E}_n'(u_n) \to \mathcal{E}'(u) = \tilde{\mathcal{E}}'(u)$, we then conclude
$$ \limsup_{n \to \infty}  \tilde{\mathcal{E}}'_n(\tilde{u}_n) \le \limsup_{n \to \infty} \mathcal{E}'_n( {u}_n)  +  \beta\eps + (1+2^{p-1}\beta \alpha^{-1})\eps \le \tilde{\mathcal{E}}'(u) +   \beta\eps + (1+2^{p-1}\beta \alpha^{-1})\eps. $$  
 Since $\eps$ was arbitrary, we obtain the $\Gamma$-limsup inequality  by using a diagonal argument. 

\noindent
\emph{Step 3:  Proof of \eqref{eq: good approx}.}  To conclude, it remains to show \eqref{eq: good approx}. First, as $(u_n)_n$ is a recovery  sequence, we find $u_n \to  u^0$ in measure on $\Omega' \setminus \overline{\Omega}$.  This along with $ u^0_n \to  u^0$ in $L^p(\Omega';\R^d)$ shows (i). To see (ii), we consider $A \in \mathcal{A}(\Omega')$, with $\overline{A} \subset \Omega' \setminus \overline{\Omega}$ and $\mathcal{E}'(u,\partial A) = 0$.   Then  \eqref{eq: f ext} and \eqref{eq: separation-recovery}(i) yield 
\begin{align}\label{eq: first convergence}
e(u_n) \to e( u^0) \ \ \text{ in } \ L^p(A;\R^{d \times d}_{\rm sym}).
\end{align}
 For $\eps >0$, we consider $V$ open with $V \supset \overline{\partial_D\Omega}$ such that  $\mathcal{E}'(u,\partial (V\cap \Omega')) = 0$, $\mathcal{L}^d(V)< \eps$,  and 
\begin{align}\label{eq: LL7}
\int_{V \cap \Omega'} f'(x,e(u)(x)) \, {\rm d}x < \eps, \ \ \ \ \int_{V \cap \Omega'} f'(x,e( u^0_n)(x))\,{\rm d}x <\eps \ \ \ \ \text{for all } n \in \N. 
\end{align}
(The second estimate is achieved  by  \eqref{eq: general bound} and the fact that $e( u^0_n) \to e( u^0)$ strongly in $L^p(\Omega';\R^{d\times d}_{\rm sym})$.)  
For $n$ large enough, we also get  $\int_{V \cap \Omega'} f'_n(x,e(u_n)(x)) \, {\rm d}x < \eps$ by \eqref{eq: separation-recovery}(i). Then, by  \eqref{eq: general bound} we obtain 
\begin{align*}
\int_{\Omega' \setminus \overline{\Omega}} |e(u_n) - e( u^0_n)|^p\,{\rm d}x    & = \int_{\Omega' \setminus (\Omega \cup V)}|e(u_n) - e( u^0_n)|^p\,{\rm d}x    + \int_{V \cap (\Omega' \setminus \overline{\Omega})}|e(u_n) - e( u^0_n)|^p \,{\rm d}x   \\
& \le  \int\limits_{\Omega' \setminus (\Omega \cup V)}|e(u_n) - e( u^0_n)|^p\, {\rm d}x   + \frac{2^{p-1}}{\alpha}    \int\limits_{V \cap \Omega'}\hspace{-0.1cm} (f'_n(x,e(u_n)) +  f'(x,e( u^0_n))).
\end{align*}
By \eqref{eq: first convergence}, \eqref{eq: LL7},  and the fact that $\Vert e( u^0_n) - e( u^0) \Vert_{L^p(\Omega')} \to 0$ we conclude  
$$\limsup_{n \to \infty}\int_{\Omega' \setminus \overline{\Omega}} |e(u_n) - e( u^0_n)|^p \, {\rm d}x \le 2^p \alpha^{-1} \eps. $$ Since $\eps$ was arbitrary, we obtain (ii). We finally prove (iii). Up to a subsequence we have
$$\mu_n := \mathcal{H}^{d-1}|_{J_{u_n} \cap (\Omega' \setminus {\Omega})} \stackrel{*}{\rightharpoonup} \mu  \ \ \ \text{weakly$^*$ in the sense of measures.}$$
By \eqref{eq: general bound2-new} and \eqref{eq: separation-recovery}(ii) we get $\mathcal{H}^{d-1}(J_{u_n} \cap U) \to 0$ for all $U \in \mathcal{A}(\Omega')$ with $\overline{U} \subset \Omega' \setminus \overline{\Omega}$ and $\mathcal{E}'(u,\partial U) = 0$. Consequently, to conclude the proof of (iii), it suffices to show $\mu(\partial_D\Omega) = 0$. We argue by contradiction. Let us assume that $\mu(\partial_D\Omega) > 0$. Then there exists a cube $Q_{2\rho}(x)$ with $x \in \partial_D\Omega$ such that $Q_{2\rho}(x) \subset \Omega'$,  $\mathcal{E}'(u,\partial Q_{2\rho}(x)) = 0$,  and   $\mu(Q_{2\rho}(x)) > \sigma > 0$. For notational simplicity, we write $Q^\rho = Q_{2\rho}(x)$ and also let $\tilde{Q}^\rho = Q_{8\rho}(x)$. We may also assume  that $\tilde{Q}^\rho  \subset \Omega'$. For $n$ large enough we get 
\begin{align}\label{eq: sigma}
\mathcal{H}^{d-1}( J_{u_n} \cap (Q^\rho \setminus {\Omega}) )  = \mu_n(Q^\rho ) > \sigma >0. 
\end{align}
Our goal is now to modify the sequence $(u_n)_n$ by a reflection method and  to  move the jump set inside $\Omega$. This will lead to a contradiction as we assumed that $(u_n)_n$ is a recovery sequence, but inside $\Omega$ the surface energy is much less than in $\Omega' \setminus {\Omega}$, cf.\ \eqref{eq: g ext}. The reflection method is a bit more delicate compared to \cite{GP, Manuel} since we deal with $GSBD^p$ instead of $GSBV^p$.    Possibly after passing to a smaller $\rho$ (not relabeled), we can assume that in a suitable coordinate system
$$\Omega \cap \tilde{Q}^\rho = \lbrace (x',y): \ x' \in (-4\rho,4\rho)^{d-1}, \ y \in (-4\rho,\tau(x')) \rbrace$$
for a Lipschitz function $\tau$ with $\Vert \tau \Vert_\infty \le \rho$.  We choose $\eta \in (2\rho,3\rho)$ such that 
$$V_\rho := \lbrace (x',y): \ x' \in (-\rho,\rho)^{d-1}, \ y \in (\tau(x')-\eta,\tau(x')+\eta) \rbrace$$ satisfies $\mathcal{E}'(u,\partial V_{\rho}) = 0 $. Note that $Q^\rho \subset V_\rho \subset \tilde{Q}^\rho \subset \Omega'$ since $\eta \in (2\rho,3\rho)$.  Let $\hat{u}$ be the function defined on $V_{\rho}$ by reflecting $u$ at $\tau(x')$, $x' \in (-\rho,\rho)^{d-1}$, i.e.,
\begin{align*}
\hat{u}(x',y) = \begin{cases} 
u(x',y)  &  \text{if } y > \tau(x'),   \\  u(x', 2\tau(x') - y) & \text{if }  y < \tau(x'). 
\end{cases}
\end{align*}
Clearly $\hat{u}\in W^{1,p}(V_{\rho};\R^d)$ as $u \in W^{1,p}(\Omega' \setminus \overline{\Omega};\R^d)$. In a similar fashion, we need to reflect the sequence $(u_n)_n$. As a preliminary step, we apply Theorem \ref{th: crismale-density2} for $ \Omega  =\tilde{Q}^\rho$ to obtain a sequence $(v_n)_n \subset GSBV^p(\tilde{Q}^\rho;\R^d) \cap L^p(\tilde{Q}^\rho;\R^d)$ such that 
\begin{align}\label{eq:new-apprxo}
v_n - u_n \to 0 \ \ \text{in measure on $\tilde{Q}^\rho$ for $n \to \infty$}, \quad   \mathcal{H}^{d-1}\big( (J_{u_n} \triangle J_{v_n})    \cap \tilde{Q}^\rho \big )\le  \frac{1}{n}  \ \ \text{for all $n \in \N$}.
\end{align}
We define $\hat{u}_n \in GSBV^p(V_{\rho};\R^d)$ by
\begin{align*}
\hat{u}_n(x',y) = \begin{cases} 
v_n(x',y)  & \text{if } y > \tau(x') -\lambda_n,  \\  v_n(x', 2(\tau(x')-\lambda_n) - y) &  \text{if } y < \tau(x')-\lambda_n, 
\end{cases}
\end{align*}
where $0 < \lambda_n \le 1/n$ is chosen such that
\begin{align}\label{eq: etk}
\mathcal{H}^{d-1} \Big(\Big\{ (x',y) \in J_{v_n}: \, x' \in (-\rho,\rho)^{d-1}, \ y \in (\tau(x')-\lambda_n,\tau(x')) \Big\}\Big) \le \frac{1}{n}.
\end{align} 
Observe that  the functions are well defined  since  $\Vert \tau \Vert_\infty \le \rho$  and $\eta < 3\rho$.  We now introduce the sequence 
$$w_n:= v_n + \hat{u} - \hat{u}_n  \in   GSBV^p(V_\rho;\R^d).$$
By \eqref{eq:new-apprxo}, $\lambda_n \to 0$, and the fact that $u_n \to u$ in measure on $\Omega'$, we get that  $w_n \to u$ in measure on $V_{\rho}$. By letting $\Gamma_n := J_{w_n} \cap J_{v_n}$, we further find 
\begin{align}\label{eq: 3 prop}
{\rm (i)}& \ \  \mathcal{H}^{d-1}( J_{w_n} \cap (V_\rho \setminus {\Omega}) ) = 0,\notag  \\ 
{\rm (ii)} &  \ \   \mathcal{H}^{d-1}(J_{w_n} \setminus \Gamma_n)  \le\mathcal{H}^{d-1} \big(\big\{ (x',y) \in V_\rho \cap J_{v_n}: \ y > \tau(x')-\lambda_n \rbrace\big).  
\end{align}
Here, the essential point is that the jump of $w_n$ lies completely inside ${\Omega}$.  By \eqref{eq: general bound2-new} and \eqref{eq: 3 prop}(i) we now get
$$G(w_n):=  \int_{J_{w_n} \cap V_\rho} g'_n(x, \nu_{w_n}) \, {\rm d}\mathcal{H}^{d-1} \le \int_{J_{v_n} \cap \Gamma_n} g'_n(x, \nu_{v_n}) \, {\rm d} \mathcal{H}^{d-1} + \beta \mathcal{H}^{d-1}(J_{w_n} \setminus \Gamma_n).  $$
Then by \eqref{eq: general bound2-new}, \eqref{eq:new-apprxo}, \eqref{eq: etk},  and  \eqref{eq: 3 prop}(ii) we derive
\begin{align*}
G(w_n)& \le \int_{J_{v_n} \cap \Gamma_n} g'_n(x, \nu_{v_n}) \, {\rm d}\mathcal{H}^{d-1} + \beta \mathcal{H}^{d-1}( J_{v_n} \cap (V_\rho \setminus {\Omega}) ) + \beta/n\\
& \le \int_{J_{u_n} \cap (V_\rho \cap \Omega)} g'_n(x, \nu_{u_n}) \, {\rm d}\mathcal{H}^{d-1} + \beta\mathcal{H}^{d-1}( J_{u_n} \cap (V_\rho \setminus {\Omega}) ) + 3\beta/n.
\end{align*}
In the second step, we also used $\Gamma_n \subset V_\rho \cap \Omega$ by \eqref{eq: 3 prop}(i).  
Now, by \eqref{eq: g ext}, \eqref{eq: sigma},  and $ Q^\rho  \subset V_\rho$ we get 
\begin{align}\label{eq: to contra}
G(w_n) &\le \int_{J_{u_n} \cap (V_\rho \cap \Omega)} g'_n(x, \nu_{u_n}) \, {\rm d}\mathcal{H}^{d-1} +3\beta/n + (\beta+1)\mathcal{H}^{d-1}( J_{u_n} \cap (V_\rho \setminus {\Omega}) ) - \sigma \notag\\
& \le  \int_{J_{u_n} \cap V_\rho} g'_n(x, \nu_{u_n}) \, {\rm d}\mathcal{H}^{d-1} +3\beta/n - \sigma.
\end{align}
On the other hand, recalling that $w_n \to u$ in measure on $V_\rho$, we get  by \eqref{eq: separation}(ii)  
$$ \int_{J_u \cap V_\rho} g'(x, \nu_u) \, {\rm d}\mathcal{H}^{d-1}  \le  \liminf_{n \to \infty} \int_{J_{w_n} \cap V_\rho} g'_n(x,\nu_{w_n}) \, {\rm d}\mathcal{H}^{d-1} = \liminf_{n \to \infty}  G(w_n).$$
Moreover, as $(u_n)_n$ is a recovery sequence for $u$ and $\mathcal{E}'(u,\partial V_\rho) = 0$, \eqref{eq: separation-recovery}(ii) yields
$$ \int_{J_u \cap V_\rho} g'(x, \nu_u) \, {\rm d}\mathcal{H}^{d-1}  =  \lim_{n\to \infty} \int_{J_{u_n} \cap V_\rho} g'_n(x, \nu_{u_n}) \, {\rm d}\mathcal{H}^{d-1}.
$$
 The previous two equations contradict \eqref{eq: to contra}. This concludes the proof of (iii). 
\end{proof}

\begin{proof}[Proof of Theorem \ref{th: Gamma existence}]
The statement follows in the spirit of  the fundamental theorem of $\Gamma$-convergence, see \cite[Theorem 1.21]{Braides:02}. As we employ nonstandard compactness results, we indicate the details for both cases (i) and (ii).

(i) Given $(u_n)_n \subset GSBD^2(\Omega')$ satisfying \eqref{eq: eps control}, we apply Theorem \ref{th: comp} on the functionals $(\mathcal{E}'_n)_n$ and find a subsequence (not relabeled),  $(y_n)_n\subset GSBD^2(\Omega')$ with $\mathcal{L}^d(\lbrace e(y_n) \neq e(u_n) \rbrace) \to 0$ and 
$$ 
 \liminf_{n \to \infty} \tilde{\mathcal{E}}'_n(y_n) = \liminf_{n \to \infty} {\mathcal{E}}'_n(y_n) \le \liminf_{n \to \infty} {\mathcal{E}}'_n(u_n) = \liminf_{n \to \infty} \tilde{\mathcal{E}}'_n(u_n) = \liminf_{n \to \infty}  \inf_{v \in GSBD^2(\Omega')} \tilde{\mathcal{E}}'_n(v).$$
 Here, the first equality holds as $y_n =  u_n^0$ on $\Omega' \setminus \overline{\Omega}$,  and the first inequality follows from \eqref{eq:good-en}.  By Theorem \ref{th: comp} we also find $u \in GSBD^2(\Omega')$ with $u =  u^0$ on $\Omega' \setminus \overline{\Omega}$ such that $y_n \to u$ in measure on $\Omega'$. Therefore, by the $\Gamma$-liminf inequality in Lemma \ref{lemma: gamma bdy} we derive
 \begin{align}\label{eq: last1}
 \tilde{\mathcal{E}}'(u) \le  \liminf_{n \to \infty} \tilde{\mathcal{E}}'_n(y_n) \le  \liminf_{n \to \infty} \tilde{\mathcal{E}}'_n(u_n) \le \liminf_{n \to \infty} \inf_{v \in GSBD^2(\Omega')} \tilde{\mathcal{E}}'_n(v).\end{align}
 By using Lemma \ref{lemma: gamma bdy} once more, for each $w \in GSBD^2(\Omega')$ with $w = u^0$ on $\Omega' \setminus \overline{\Omega}$ we find a recovery sequence $(w_n)_n$ converging to $w$ in measure satisfying $\lim_{n\to \infty} \tilde{\mathcal{E}}_n'(w_n) = \tilde{\mathcal{E}}'(w)$.  This yields 
  \begin{align}\label{eq: last2}
\limsup_{n \to \infty}  \inf_{v \in GSBD^2(\Omega')} \tilde{\mathcal{E}}'_n(v) \le \lim_{n \to \infty}  \tilde{\mathcal{E}}'_n(w_n) =  \tilde{\mathcal{E}}'(w).
  \end{align}
By combining \eqref{eq: last1}--\eqref{eq: last2} we derive 
  \begin{align}\label{eq: last3}
 \tilde{\mathcal{E}}'(u) \le \liminf_{n \to \infty} \inf_{v \in GSBD^2(\Omega')} \tilde{\mathcal{E}}'_n(v) \le  \limsup_{n \to \infty} \inf_{v \in GSBD^2(\Omega')} \tilde{\mathcal{E}}'(v) \le \tilde{\mathcal{E}}'(w).
  \end{align}
  Since $w$ was arbitrary, we get that $u$ is a minimizer of $\tilde{\mathcal{E}}'$. The statement follows from \eqref{eq: last1} and \eqref{eq: last3} with $w=u$. In particular, the limit in \eqref{eq: eps control2} does not depend on the specific choice of the subsequence and thus \eqref{eq: eps control2} holds for the whole sequence.
  
  (ii)   Given $(u_n)_n \subset GSBD^p(\Omega')$ satisfying \eqref{eq: eps control}, we apply   Theorem \ref{th: comp}   on the functionals $(\mathcal{E}'_n)_n$ and find a subsequence (not relabeled),  $(y_n)_n\subset GSBD^p(\Omega')$ with $\mathcal{L}^d(\lbrace e(y_n) \neq e(u_n) \rbrace) = 0$  (see \eqref{eq: evenbetter!})  and  $u \in  GSBD^p(\Omega')$ with  $u =  u^0$ on $\Omega' \setminus \overline{\Omega}$    such that $y_n \to u$ in measure on $\Omega'$ and $e(y_n)  \rightharpoonup e(u)$   weakly in $L^p(\Omega'; \R^{d\times d}_{\rm sym})$. Now, by   \eqref{eq: separation}(i)  and the fact that $e(y_n) = e(u_n)$ $\mathcal{L}^d$-a.e.\ in $\Omega'$ we find
\begin{align}\label{eq: vitonew1}
\int_{\Omega'} f'(x, e(u)(x)) \, {\rm d}x &\le  \liminf_{n \to \infty} \int_{\Omega'} f'_n(x, e(y_n)(x)) \, {\rm d}x = \liminf_{n \to \infty} \int_{\Omega'} f'_n(x, e(u_n)(x)) \, {\rm d}x.
\end{align}
 Denote the extension of $\hat{g}$ defined in \eqref{eq: g ext} by $\hat{g}'$. The surface density $g'$  of the $\Gamma$-limit $\mathcal{E}'$  satisfies $g' \le \hat{g}'$.  Moreover, as $\hat{g}$ is continuous on $\Omega \times \mathbb{S}^{d-1}$, we get  that $g'(x_0,\cdot)$   is \RRR even and \EEE symmetric jointly convex for all $x_0 \in  \Omega$ by Corollary \ref{cor: surf}.  Note that $g'$ satisfies the properties stated in Remark \ref{rem: lastone} \RRR (with $\Omega'$ in place of $\Omega$ and $D=\Omega$). \EEE  Therefore,  we  can apply Lemma \ref{lemma: vito-lsc} on $g'$. Hence, we obtain 
     \begin{align}\label{eq: vitonew2}
\int_{J_u} g'(x, \nu_u) \, {\rm d}\mathcal{H}^{d-1} &\le  \liminf_{n \to \infty} \int_{J_{u_n}} g'(x, \nu_{u_n}) \, {\rm d}\mathcal{H}^{d-1} \le \liminf_{n \to \infty} \int_{J_{u_n}} \hat{g}'(x, \nu_{u_n}) \, {\rm d}\mathcal{H}^{d-1}.
\end{align}
By combining \eqref{eq: vitonew1}--\eqref{eq: vitonew2} we find $\mathcal{E}'(u) \le \liminf_{n \to \infty} \mathcal{E}'(u_n)$. Clearly, as $u = u^0 $ on $\Omega' \setminus \overline{\Omega}$, this also yields $\tilde{\mathcal{E}}'(u) \le \liminf_{n \to \infty} \tilde{\mathcal{E}}'(u_n)$. Now we can proceed as in (i) below \eqref{eq: last1} to conclude the proof  (with the only difference that we do not get $\lim_{n \to \infty} \tilde{\mathcal{E}}'_n(y_n) = \tilde{\mathcal{E}}'(u)$.) 
\end{proof}

\section*{Acknowledgements} 
This work was supported by the DFG project FR 4083/1-1 and  by the Deutsche Forschungsgemeinschaft (DFG, German Research Foundation) under Germany's Excellence Strategy EXC 2044 -390685587, Mathematics M\"unster: Dynamics--Geometry--Structure. 
The work of Francesco Solombrino is part of the project “Variational methods for stationary and evolution problems with singularities and interfaces” PRIN 2017 financed by the Italian Ministry of Education, University, and Research.

\appendix

\section{Proof of Lemma \ref{lemma: partition} under simplifying assumption}\label{proof1XXX}

\begin{proof}
Here, we present a simplified proof where we assume that each $J_{u_n}$ consists of a bounded number of closed, continuous curves, denoted by $\gamma^n_1,\ldots,\gamma^n_{G_n}$, where $\sup_{n\in \N} G_n < + \infty$.  To simplify notation, we suppose that $\nu = e_2$.  This is not restrictive since we  can first prove the statement for the functions $w_n(x):=R^T u_n(Rx)$  for $x \in Q_1^{e_2} = Q_1$,  where $R$ is a rotation with  $Re_2 = \nu$,  and then rotate back the obtained partition  onto $Q_1^\nu$.  In view of  \eqref{eq: properties-part}(ii),   we define 
\begin{align}\label{eq: bound}
C_0 :=    \max_{n \in \N} G_n + \sup_{n \in \N} \mathcal{H}^1(J_{u_n}) < + \infty.   
\end{align}
For  $k \in \N$ fixed, we define  
\begin{align}\label{eq: Vk}
U_k = Q_1 \cap  \big\{ |   x \cdot e_2   | <  \tfrac{1}{2} k^{-1}  \big\}, \quad \quad \quad \quad V_k =  \overline{Q_1} \cap \big\{ | x \cdot e_2| \le  k^{-1/4} \big\}. 
\end{align}
 For each $k \in \N$,  our strategy will lie in combining different components of the jump set $J_{u_n}$ inside $V_k$.  Then, we will obtain the sets $S_n^+$ and $S_n^-$ in the statement   by a diagonal argument.  We introduce some further notation: we cover  $U_k$    up to a set of negligible $\mathcal{L}^2$-measure by $k$ pairwise disjoint cubes of sidelength $1/k$, denoted by $Q_1^k, \ldots, Q_{k}^k$ with corresponding centers $x_1^k,\ldots, x_{k}^k$. 
 
We will first show that each of these cubes necessarily intersects the jump set $J_{u_n}$ (Step 1). Based on this, we will combine different connected components of the  jump  set  $J_{u_n}$ with small segments of length at most $\sqrt{5}/k$ (Step 2). This will eventually allow us to define the sets $S_n^+$ and $S_n^-$  satisfying \eqref{eq: properties-part2} (Step 3).

\noindent \emph{Step 1: $J_{u_n} \cap Q_j^k \neq \emptyset$ for all $j=1\ldots,k$ and each $n \ge n_k$,  where $n_k \in \N$ depends on $k$.} We suppose by contradiction that the statement  were  wrong, i.e., we find $Q_j^k$  with $J_{u_n} \cap Q_j^k = \emptyset$ for some $n \ge n_k$, where $n_k \in \N$ depending on $k$  is specified below  (see \eqref{eq: bound2-5} and \eqref{eq: bn}).   First, we observe that then $u_n$ would be a Sobolev function when restricted to $Q_j^k$. Thus,  the Korn-Poincar\'e inequality  implies 
\begin{align}\label{eq: korn}
 \Vert u_n - a_n \Vert_{L^1(Q_j^k)} \le C_k \Vert e(u_n) \Vert_{L^p(Q_j^k)} \le C_k \Vert e(u_n) \Vert_{L^p(Q_1)}  
 \end{align} 
     for a constant $C_k$ only depending on $k$, where $a_n$ defined by $a_n(x) = A_n\, x + b_n$ for $x \in \R^2$  is a suitable rigid motion.   As $u_n \to u_{0,\zeta,0,e_2}$ in measure on $Q_1$, see \eqref{eq: properties-part}(iii),  we find a sequence $(\eta_n)_n \subset (0,+\infty)$ with $\eta_n \to 0$ and $n_k \in \N$ sufficiently large such that the sets  (recall \eqref{eq: half-cub not}) 
\begin{align}\label{eq: Bdef}
B_n^- = \lbrace x\in Q^{e_2,-}_1 \colon \, |u_n(x)| < \eta_n\rbrace, \quad \quad \quad \quad B_n^+ = \lbrace x\in Q^{e_2,+}_1 \colon \, |u_n(x) - \zeta | < \eta_n\rbrace 
\end{align}
fulfill for each $n \ge n_k$ that
\begin{align}\label{eq: bound2-5}
\mathcal{L}^2\big( Q^{e_2,\pm}_1 \setminus B_n^\pm \big) \le \tfrac{1}{4}k^{-2}.
\end{align}
As $\mathcal{L}^2(Q^{e_2,\pm}_1 \cap Q_j^k) = \frac{1}{2} k^{-2}$, this implies 
\begin{align}\label{eq: bound2}
\mathcal{L}^2 \big( B^\pm_n \cap Q_j^k  \big) \ge  \tfrac{1}{4}k^{-2}.   
\end{align}
Now, we define $b_n^- = b_n$ and $b_n^+  = b_n - \zeta$.  By Lemma \ref{lemma: rigid motion} for $\psi(t) = t$, $R= \sqrt{2}/2$,  $\delta =1/4$, $G =  A_n/k$,  $b=b^\pm_n  + A_nx^k_j $,  and  $E = k ((B^\pm_n \cap Q_j^k) - x_j^k)$,  and by a change of variables  we get
\begin{align}\label{eq: bound3}
|A_n|/k + |b^\pm_n + A_nx^k_j|   \le   c  \fint_{E} |G \, x +b| \, {\rm d}x  =  \frac{c}{\mathcal{L}^2(B^\pm_n \cap Q_j^k)}  \int_{B^\pm_n \cap Q_j^k} |A_n \, x +b^\pm_n| \, {\rm d}x,
\end{align} 
 where $c>0$ is a universal constant.  (Here, we used that $\mathcal{L}^2(E) \ge 1/4$.)   As $b_n^+ - b_n= -\zeta$ and  $b_n^- - b_n = 0$, we derive  by \eqref{eq: korn} and \eqref{eq: Bdef} that 
  \begin{align*}
 \Vert A_n\, x + b^\pm_n  \Vert_{L^1(B^\pm_n \cap Q_j^k)} & =  \Vert a_n + b^\pm_n  -b_n \Vert_{L^1(B^\pm_n \cap Q_j^k)} \\ 
 & \le  \Vert a_n - u_n \Vert_{L^1(Q_j^k)}  + \Vert u_n + b^\pm_n  -b_n \Vert_{L^1(B^\pm_n \cap Q_j^k)}   \\&
\le  C_k \Vert e(u_n) \Vert_{L^p(Q_1)}   + \eta_n\mathcal{L}^2(B^\pm_n \cap Q_j^k). 
  \end{align*}
This along with \eqref{eq: bound2}--\eqref{eq: bound3} yields $|b_n^\pm+ A_nx^k_j| \le 4ck^2 C_k \Vert e(u_n) \Vert_{L^p(Q_1)} + c\eta_n$. Therefore, by \eqref{eq: properties-part}(i) and $\eta_n \to 0$ we find 
\begin{align}\label{eq: bn}
|b_n^\pm+ A_nx^k_j| \le |\zeta|/3 \quad \quad \text{for all } n \ge n_k,
\end{align}
where $n_k \in \N$ depends on $k$. Now,  however, \eqref{eq: bn} contradicts the fact that $|b_n^+ - b_n^-| = |\zeta|$. This shows that $Q_j^k$ necessarily intersects $J_{u_n}$, and the first step of the proof is concluded.

\noindent \emph{Step 2: Combining  components of the jump set with segments.} Recall the definition of $U_k$ and   $V_k$ in \eqref{eq: Vk}. Let $k \in \N$ and $n \ge n_k$.  The goal of this step is to construct a closed set $\Gamma_n^k \subset V_k$ such that $\Gamma_n^k$ contains the points   $(-\tfrac{1}{2},0)$ and $(\tfrac{1}{2},0)$,  is path connected,  and satisfies
\begin{align}\label{eq: Gammaest}
\mathcal{H}^1\big(\Gamma_n^k \setminus J_{u_n}\big)  \le ck^{-1/2},
\end{align}   
where $c$ depends only on $C_0$ in \eqref{eq: bound}. 

 Let us  construct $\Gamma_n^k$.  We denote  by   $\tau^n_1,\ldots,\tau^n_{T_n}$ the connected components of $J_{u_n} \cap V_k$, which intersect $U_k$. Note that each of these components is a closed, continuous curve, and that we have
\begin{align}\label{eq: incid}
J_{u_n} \cap U_k = \bigcup\nolimits_{j=1}^{T_n} \tau_j^n  \cap U_k.    
\end{align} 
  We now estimate the number $T_n$ of these connected components. To this end, we decompose the components of the original jump set $J_{u_n}$ into the  index  sets
$$\mathcal{T}^{\rm small}_n = \big\{ j=1\ldots,G_n\colon\, \mathcal{H}^1(\gamma^n_j) \le \tfrac{1}{2} k^{-1/4}\big\} \ \ \ \ \text{ and  } \ \ \ \  \mathcal{T}_n^{\rm large} = \lbrace 1, \ldots, G_n\rbrace \setminus \mathcal{T}_n^{\rm small}.$$
Then, as each component in $(\tau_j^n)_{j=1}^{T_n}$ intersects $U_k$ and each curve in $\mathcal{T}_n^{\rm large}$ is split into different curves of   $(\tau_j^n)_{j=1}^{T_n}$ we obtain by \eqref{eq: bound}
\begin{align}\label{eq: Mn}
T_n \le \# \mathcal{T}_n^{\rm small} +  \sum\nolimits_{j \in \mathcal{T}_n^{\rm large} }   \Big\lfloor \frac{\mathcal{H}^1(\gamma^n_j)}{ {\rm dist}(U_k,\partial V_k \cap Q_1) }    \Big\rfloor \le G_n + 2k^{1/4}  \sum\nolimits_{j=1}^{G_n} \mathcal{H}^1(\gamma_j^n) \le 2C_0 \,  k^{1/4}.
\end{align}
Here, in the second step we used that $ {\rm dist}(U_k,\partial V_k \cap Q_1)  =  k^{-1/4} - \frac{1}{2}k^{-1} \ge \frac{1}{2}k^{-1/4}$.

We now add two additional elements to the family of curves, namely the segments 
\begin{align}\label{eq: extra paths}
\tau_{T_n+1}^n =  [-\tfrac{1}{2},-\tfrac{1}{2} + 1/k]  \times \lbrace 0 \rbrace \quad \quad \text{and} \quad \quad \tau_{T_n+2}^n =  [\tfrac{1}{2}-1/k, \tfrac{1}{2}]  \times \lbrace 0 \rbrace.
\end{align}
We connect  the  different components $(\tau_j^n)_{j=1}^{T_n+2}$ with segments: for each pair $(\tau_{j_1}^n, \tau_{j_2}^n)$, $j_1 \neq j_2$, with ${\rm dist}(\tau_{j_1}^n, \tau_{j_2}^n) \le \sqrt{5}/k$, we choose a closed segment of length at most $\sqrt{5}/k$ contained in $V_k$ which connects  $\tau_{j_1}^n$ with $\tau_{j_2}^n$. Denote the union of the components $(\tau_j^n)_{j=1}^{T_n+2}$ with these segments by $\Gamma_n^k$. 

We show that $\Gamma_n^k$ has the desired properties. First, $\Gamma_n^k \subset V_k$ by definition. By construction,  $\Gamma_n^k$ contains the points  $(-\tfrac{1}{2},0)$ and $(\tfrac{1}{2},0)$,  see \eqref{eq: extra paths}. To see that $\Gamma_n^k$ is path connected, we first note that  $\Gamma_n^k$ intersects each cube $Q_j^k$, $j=1,\ldots,k$, by Step 1 and by \eqref{eq: incid}. Then, each component in $(\tau_j^n)_{j=1}^{T_n+2}$ intersecting a cube $Q_j^k$ is connected to all components intersecting the cube $Q_j^k$ or the adjacent cubes $Q_{j-1}^k$ and $Q_{j+1}^k$ (if existent), for the maximal distance of two points in adjacent cubes is $\sqrt{5}/k$. This shows that $\Gamma_n^k$ is path connected. Finally,  since $ \bigcup_{j=1}^{T_n}  \tau^n_j \subset J_{u_n}$, we get by \eqref{eq: Mn} 
that 
$${\mathcal{H}^1\big(\Gamma_n^k \setminus  (J_{u_n} \cup \tau_{T_n+1}^n \cup  \tau_{T_n+2}^n)  \big) \le  \sqrt{5}k^{-1} \frac{ (T_n+2) (T_n+1)}{2} \le   c  k^{-1} (C_0 \,  k^{1/4})^2 \le ck^{-1/2},}$$
 where $c$ depends on $C_0$. As $\mathcal{H}^1(\tau_{T_n+j}^n) \le k^{-1}$ for $j=1,2$,   this shows \eqref{eq: Gammaest} and concludes Step 2 of the proof.

\noindent \emph{Step 3: Definition of $S_n^+$ and $S_n^-$.} We now define the sets  $S_n^+$ and $S_n^-$ and establish \eqref{eq: properties-part2}. For each $k \in \N$ and each $n \ge n_k$, we denote the (possibly countably many) connected components of the open set $Q_1 \setminus \Gamma_n^k$ by $(P_j^{n,k})_{j=1}^J$, $J \in \N \cup \lbrace + \infty \rbrace$. As $\Gamma^k_n$ is path connected and contained in $V_k$, as well as  $(-\tfrac{1}{2},0),(\tfrac{1}{2},0) \in \Gamma^k_n$,  we observe that each set $P_j^{n,k}$ intersects at most one of the sets $Q^{e_2,+}_1 \setminus V_k$ and $Q^{e_2,-}_1 \setminus V_k$. We now define $\hat{S}_{n,k}^-$ as the union of components $(P^{n,k}_j)_{j=1}^J$ which do not intersect     $Q^{e_2,+}_1 \setminus V_k $. We  also  let $\hat{S}_{n,k}^+ := (Q_1 \setminus \Gamma_n^k) \setminus  \hat{S}_{n,k}^-$,  and note that $\hat{S}_{n,k}^+$ does not intersect  $Q^{e_2,-}_1 \setminus V_k $.  We define  the neighborhoods $N_k = Q_1 \setminus Q_{1-k^{-1/4}}$ and observe that the sets  $\hat{S}_{n,k}^\pm$  will possibly not satisfy \eqref{eq: properties-part2}(ii). Therefore, we introduce the  sets    
\begin{align}\label{eq: Sdef}
S_{n,k}^+  =\big( \hat{S}_{n,k}^+    \cup  \big(N_k \cap  Q^{e_2,+}_1\cap V_k\big) \big) \setminus  \big(N_k \cap  Q^{e_2,-}_1\cap V_k\big), \notag \\  S_{n,k}^-  = \big(\hat{S}_{n,k}^-    \cup  \big(N_k \cap  Q^{e_2,-}_1 \cap V_k\big) \big) \setminus  \big(N_k \cap  Q^{e_2,+}_1 \cap V_k\big).
\end{align}
Clearly, by definition we have $\mathcal{L}^2\big(Q_1 \setminus (S^+_{n,k} \cup S^-_{n,k})\big) = 0$ for all $k \in \N$ and $n \ge n_k$, and 
\begin{align}\label{eq: inclusion}
Q^{e_2,\pm}_1  \setminus V_k  \subset S^\pm_{n,k} \subset Q^{e_2,\pm}_1 \cup V_k.
\end{align}
    We can now check that  \begin{align}\label{eq: properties-part3}
{\rm (i)} & \ \  \mathcal{L}^2\big( S^\pm_{n,k} \triangle Q^{e_2,\pm}_1   \big) \le 2k^{-1/4},\notag\\
{\rm (ii)} &  \ \  N_k \cap Q^{e_2,\pm}_1 \subset S^\pm_{n,k},\notag\\
{\rm (iii)} & \ \  \mathcal{H}^1\big( ( \partial S^+_{n,k} \cap \partial S^-_{n,k} )           \setminus  J_{u_n}     \big) \le ck^{-1/4},
\end{align}
for  $c>0$  depending on $C_0$.   Indeed,  (i) is a consequence of  \eqref{eq: Vk} and  \eqref{eq: inclusion}.  By    \eqref{eq: Sdef} and \eqref{eq: inclusion}   we get (ii). Finally, we show (iii). First,  \eqref{eq: Sdef} and the definition of $N_k$ imply   
$$\mathcal{H}^1\big( ( \partial S^+_{n,k} \cap \partial S^-_{n,k} )           \setminus  J_{u_n}     \big) \le \mathcal{H}^1\big( ( \partial \hat{S}^+_{n,k} \cap \partial \hat{S}^-_{n,k} )           \setminus  J_{u_n}     \big)   + \sum\nolimits_{j=\pm}\mathcal{H}^1(\partial(N_k \cap  Q^{e_2,j}_1\cap V_k)).$$
Then,  as $\partial \hat{S}^+_{n,k} \cap \partial \hat{S}^-_{n,k} \subset \Gamma^k_n$  and $\mathcal{H}^1(\partial(N_k \cap  Q^{e_2,\pm}_1\cap V_k)) \le ck^{-1/4}$,  (iii) follows from \eqref{eq: Gammaest}.

Finally, we obtain the desired sets $S^+_{n}$ and $S^-_n$ satisfying \eqref{eq: properties-part2} from \eqref{eq: properties-part3}   by a suitable diagonal argument. This concludes the proof.  
\end{proof}

\typeout{References}

\end{document}

%% file: disegno1.pdf_tex
\begingroup%
  \makeatletter%
  \providecommand\color[2][]{%
    \errmessage{(Inkscape) Color is used for the text in Inkscape, but the package 'color.sty' is not loaded}%
    \renewcommand\color[2][]{}%
  }%
  \providecommand\transparent[1]{%
    \errmessage{(Inkscape) Transparency is used (non-zero) for the text in Inkscape, but the package 'transparent.sty' is not loaded}%
    \renewcommand\transparent[1]{}%
  }%
  \providecommand\rotatebox[2]{#2}%
  \newcommand*\fsize{\dimexpr\f@size pt\relax}%
  \newcommand*\lineheight[1]{\fontsize{\fsize}{#1\fsize}\selectfont}%
  \ifx\svgwidth\undefined%
    \setlength{\unitlength}{2057.24041474bp}%
    \ifx\svgscale\undefined%
      \relax%
    \else%
      \setlength{\unitlength}{\unitlength * \real{\svgscale}}%
    \fi%
  \else%
    \setlength{\unitlength}{\svgwidth}%
  \fi%
  \global\let\svgwidth\undefined%
  \global\let\svgscale\undefined%
  \makeatother%
  \begin{picture}(1,0.32769834)%
    \lineheight{1}%
    \setlength\tabcolsep{0pt}%
    \put(0,0){\includegraphics[width=\unitlength,page=1]{disegno1.pdf}}%
    \put(0.06590588,0.19683706){\color[rgb]{0,0,0}\makebox(0,0)[lt]{\lineheight{1.25}\smash{\begin{tabular}[t]{l}\small{$\Omega'\setminus \bar{\Omega}$}\end{tabular}}}}%
    \put(0.28391119,0.08484416){\color[rgb]{0,0,0}\makebox(0,0)[lt]{\lineheight{1.25}\smash{\begin{tabular}[t]{l}\small{$\Omega$}\end{tabular}}}}%
    \put(0.19774561,0.23739093){\color[rgb]{0,0,0}\makebox(0,0)[lt]{\lineheight{1.25}\smash{\begin{tabular}[t]{l}\small{$\partial_D\Omega$}\end{tabular}}}}%
    \put(0.04319971,0.09073117){\color[rgb]{0,0,0}\makebox(0,0)[lt]{\lineheight{1.25}\smash{\begin{tabular}[t]{l}\small{$U$}\end{tabular}}}}%
    \put(0.12830196,0.09259764){\color[rgb]{0,0,0}\makebox(0,0)[lt]{\lineheight{1.25}\smash{\begin{tabular}[t]{l}\small{$V$}\end{tabular}}}}%
    \put(0.54310106,0.09073117){\color[rgb]{0,0,0}\makebox(0,0)[lt]{\lineheight{1.25}\smash{\begin{tabular}[t]{l}\small{$U$}\end{tabular}}}}%
    \put(0.62820332,0.09259764){\color[rgb]{0,0,0}\makebox(0,0)[lt]{\lineheight{1.25}\smash{\begin{tabular}[t]{l}\small{$V$}\end{tabular}}}}%
    \put(0.84204525,0.25322347){\color[rgb]{0,0,0}\makebox(0,0)[lt]{\lineheight{1.25}\smash{\begin{tabular}[t]{l}\small{$A$}\end{tabular}}}}%
    \put(0.84725332,0.18447674){\color[rgb]{0,0,0}\makebox(0,0)[lt]{\lineheight{1.25}\smash{\begin{tabular}[t]{l}\small{$A'$}\end{tabular}}}}%
  \end{picture}%
\endgroup%